\documentclass[a4paper,10pt]{article}

\usepackage{geometry}
\usepackage{epsfig}
\usepackage{amsmath}
\usepackage{amssymb}
\usepackage{amsthm}
\usepackage{mathrsfs}
\usepackage{color}
\usepackage{amscd}
\usepackage{euscript}
\usepackage{stmaryrd}
\usepackage{hyperref}
\usepackage{tikz}
\usepackage{tikz-cd}

\usepackage{amsthm}
\newtheorem{theorem}{Theorem}[section]
\newtheorem{definition}[theorem]{Definition}
\newtheorem{lemma}[theorem]{Lemma}
\newtheorem{proposition}[theorem]{Proposition}
\newtheorem{corollary}[theorem]{Corollary}
\newtheorem{remark}[theorem]{Remark}
\newtheorem{example}[theorem]{Example}
\newtheorem{examples}[theorem]{Examples}
\newtheorem{assumption}[theorem]{Assumption}

\usepackage{amsfonts}

\newcommand{\hh}{{\mathbb{H}}}
\newcommand{\cc}{{\mathbb{C}}}
\newcommand{\rr}{{\mathbb{R}}}
\newcommand{\zz}{{\mathbb{Z}}}
\newcommand{\nn}{{\mathbb{N}}}

\newcommand{\s}{{\mathbb{S}}}

\newcommand{\sr}{\mathcal{SR}}
\newcommand{\I}{\mathcal{I}}

\newcommand{\T}{\mathcal{T}}
\newcommand{\B}{\mathcal{B}}

\newcommand{\torus}{\mathbb{T}}
\newcommand{\mon}{\mathrm{Mon}}
\newcommand{\reg}{\operatorname{Reg}}
\newcommand{\stem}{\operatorname{Stem}}
\newcommand{\slice}{\mathcal{S}}

\renewcommand{\k}{{\bf k}}
\newcommand{\F}{\mathscr{F}}

\newcommand{\debar}{\overline{\partial}}

\newcommand{\Span}{\operatorname{Span}}

\title{{\bf A manifold Fueter-Sce phenomenon\\ in one hypercomplex variable
}}

\author{Riccardo Ghiloni\\
\small Dipartimento di Matematica, Universit\`a di Trento\\ 
\small Via Sommarive 14, I-38123 Povo Trento, Italy\\
\small riccardo.ghiloni@unitn.it\\
\and
Caterina Stoppato
\\ 
\small Dipartimento di Matematica e Informatica ``U. Dini'', Universit\`a di Firenze \\
\small Viale Morgagni 67/A, I-50134 Firenze, Italy\\
\small caterina.stoppato@unifi.it}

\date{  }

\begin{document}

\maketitle


\begin{abstract}
Fueter's theorem states, in modern terms, that the Laplacian maps slice-regular quaternionic functions into Fueter-regular functions with axial symmetry. This phenomenon is also present in the Clifford setting, where both slice-monogenic functions and generalized partial-slice monogenic functions are mapped by the Laplacian into monogenic functions with axial symmetry. These results are due, respectively, to Sce and Qian and to Xu and Sabadini.

The present work puts the Fueter-Sce phenomenon into context for the wider class of strongly $T$-regular functions. It shows that the phenomenon appears over general associative $*$-algebras. Moreover, the symmetry considered here is multi-axial in a sense introduced by Eelbode. Additionally, but more surprisingly, the phenomenon studied by Fueter, Sce, Xu and Sabadini turns out to be the last step in a multi-step process. A new phenomenon in one hypercomplex variable is therefore discovered.
\end{abstract}



\section{Introduction}\label{sec:introduction}

This work uncovers a new Fueter-Sce phenomenon in one hypercomplex variable. Several novelties appear. Firstly, Fueter's theorem of~\cite{fueter1}, Sce's theorem of~\cite{sce}, and a more recent result of Xu and Sabadini~\cite{xsfuetersce} are unified into a single statement, valid over general associative $*$-algebras. Secondly, the main results obtained are multi-axial in the sense of Eelbode's work~\cite{eelbode2014}. Lastly, but most importantly, the phenomenon studied in~\cite{fueter1,sce,xsfuetersce} turns out to be the last step of a longer process involving several steps. To better explain all three novelties, some introduction is in order.

Hypercomplex function theory originated from the search for analogs, over higher-dimensional algebras, of the theory of holomorphic functions. Numerous theories in one hypercomplex variable were developed over the last century. Fueter introduced his theory of quaternionic regular functions in~\cite{fueter1,fueter2}, see also~\cite{sudbery}. Generalizing Fueter's ideas to Clifford algebras led to the well-established theory of monogenic functions over Clifford algebras, see~\cite{librosommen,librocnops,librogurlebeck2} and references therein. Following an idea of Cullen~\cite{cullen}, Gentili and Struppa introduced in~\cite{cras,advances} the completely distinct theory of quaternionic slice-regular functions, which was vastly developed over the last twenty years: see~\cite{librospringer2} and references therein. Over Clifford algebras, Colombo, Sabadini and Struppa introduced in~\cite{israel} the theory of slice-monogenic functions (or slice-hyperholomorphic functions, see~\cite{librodaniele2}). These endeavors are not limited to the associative setting: for instance, an octonionic function theory was introduced in~\cite{dentonisce} (see~\cite{librotraduzionesce} for a translation into English), while octonionic slice-regular functions were defined in~\cite{rocky}. The work~\cite{perotti} introduced the fundamental concept of slice function and set the grounds for the study of slice-regularity over general alternative $*$-algebras.

A turning point in function theory in one hypercomplex variable was announced in~\cite{unifiednotion} and developed in~\cite{unifiedtheory}. For a fixed alternative $*$-algebra $A$ and an appropriately chosen $(N+1)$-dimensional real subspace $V$ of $A$, the new concept of \emph{$T$-regular function} provided a whole spectrum of theories for functions $V\to A$, varying with the choice of a $\tau\in\nn$ and of a \emph{list of steps}
\[T=(t_0,t_1,\ldots,t_\tau)\in\nn^{\tau+1},\qquad0\leq t_0<t_1<\ldots<t_\tau=N\,.\]
The traditional approach, based on the Cauchy-Riemann-Fueter operator or on the Dirac operator, and the more recent approach, based on slice-wise holomorphy, sit at the opposite edges of this spectrum: they lead to the theory of $(N)$-regular functions and to the theory of $(0,N)$-regular functions, respectively. Already when $A=\hh=V$, in addition to $(3)$-regularity (Fueter's theory) and $(0,3)$-regularity (Gentili and Struppa's theory) a new function theory emerged: the theory of $(1,3)$-regular functions, studied in~\cite{unifiednotion}. The concepts of \emph{$T$-function} and of \emph{strongly $T$-regular function} were also introduced in~\cite{unifiednotion} and studied in~\cite{unifiedtheory}. Independently, the works~\cite{xsannouncement,xsgeneralizedpartialslice} developed the notion of \emph{generalized partial-slice monogenic function}, which coincides with the notion of $T$-regular function if $A$ is chosen to be the Clifford algebra $C\ell(0,N)$, $V$ to be the paravector subspace $\rr^{N+1}$ and $T$ to be of the form $(t_0,N)$ (whence $\tau=1$).

The present work studies the Fueter-Sce phenomenon for $T$-regular functions over general associative $*$-algebras. Over quaternions, Fueter discovered this phenomenon already in his original work~\cite{fueter1}. Namely, he showed how to obtain Fueter-regular functions $\hh\to\hh$ with axial symmetry as Laplacians of quaternionic functions of a special type, which would now be described as slice-regular quaternionic functions (or $(0,3)$-regular functions $\hh\to\hh$) preserving $\cc$. For $N=2n+1$, Sce similarly constructed in~\cite{sce} monogenic functions $\rr^{N+1}\to C\ell(0,N)$ with axial symmetry by applying the $n$th iterate of the Laplacian to specific Clifford functions, which would now be described as slice-monogenic functions (or $(0,N)$-regular functions $\rr^{N+1}\to C\ell(0,N)$) preserving $\Span(e_\emptyset,e_1)$. A translation of~\cite{sce} into English can be found in~\cite{librotraduzionesce}. Qian successfully addressed the same problem for $N$ even in~\cite{qian1997}. These classical results are overviewed in~\cite[\S11.2.3]{librogurlebeck2}, which also cites~\cite{sommen1988}. Some generalizations include~\cite{integralfueter,dongqian,eelbodesouckvanlancker2012,eelbodesouckvanlancker2014,feicerejeiraskaehler,kouqiansommen,penapenaqiansommen,penapenasommen2006,penapenasommen2010,penapenasommen2010bis,sommen2000}. We also wish to mention the works~\cite{inversefueter,colombosabadinisommen2013,dongkouqian} on the inversion of the Fueter-Sce theorem, along with the survey~\cite{qian2015}. In~\cite{xsfuetersce}, Xu and Sabadini uncovered the same phenomenon for generalized partial-slice monogenic functions: in our current terminology, for $f:\rr^{N+1}\to C\ell(0,N)$ that is strongly $(t_0,N)$-regular with $N-t_0=2n+1$, they proved that the $n$th Laplacian of $f$ is a monogenic function $\rr^{N+1}\to C\ell(0,N)$ with a new type of axial symmetry, which is the same as being a $(t_0,N)$-function. Eelbode's results of~\cite{eelbode2014} are also strictly related and can be reinterpreted in the following terms: for a very specific type of function $f:\rr^{N+1}\to C\ell(0,N)$ that is strongly $(0,t_1,N)$-regular with $t_1=2n_1+1$ and $N-t_1=2n_2+1$, the $(n_1+n_2)$-th Laplacian of $f$ is a monogenic function $\rr^{N+1}\to C\ell(0,N)$ with biaxial symmetry. This specific biaxial symmetry is subsumed in the concept of $(0,t_1,N)$-function. In general, being a $T$-function may be described as having $\tau$-axial symmetry. The aim of the present work is to put the Fueter-Sce phenomenon into context for the wider class of strongly $T$-regular functions over general associative $*$-algebras.

To describe the present work in further detail, after our previous choice of $T$, let us set 
\[T_1:=(t_1,\ldots,t_\tau),\ldots,T_\sigma:=(t_\sigma,\ldots,t_\tau),\ldots,T_\tau:=(t_\tau)=(N)\]
and fix $\sigma\in\{1,\ldots,\tau\}$. After preliminaries in Section~\ref{sec:preliminaries}, the operators $\debar_T,\partial_T,\Delta_T$ are defined in Section~\ref{sec:globaloperators}. Section~\ref{sec:naturalinclusion} proves that $T$-functions are naturally included among $T_1$-functions, whence among $T_\sigma$ functions. Section~\ref{sec:tildeTlaplacians} computes the effect on $T$-functions of $\debar_{T_1}$ and of $\Delta_{T_\sigma}$. In particular, it shows that both $\debar_{T_1}$ and $\Delta_{T_\sigma}$ map $T$-functions into $T$-functions. Section~\ref{sec:iteratesharmonic} computes, for any $f$ in the kernel of $\Delta_T$ and any $n\in\nn$, the value of $\Delta_{T_1}^nf$. In case $t_1-t_0=2n_1+1$, it turns out that $\Delta_{T_1}^nf\equiv0$ for all $n>n_1$. Section~\ref{sec:FueterSceTfunctions} proves three important results. The first one unifies the results of~\cite{fueter1,sce,xsfuetersce} and generalizes them to algebras other than $C\ell(0,N)$ and to vector spaces $V$ other than the paravector space $\rr^{N+1}$:

\begin{theorem}
Assume $t_1-t_0$ to be an odd natural number $2n_1+1$. For every strongly $T$-regular function $f$, the function $\Delta_{T_1}^{n_1}f$ is a strongly $T_1$-regular function and still a $T$-function.
\end{theorem}

The second result we wish to highlight is the main theorem:

\begin{theorem}
Assume there exist $n_1,\ldots,n_\sigma\in\nn$ such that $t_h-t_{h-1}=2n_h+1$ for every $h\in\{1,\ldots,\sigma\}$. For every strongly $T$-regular function $f$, the function
\[\Delta_{T_\sigma}^{n_\sigma}\ldots\Delta_{T_2}^{n_2}\Delta_{T_1}^{n_1}f\]
is a strongly $T_\sigma$-regular function and still a $T$-function.
\end{theorem}

\begin{corollary}
If $t_h-t_{h-1}$ is an odd natural number $2n_h+1$ for every $h\in\{1,\ldots,\tau\}$, then $\Delta_{T_\tau}^{n_\tau}\ldots\Delta_{T_2}^{n_2}\Delta_{T_1}^{n_1}f$ is a strongly $(N)$-regular function that is still a $T$-function. In case the codomain of $f$ is the Clifford algebra $C\ell(0,N)$ and the domain is the paravector space $\rr^{N+1}$, then $\Delta_{T_\tau}^{n_\tau}\ldots\Delta_{T_2}^{n_2}\Delta_{T_1}^{n_1}f$ is a monogenic function $\rr^{N+1}\to C\ell(0,N)$ with $\tau$-axial symmetry.
\end{corollary}

All of the aforementioned results stated are tested  on explicit polynomial examples. The concluding remarks in Section~\ref{sec:conclusions} conjecture that it is also possible, following Qian's approach of~\cite{qian1997}, to drop the oddness hypotheses in the aforementioned theorems and corollary.


\section{Preliminaries}\label{sec:preliminaries}

Recall that a real $*$-algebra of finite dimension is a finite-dimensional $\rr$-vector space, equipped with an $\rr$-bilinear multiplication and with a $*$-involution, i.e., an involutive $\rr$-linear antiautomorphism $x\mapsto x^c$.

\begin{assumption}
We fix an associative real $*$-algebra $(A,+,\cdot,^c)$ of finite dimension $d$. Moreover, we endow $A$ and all its real vector subspaces with the natural topology and differential structure as a real vector space.
\end{assumption}

Set $\nn^*:=\nn\setminus\{0\}$. For any $m\in\nn^*$, let $\mathscr{P}(m)$ denote the power set of $\{1,\ldots,m\}$. Furthermore: for all $K\in\mathscr{P}(m)$, let $|K|$ denote the cardinality of $K$.

\begin{examples}[Clifford algebras]
An associative $*$-algebra $C\ell(p,q)$, called a Clifford algebra, is constructed on the real vector space $\rr^{2^{m}}$ with $m=p+q$ by adding a multiplication and a $*$-involution, called Clifford conjugation, along the following lines:
\begin{itemize}
\item $(e_K)_{K\in\mathscr{P}(m)}$ denotes the standard basis of $\rr^{2^{m}}$; if $K=\{k_1,\ldots,k_s\}$ with $1\leq k_1<\ldots<k_s\leq m$, then the element $e_K$ is also denoted as $e_{k_1\ldots k_s}$;
\item $e_\emptyset$ is defined to be the neutral element and also denoted as $1$;
\item $e_k^2:=1$ for all $k\in\{1,\ldots,p\}$ and $e_k^2:=-1$ for all $k\in\{p+1,\ldots,m\}$;
\item if $1\leq k_1<\ldots<k_s\leq m$, then the product $e_{k_1}\cdots e_{k_s}$ is defined to be $e_{k_1\ldots k_s}$;
\item $e_he_k = -e_ke_h$ for all distinct $h,k\in\{1,\ldots,m\}$;
\item $e_K^c:=e_K$ if $|K|\equiv0,3 \mod 4$ and $e_K^c:=-e_K$ if $|K|\equiv1,2 \mod 4$.
\end{itemize}
In particular, $C\ell(0,1)$ is the field of complex numbers $\cc$ endowed with its standard conjugation $z\mapsto\bar{z}$, if we denote $e_1$ by $i$. $C\ell(0,2)$ is the skew field of quaternions $\hh$ endowed with its standard conjugation $q\mapsto\bar{q}$, if we denote $e_1,e_2,e_{12}$ by $i,j,k$, respectively. 
\end{examples}

For more details on these examples and their history, we refer the reader to~\cite{ebbinghaus,librogurlebeck2}. On $A$, we use the notations $t(x):=x+x^c$ and $n(x):=xx^c$ for all $x\in A$. The elements of
\[\s_A:=\{x\in A: t(x)=0,n(x)=1\}\]
are called the \emph{imaginary units} of $A$.

\begin{assumption}
We take the assumption $\s_A\neq\emptyset$.
\end{assumption}

The \emph{quadratic cone} of $A$,
\[Q_A:=\rr\cup\{x\in A\setminus\rr:t(x)\in\rr,n(x)\in\rr,4n(x)>t(x)^2\}\]
was defined in~\cite{perotti}, which also proved that
\[Q_A=\bigcup_{J\in\s_A}\cc_J\,,\]
where $\cc_J:=\rr+J\rr$ for all $J\in\s_A$. Now, $\cc_J$ is $*$-isomorphic to $\cc$. Thus, for any $x=\alpha+\beta J\in Q_A$ (with $\alpha,\beta\in\rr,J\in\s_A$): the conjugate $x^c=\alpha-\beta J$ belongs to $\cc_J\subset Q_A$; $t(x)=2\alpha\in\rr$; $n(x)=n(x^c)=\alpha^2+\beta^2$ is a positive real number; provided $x\neq0$, the element $x$ has a multiplicative inverse, namely $x^{-1}=n(x)^{-1}x^c=x^cn(x)^{-1}$, which still belongs to $Q_A$. In particular, $x\in Q_A\setminus\{0\}$ is neither a left nor a right zero divisor. Our previous assumption $\s_A\neq\emptyset$ guarantees that $\rr\subsetneq Q_A$. The following definition was given in~\cite[\S3]{perotticr} (see also~\cite[Lemma 1.4]{volumeintegral}).

\begin{definition}
Let $M$ be a real vector subspace of $A$. A \emph{hypercomplex basis} of $M$ is an ordered real vector basis $(v_0,v_1,\ldots,v_m)$ of $M$ such that: $m\geq1$; $v_0=1$; $v_s\in\s_A$ and $v_sv_t=-v_tv_s$ for all distinct $s,t\in\{1,\ldots,m\}$. If $\rr\subsetneq M\subseteq Q_A$, then $M$ is called a \emph{hypercomplex subspace} of $A$. If $M$ is a hypercomplex subspace of $A$, a \emph{domain} $G$ in $M$ is a nonempty connected open subset $G$ of $M$.
\end{definition}

For every ordered real vector basis $\B'=(v_0,v_1,\ldots,v_d)$ of $A$ with $v_0=1$, our $*$-algebra $A$ can be endowed with the standard Euclidean scalar product $\langle\cdot,\cdot\rangle=\langle\cdot,\cdot\rangle_{\B'}$ and norm $\Vert\cdot\Vert=\Vert\cdot\Vert_{\B'}$ associated to $\B'$, i.e., with the Hilbert space structure that makes
\[L_{\B'}:\rr^{d+1}\to A\,,\quad L_{\B'}(x_0,\ldots,x_d)=\sum_{s=0}^dx_s\,v_s=\sum_{s=0}^dv_s\,x_s\]
a Hilbert space isomorphism. We recall some properties and some examples from~\cite[\S3]{perotticr}, from~\cite{unifiednotion} and from~\cite{unifiedtheory}.

\begin{theorem}\label{thm:hypercomplexbasis}
Fix a real vector subspace $M$ of $A$. It is a hypercomplex subspace of $A$ if, and only if, it admits a hypercomplex basis $\B=(v_0,v_1,\ldots,v_m)$. If this is the case, if we complete $\B$ to a real vector basis $\B'=(v_0,v_1,\ldots,v_m,v_{m+1},\ldots,v_d)$ of $A$ and if we endow $A$ with $\langle\cdot,\cdot\rangle=\langle\cdot,\cdot\rangle_{\B'}$ and $\Vert\cdot\Vert=\Vert\cdot\Vert_{\B'}$, then
\begin{align}
&t(xy^c)=t(yx^c)=2\langle x,y\rangle\,,\label{eq:cliffordscalarproduct}\\
&n(x)=n(x^c)=\Vert x\Vert^2\,,\label{eq:cliffordnorm}
\end{align}
for all $x,y\in M$. As a consequence, the intersection $\s_A\cap M$ is a compact set: namely, the unit $(m-1)$-sphere centered at the origin in $\Span(v_1,\ldots,v_m)$, with respect to the norm $\Vert\cdot\Vert$.
\end{theorem}

\begin{example}[Paravectors]\label{ex:paravectors}
The space of paravectors $\rr^{m+1}$ is a hypercomplex subspace of $C\ell(0,m)$, with hypercomplex basis $\B=(e_\emptyset,e_1,\ldots,e_m)$. We complete $\B$ to the standard basis $\B'=(e_K)_{K\in\mathscr{P}(m)}$ of $C\ell(0,m)$. Equalities~\eqref{eq:cliffordscalarproduct},~\eqref{eq:cliffordnorm} and
\begin{equation}\label{eq:cliffordmultiplicativenorm}
\Vert ax\Vert=\Vert a\Vert\,\Vert x\Vert=\Vert x\Vert\,\Vert a\Vert=\Vert xa\Vert\,,
\end{equation}
hold true for all $a\in C\ell(0,m),x,y\in\rr^{m+1}$.

On the other hand, for $m\geq3$, the norm $\Vert\cdot\Vert$ is not multiplicative over general elements of $C\ell(0,m)$. For instance: the elements $a=1+e_{123}$ and $b=1-e_{123}$ have $ab=0$, whence $\Vert ab\Vert=0\neq2=\Vert a\Vert\Vert b\Vert$.
\end{example}

\begin{examples}\label{ex:svectors}
For any $h\in\{1,\ldots,m\}$ with $h\equiv1\,\mathrm{mod}\,4$,
\[V_h:=\left\{x_0+\sum_{1\leq k_1<\ldots<k_h\leq m}x_{k_1\ldots k_h}e_{k_1\ldots k_h} : x_0, x_{k_1\ldots k_h}\in\rr\right\}\]
is a hypercomplex subspace of $C\ell(0,m)$. It has hypercomplex basis $\B=(e_{k_1\ldots k_h})_{1\leq k_1<\ldots<k_h\leq m}$, whence $\dim V_h=\binom{m}{h}+1\geq\left(\frac{m}{h}\right)^h+1$. If we set $h(m):=4\lfloor \frac{m+2}8\rfloor+1$ (whence $\frac{m}2-2<h(m)\leq\frac{m}2+2$), then $\dim V_{h(m)}$ grows exponentially with $m$. Again, we can complete $\B$ to $\B'=(e_K)_{K\in\mathscr{P}(m)}$ of $C\ell(0,m)$. Equalities~\eqref{eq:cliffordscalarproduct},~\eqref{eq:cliffordnorm} and~\eqref{eq:cliffordmultiplicativenorm} hold true for all $a\in C\ell(0,m),x,y\in V_h$. In the special case $h=1=m$, we find that $\cc=C\ell(0,1)$ is a hypercomplex subspace of itself with $\B=\B'=\{1,i\}$. We also recover the well-known fact that Euclidean norm on $\cc$ is multiplicative.

\end{examples}

\begin{example}\label{ex:2mod4}
For every $h\leq m$ with $h\equiv2\,\mathrm{mod}\,4$, the set $W_h=\Span(e_\emptyset,e_1,e_2,\ldots,e_h,e_{12\ldots h})$ is a hypercomplex subspace of $C\ell(0,m)$ that properly includes the space of paravectors. Once more, we can complete $\B=(e_\emptyset,e_1,e_2,\ldots,e_h,e_{12\ldots h})$ to $\B'=(e_K)_{K\in\mathscr{P}(m)}$ of $C\ell(0,m)$. Equalities~\eqref{eq:cliffordscalarproduct},~\eqref{eq:cliffordnorm} and~\eqref{eq:cliffordmultiplicativenorm} hold true for all $a\in C\ell(0,m),x,y\in W_h$. In the special case $h=2=m$, we find that $\hh=C\ell(0,2)$ is a hypercomplex subspace of itself with $\B=\B'=\{1,i,j,k\}$. We also recover the well-known fact that Euclidean norm on $\hh$ is multiplicative.
\end{example}

We henceforth make the following assumption.

\begin{assumption}\label{ass:associative}
We assume $V$ to be a hypercomplex subspace of the associative real $*$-algebra $A$, with a hypercomplex basis $\B=(v_0,v_1,\ldots,v_N)$ for some $N\in\nn^*$. We complete $\B$ to a real vector basis $\B'=(v_0,v_1,\ldots,v_N,v_{N+1},\ldots,v_d)$ of $A$ and endow $A$ with the standard Euclidean scalar product $\langle\cdot,\cdot\rangle=\langle\cdot,\cdot\rangle_{\B'}$ and norm $\Vert\cdot\Vert=\Vert\cdot\Vert_{\B'}$ associated to $\B'$.
\end{assumption}

In~\cite{unifiednotion} we gave the following definitions, within the hypercomplex subspace $V$.

\begin{definition}
For $0\leq\ell<m\leq N$, we set $\rr_{\ell,m}:=\Span(v_\ell,\ldots,v_m)$. The unit $(m-\ell)$-sphere within the $(m-\ell+1)$-dimensional subspace $\rr_{\ell,m}$ is denoted by $\s_{\ell,m}$. For any \emph{number of steps} $\tau\in\{0,\ldots,N\}$ and any \emph{list of steps} $T=(t_0,\ldots,t_\tau)\in\nn^{\tau+1}$, with $0\leq t_0<t_1<\ldots<t_\tau=N$, the \emph{$T$-fan} is defined as
\[\rr_{0,t_0}\subsetneq\rr_{0,t_1}\subsetneq\ldots\subsetneq\rr_{0,t_\tau}=V\,.\]
The first subspace, $\rr_{0,t_0}$, is called the \emph{mirror}. The \emph{$T$-torus} is defined as
\[\torus:=\s_{t_0+1,t_1}\times\ldots\times\s_{t_{\tau-1}+1,t_\tau}\]
when $\tau\geq1$ and as $\torus:=\emptyset$ when $\tau=0$.
\end{definition}

We point out that $N\geq t_0+\tau$ in the previous definition. All elements of the $T$-fan are hypercomplex subspaces of $A$ with the possible exception of the mirror $\rr_{0,t_0}$, which reduces to the real axis $\rr$ if $t_0=0$. If $\tau\geq1$ then, for every $h\in\{1,\ldots,\tau\}$, the sphere $\s_{t_{h-1}+1,t_h}$ is a $(t_h-t_{h-1}-1)$-dimensional subset of the compact set $\s_A\cap V=\s_{1,N}$ and the $T$-torus $\torus$ is a $(N-t_0-\tau)$-dimensional compact set contained in $(\s_A)^\tau$.

\begin{example}[Paravectors]
If $V$ is the space $\rr^{n+1}$ of paravectors in $C\ell(0,n)$ (see Example~\ref{ex:paravectors}), then the $T$-fan is
\[\rr^{t_0+1}\subsetneq\rr^{t_1+1}\subsetneq\ldots\subsetneq\rr^{t_\tau+1}=\rr^{n+1}\,.\]
\end{example}

\begin{example}[Quaternions]
If $V=\hh$ within $\hh$: the $3$-fan is $\hh$; the $(2,3)$-fan is $\rr+i\rr+j\rr\subsetneq\hh$; the $(1,3)$-fan is $\cc\subsetneq\hh$; the $(0,3)$-fan is $\rr\subsetneq\hh$; the $(1,2,3)$-fan is $\cc\subsetneq\rr+i\rr+j\rr\subsetneq\hh$; the $(0,2,3)$-fan is $\rr\subsetneq\rr+i\rr+j\rr\subsetneq\hh$; the $(0,1,3)$-fan is $\rr\subsetneq\cc\subsetneq\hh$; and the $(0,1,2,3)$-fan is $\rr\subsetneq\cc\subsetneq\rr+i\rr+j\rr\subsetneq\hh$.
\end{example}

We recall some further material from~\cite{unifiednotion}.

\begin{remark}\label{rmk:decomposedvariable}
Any $x=\sum_{\ell=0}^Nv_\ell x_\ell\in V$ decomposes as $x=x^0+x^1+\ldots+x^\tau$, where $x^h:=\sum_{\ell=t_{h-1}+1}^{t_h}x_\ell v_\ell\in\rr_{t_{h-1}+1,t_h}$ (with $t_{-1}:=-1$). The decomposition is orthogonal, whence unique. If $\tau\geq1$, then there exist $\beta=(\beta_1,\ldots,\beta_\tau)\in\rr^\tau$ and $J=(J_1,\ldots,J_\tau)\in\torus$ such that
\begin{equation}\label{eq:decomposedvariable}
x=x^0+\beta_1J_1+\ldots+\beta_\tau J_\tau\,.
\end{equation}
Equality~\eqref{eq:decomposedvariable} holds true exactly when, for every $h\in\{1,\ldots,\tau\}$: either $x^h\neq0,\beta_h=\pm\Vert x^h\Vert$ and $J_h=\frac{x^h}{\beta_h}$; or $x^h=0,\beta_h=0$ and $J_h$ is any element of $\s_{t_{h-1}+1,t_h}$.
\end{remark}

\begin{lemma}
If $\tau\geq1$, fix $J=(J_1,\ldots,J_\tau)\in\torus$ and set
\[\rr^{t_0+\tau+1}_J:=\Span(\B_J)\,,\quad\B_J:=(v_0,v_1,\ldots,v_{t_0},J_1,\ldots,J_\tau)\,.\]
If $\tau=0$ (whence $t_0=N\geq1$), set $J:=\emptyset,\B_\emptyset:=(v_0,v_1,\ldots,v_{t_0})=\B,\rr^{t_0+1}_\emptyset:=\Span(\B_\emptyset)=V$. 
In either case, $\B_J$ is a hypercomplex basis of $\rr^{t_0+\tau+1}_J$, which is therefore a hypercomplex subspace of $A$ contained in $V$. Moreover, if $J'\in\torus$, then the equality $\rr^{t_0+\tau+1}_J=\rr^{t_0+\tau+1}_{J'}$ is equivalent to $J'\in\{\pm J_1\}\times\ldots\times\{\pm J_\tau\}$.
\end{lemma}

As remarked in~\cite{unifiedtheory}, the hypercomplex basis $\B_J:=(v_0,v_1,\ldots,v_{t_0},J_1,\ldots,J_\tau)$ of $\rr^{t_0+\tau+1}_J$ can always be completed to a basis $(\B_J)'$ of $A$ that is orthonormal with respect to $\langle\cdot,\cdot\rangle_{\B'}$, so that $\langle\cdot,\cdot\rangle_{(\B_J)'}=\langle\cdot,\cdot\rangle_{\B'}$ and $\Vert\cdot\Vert_{(\B_J)'}=\Vert\cdot\Vert_{\B'}$. In addition to the previously defined $L_{\B'}:\rr^{d+1}\to A$, we will use
\[L_{\B_J}:\rr^{t_0+\tau+1}\to \rr^{t_0+\tau+1}_J\,,\quad L_{\B_J}(x_0,\ldots,x_{t_0+\tau})=\sum_{s=0}^{t_0}x_s\,v_s+\sum_{u=1}^\tau x_{t_0+u}\,J_u,.\]
For every domain $G$ in $\rr^{t_0+\tau+1}_J$ and every $p\in\nn$, the symbol $\mathscr{C}^p(G,A)$ denotes the bilateral $A$-module of $\mathscr{C}^p$ functions $G\to A$. The following definition was given in~\cite{unifiednotion}, following~\cite{perotticr}.

\begin{definition}\label{def:Jmonogenic}
If $\tau\geq1$, fix $J=(J_1,\ldots,J_\tau)\in\torus$. If $\tau=0$, set $J:=\emptyset$. Fix a domain $G$ in $\rr^{t_0+\tau+1}_J$, set $\widehat{G}:=L_{\B_J}^{-1}(G)$ and let $\phi\in\mathscr{C}^1(G,A)$. For $s\in\{0,\ldots,t_0+\tau\}$, we define
\[\partial_s\phi:=L_{\B'}\circ\left(\frac{\partial}{\partial x_s}\left(L_{\B'}^{-1}\circ \phi\circ (L_{\B_J})_{|_{\widehat{G}}}\right)\right)\circ (L_{\B_J}^{-1})_{|_G}\in\mathscr{C}^0(G,A)\,.\]
The \emph{$J$-Cauchy-Riemann operator} $\debar_J:\mathscr{C}^{1}(G,A)\to\mathscr{C}^0(G,A)$ and the operators $\partial_J:\mathscr{C}^{1}(G,A)\to\mathscr{C}^0(G,A)$ and $\Delta_J:\mathscr{C}^{2}(G,A)\to\mathscr{C}^0(G,A)$ are defined as follows:
\begin{align*}
&\debar_J\phi:=\sum_{s=0}^{t_0}v_s\,\partial_s\phi+\sum_{u=1}^\tau J_u\,\partial_{t_0+u}\phi\,,\\
&\partial_J\phi:=\partial_0\phi-\sum_{s=1}^{t_0}v_s\,\partial_s\phi-\sum_{u=1}^\tau J_u\,\partial_{t_0+u}\phi\,,\\
&\Delta_J\phi:=\sum_{s=0}^{t_0+\tau}\partial_s^2\phi\,.
\end{align*}
The right $A$-submodule of those $\phi\in\mathscr{C}^1(G,A)$ such that $\debar_J \phi\equiv0$ is denoted by $\mon_J(G,A)$ and its elements are called \emph{$J$-monogenic} functions. The elements of the kernel of $\Delta_J$ are called \emph{$J$-harmonic} functions.
\end{definition}

According to~\cite{unifiedtheory}, for any $x\in G,s\in\{0,\ldots,t_0\},u\in\{1,\ldots,\tau\}$,
\begin{align*}
&\partial_s\phi(x)=\lim_{\rr\ni \varepsilon\to0}\varepsilon^{-1}\left(\phi(x+\varepsilon v_s)-\phi(x)\right)\,,\\
&\partial_{t_0+u}\phi(x)=\lim_{\rr\ni \varepsilon\to0}\varepsilon^{-1}\left(\phi(x+\varepsilon J_u)-\phi(x)\right)\,.
\end{align*}
As a consequence, the operators $\debar_J,\partial_J,\Delta_J$ do not depend on the whole basis $\B'$ of $A$ chosen, but only on the choice of $J$ in $\torus$. Informally, referring to the decomposition~\eqref{eq:decomposedvariable} of the variable $x$, we have
\begin{align*}
\debar_J&=\partial_{x_0}+v_1\partial_{x_1}+\ldots+v_{t_0}\partial_{x_{t_0}}+J_1\partial_{\beta_1}+\ldots+J_\tau\partial_{\beta_\tau}\,,\\
\partial_J&=\partial_{x_0}-v_1\partial_{x_1}-\ldots-v_{t_0}\partial_{x_{t_0}}-J_1\partial_{\beta_1}-\ldots-J_\tau\partial_{\beta_\tau}\,,\\
\Delta_J&=\partial_{x_0}^2+\partial_{x_1}^2+\ldots+\partial_{x_{t_0}}^2+\partial_{\beta_1}^2+\ldots+\partial_{\beta_\tau}^2\,.
\end{align*}
Using the formal definition of $\debar_J$ is necessary to guarantee that, when $J,J'\in\torus$ are such that $\rr^{t_0+\tau+1}_J=\rr^{t_0+\tau+1}_{J'}$, then $\debar_J=\debar_{J'}$. Similar considerations apply to $\partial_J,\Delta_J$. The equalities $\debar_J\partial_J=\partial_J\debar_J=\Delta_J$ hold true on $\mathscr{C}^{2}(G,A)$. Moreover,~\cite{unifiedtheory}proved that $J$-monogenic functions are $J$-harmonic, whence real analytic. In the special case when $\tau=0$, whence $t_0=N$, our last definition can be rephrased informally as $\debar_\emptyset:=\debar_\B=\partial_{x_0}+v_1\partial_{x_1}+\ldots+v_N\partial_{x_N}$, as well as $\partial_\emptyset:=\partial_\B=\partial_{x_0}-v_1\partial_{x_1}-\ldots-v_N\partial_{x_N}$ and $\Delta_\emptyset:=\Delta_\B=\partial_{x_0}^2+\partial_{x_1}^2+\ldots+\partial_{x_N}^2$.

We now recall the concept of $T$-regular function from~\cite{unifiednotion}. In the special case with $A=C\ell(0,n),V=\rr^{n+1}$ and $\tau=1$, it was independently constructed in~\cite{xsgeneralizedpartialslice} (see also~\cite{xsannouncement}) under the name of \emph{generalized partial-slice monogenic function}.

\begin{definition}
For any $Y\subseteq V,f:Y\to A$ and $J\in\torus$ (or $J=\emptyset$, in case $\tau=0$), the intersection $Y_J:=Y\cap\rr^{t_0+\tau+1}_J$ is called the \emph{$J$-slice} of $Y$ and we set $f_J:=f_{|_{Y_J}}$. Now fix a domain $\Omega$ in $V$. A function $f:\Omega\to A$ is termed \emph{$T$-regular} if the restriction $f_J:\Omega_J\to A$ is $J$-monogenic for every $J\in\torus$, if $\tau\geq1$ (for $J=\emptyset$, if $\tau=0$). It is termed \emph{$T$-harmonic} if $f_J$ is $J$-harmonic for every $J\in\torus$, if $\tau\geq1$ (for $J=\emptyset$, if $\tau=0$). The class of $T$-regular functions $\Omega\to A$ is denoted by $\reg_T(\Omega,A)$.
\end{definition}

The class $\reg_T(\Omega,A)$ is a right $A$-module. Furthermore, if $f\in\reg_T(\Omega,A)$ and $p\in\rr_{0,t_0}$, then setting $g(x):=f(x+p)$ defines a $g\in\reg_T(\Omega-p,A)$. Over $C\ell(0,n)$, $T$-regularity comprises a spectrum of function theories, with the two best-known function theories sitting at the edges of the spectrum.

\begin{example}[Paravectors]
Fix a domain $\Omega$ within the paravector subspace $\rr^{n+1}$ of $C\ell(0,N)$ (see Example~\ref{ex:paravectors}). For any function $f:\Omega\to C\ell(0,N)$:
\begin{itemize}
\item $f$ is $(N)$-regular if, and only if, it is in the kernel of the operator $\partial_{x_0}+e_1\partial_{x_1}+\ldots+e_{N}\partial_{x_N}$; this is the definition of \emph{monogenic} function (see, e.g.,~\cite{librosommen,librocnops,librogurlebeck2});
\item $f$ is $(0,N)$-regular if, and only if, for any $J_1\in\s_{1,N}=\s_{C\ell(0,N)}\cap\rr^{N+1}$, the restriction $f_{J_1}$ to the planar domain $\Omega_{J_1}\subseteq\cc_{J_1}$ is a holomorphic map $(\Omega_{J_1},J_1)\to(C\ell(0,N),J_1)$; this is the same as being \emph{slice-monogenic},~\cite{israel} (or \emph{slice-hyperholomorphic},~\cite{librodaniele2}).
\end{itemize}
\end{example}

Distinct choices of $T$ do not necessarily produce distinct function classes. This general fact is made more precise in the following result from~\cite{unifiedtheory}.

\begin{proposition}\label{prop:shortsteps}
Let $T,T'$ be two lists of steps for $V$. The inclusion $\reg_T(\Omega,A)\subseteq\reg_{T'}(\Omega,A)$ is equivalent to the equality $\reg_T(\Omega,A)=\reg_{T'}(\Omega,A)$ and to the following property: one among the lists $T,T'$ comprises the other, possibly preceded by some steps of the form $(m,m+1)$.
\end{proposition}

A complete classification of $T$-regularity over the hypercomplex subspace $\hh$ of $\hh$ has been achieved in~\cite{unifiednotion}, in the following terms.

\begin{example}[Quaternions]
Fix a domain $\Omega$ in $\hh$ and a function $f:\Omega\to\hh$. Then:
\begin{itemize}
\item $f$ is $(3)$-regular $\Leftrightarrow f$ belongs to the kernel of the left Cauchy-Riemann-Fueter operator $\partial_{x_0}+i\partial_{x_1}+j\partial_{x_2}+k\partial_{x_3} \Leftrightarrow f$ is a \emph{left Fueter-regular} function (see~\cite{fueter1,fueter2,sudbery});
\item $f$ is $(1,3)$-regular $\Leftrightarrow$
\[\hskip35pt\debar_{J_1}f(x_0+ix_1+\beta_1J_1)=(\partial_{x_0}+i\partial_{x_1}+J_1\partial_{\beta_1})f(x_0+ix_1+\beta_1J_1)\equiv0\]
for all $J_1$ in the $(1,3)$-torus $\s_{2,3}$, which is simply the circle $\s_\hh\cap(j\rr+k\rr)$ (a theory studied in~\cite{unifiednotion});
\item $f$ is $(0,3)$-regular $\Leftrightarrow $ for any $J_1\in\s_{1,3}=\s_\hh$, the restriction $f_{J_1}$ to the planar domain $\Omega_{J_1}\subseteq\cc_{J_1}$ is a holomorphic map $(\Omega_{J_1},J_1)\to(\hh,J_1) \Leftrightarrow f$ is a \emph{slice-regular} function,~\cite{librospringer2} (or \emph{Cullen-regular} in the original articles~\cite{cras,advances}).
\end{itemize}
The classes of $(2,3)$-regular functions, $(1,2,3)$-regular functions, and $(0,1,2,3)$-regular functions all coincide with the class of $(3)$-regular (or left-Fueter regular) functions. The class of $(0,1,3)$-regular functions coincides with the class of $(1,3)$-regular functions. The class of $(0,2,3)$-regular functions (does not coincide with, but) is conjugate to the class of $(0,1,3)$-regular functions via the real vector space isomorphism $\hh\to\hh$ mapping the standard basis $(1,i,j,k)$ into $(1,k,-j,i)$.
\end{example}

Recall that, for $x=x_0+v_1x_1+\ldots+v_nx_n$, we have set
\begin{align*}
x^0&=x_0+v_1x_1+\ldots+v_{t_0}x_{t_0}\,,\\
x^1&=v_{t_0+1}x_{t_0+1}+\ldots+v_{t_1}x_{t_1}\,,\\
\vdots\\
x^\tau&=v_{t_{\tau-1}+1}x_{t_{\tau-1}+1}+\ldots+v_{t_\tau}x_{t_\tau}\,,
\end{align*}
where $t_\tau=N$ by construction. Let us consider the elements $\epsilon_1=(1,0,\ldots,0)$, $\epsilon_2=(0,1,\ldots,0)$, \ldots, $\epsilon_{t_0+\tau}=(0,0,\ldots,1)$ of $\nn^{t_0+\tau}$. The article~\cite{unifiedtheory} constructed the following polynomial $T$-regular functions.

\begin{definition}
We set $\T_\k:\equiv0$ if $\k\in\zz^{t_0+\tau}\setminus\nn^{t_0+\tau}$ and $\T_\k:\equiv1$ if $\k=(0,\ldots,0)$. For $\k=(k_1,\ldots,k_{t_0+\tau})\in\nn^{t_0+\tau}\setminus\{(0,\ldots,0)\}$, we define recursively
\begin{align*}
|\k|\T_\k(x)&:=\sum_{s=1}^{t_0}k_s\T_{\k-\epsilon_s}(x)\left(x_s-(-1)^ax_0v_s\right)+\sum_{s=t_0+1}^{t_0+\tau}(-1)^{b_s}k_s\T_{\k-\epsilon_s}(x)\left(x_0+(-1)^{a_s}x^{s-t_0}\right)
\end{align*}
where $a:=\sum_{u=t_0+1}^{t_0+\tau}k_u, a_s:=a-k_s$ and $b_s:=\sum_{u=s+1}^{t_0+\tau}k_u$. For all $k\in\nn$, we define
\[\F_k:=\{\T_\k\}_{|\k|=k}\,.\]
\end{definition}

We point out that $a_s+b_s=\sum_{u=t_0+1}^{s-1}k_u+2b_s$, whence $(-1)^{a_s+b_s}=(-1)^{c_s},c_s:=\sum_{t_0<u<s}k_u$. While the definition we gave is somewhat technical, it has a strong link to the analogs, in the theory of $J$-monogenic functions, of the classical Fueter polynomials over Clifford algebras. This link is made explicit in~\cite[Lemma 5.8]{unifiedtheory}. For instance, the set $\F_1$ consists of the functions
\begin{align*}
&\T_{\epsilon_s}(x)=x_s-x_0v_s&&1\leq s\leq t_0\,,\\
&\T_{\epsilon_{t_0+u}}(x)=x_0+x^u&&1\leq u\leq\tau\,.
\end{align*}
The functions $\T_{\epsilon_1},\ldots,\T_{\epsilon_{t_0}}$ are called the \emph{$T$-Fueter variables} and the functions $\T_{\epsilon_{t_0+1}},\ldots,\T_{\epsilon_{t_0+\tau}}$ are called the \emph{$T$-Cullen variables}. For every $k\in\nn$,~\cite[Theorem 5.14]{unifiedtheory} proves that the finite sequence $\F_k=\{\T_\k\}_{|\k|=k}$ is an $A$-basis for the right $A$-module of $T$-regular functions $V\to A$ that are homogeneous real polynomial maps of degree $k$.

\begin{example}
When $T=(N)$ (whence $t_0=N,\tau=0$), the set $\F_1$ comprises the $T$-Fueter variables $x_1-x_0v_1,\ldots,x_N-x_0v_N$. In the special case when $A=C\ell(0,N)$ and $V$ is the space of paravectors $\rr^{N+1}$, the polynomials $\{\T_\k\}_{\k\in\zz^N}$ coincide with the Fueter polynomials in the classical theory of monogenic functions $\rr^{N+1}\to C\ell(0,N)$. In the special case when $A=\hh=V$, the classical Fueter-regular polynomial functions $\hh\to\hh$ are recovered.
\end{example}

\begin{example}
When $T=(0,N)$ (whence $\tau=1$), the set $\F_1$ comprises one $T$-Cullen variable, which is the whole variable $x$, and for every $k\in\nn$ the set $\F_k$ consists of the single function $x^k$. In the special case when $A=C\ell(0,N)$ and $V$ is the space of paravectors $\rr^{N+1}$, this fact is consistent with the well-known fact that polynomial slice-monogenic functions are restrictions to $V$ of elements of $A[x]$. In the special case when $A=\hh=V$ (whence $N=3$), it is consistent with fact that the set of polynomial quaternionic slice-regular functions coincides with $\hh[x]$.
\end{example}

\begin{example}
Let $A=\hh=V$ and $T=(1,3)$. The set $\F_1$ comprises one $T$-Fueter variable, $\T_{(1,0)}(x)=x_{1}-ix_{0}$, and one $T$-Cullen variable, $ \T_{(0,1)}(x)=x_0+x^1=x_{0}+jx_{2}+kx_{3}$. Neither is a slice-regular
function $\hh\to \hh$ and the latter is not Fueter-regular.
\end{example}

Let us now compute some higher-degree examples, which we will repeatedly use in the present work.

\begin{example}\label{ex:(0,3,6)}
If $A=C\ell(0,6)$, if $V$ is the space of paravectors $\rr^7$ and if $T=(0,3,6)$, then $\F_1$ comprises two $T$-Cullen variables:
\[\T_{(1,0)}(x)=x_0+x^1=x_0+x_1e_1+x_2e_2+x_3e_3\,,\quad \T_{(0,1)}(x)=x_0+x^2=x_0+x_4e_4+x_5e_5+x_6e_6\,.\]
$\F_{2}=\{\T_{(2,0)},\T_{(1,1)},\T_{(0,2)}\}$, where
\begin{align*}
&\T_{(2,0)}(x)=(x_0+x^1)^{2}=x_0^2+2x_0x^1-\Vert x^1\Vert^2\,,\quad \T_{(0,2)}(x)=(x_0+x^2)^{2}=x_0^2+2x_0x^2-\Vert x^2\Vert^2\,,\\
&2\T_{(1,1)}(x)=-(x_0+x^2)(x_0-x^1)+(x_0+x^1)(x_0-x^2)=2(x_0x^1-x_0x^2-x^1x^2)\,.
\end{align*}
$\F_{3}=\{\T_{(3,0)},\T_{(2,1)},\T_{(1,2)},\T_{(0,3)}\}$, where
\begin{align*}
\T_{(3,0)}(x)&=(x_0+x^1)^{3}=x_0^3+3x_0^2x^1-3x_0\Vert x^1\Vert^2-\Vert x^1\Vert^2x^1\,,\\
\T_{(0,3)}(x)&=(x_0+x^2)^{3}=x_0^3+3x_0^2x^2-3x_0\Vert x^2\Vert^2-\Vert x^2\Vert^2x^2\,,\\
3\T_{(2,1)}(x)&=-2\T_{(1,1)}(x)(x_0-x^1)+\T_{(2,0)}(x)(x_0+x^2)\\
&=x_0^3+3x_0^2x^2-3x_0\Vert x^1\Vert^2+6x_0x^1x^2-3\Vert x^1\Vert^2x^2\,,\\
3\T_{(1,2)}(x)&=\T_{(0,2)}(x)(x_0+x^1)+2\T_{(1,1)}(x)(x_0-x^2)\\
&=x_0^3+3x_0^2x^1-6x_0x^1x^2-3x_0\Vert x^2\Vert^2-3x^1\Vert x^2\Vert^2\,.
\end{align*}
The following example of degree $5$ will be relevant later:
\begin{align*}
10\T_{(3,2)}(x)&=6\T_{(2,2)}(x)(x_0+x^1)+4\T_{(3,1)}(x)(x_0-x^2)\\
&=3\T_{(1,2)}(x)(x_0+x^1)^2+3\T_{(2,1)}(x)(x_0+x^2)(x_0+x^1)\\
&\quad-3\T_{(2,1)}(x)(x_0-x^1)(x_0-x^2)+\T_{(3,0)}(x)(x_0-x^2)^2\\
&=3\T_{(1,2)}(x)(x_0^2+2x_0x^1-\Vert x^1\Vert^2)+6\T_{(2,1)}(x)(x_0x^1+x_0x^2-x^1x^2)\\
&\quad+\T_{(3,0)}(x)(x_0^2-2x_0x^2-\Vert x^2\Vert^2)\,.
\end{align*}
\end{example}

\begin{example}\label{ex:(1,4,7)}
Let $A=C\ell(0,6)$ as in the previous example but let us now choose $V=\Span(e_\emptyset,e_1,e_2,e_3,e_4,e_5,e_6,e_{123456})$. The choice of this eight-dimensional space, which properly includes the seven-dimensional space of paravectors, is a special case of Example~\ref{ex:2mod4}. Let $T=(1,4,7)$. The set $\F_1$ comprises one $T$-Fueter variable,
\[\T_{(1,0,0)}(x)=x_1-x_0e_1\,,\]
and two $T$-Cullen variables,
\begin{align*}
\T_{(0,1,0)}(x)&=x_0+x^1=x_0+x_2e_2+x_3e_3+x_4e_4\,,\\
\T_{(0,0,1)}(x)&=x_0+x^2=x_0+x_5e_5+x_6e_6+x_{123456}e_{123456}\,.
\end{align*}
$\F_{2}$ consist of the functions
\begin{align*}
&\T_{(2,0,0)}(x)=(x_1-x_0e_1)^2\,,\quad\T_{(0,2,0)}(x)=(x_0+x^1)^{2}\,,\quad \T_{(0,0,2)}(x)=(x_0+x^2)^{2}\,,\\
&\T_{(1,1,0)}(x)=\frac12(x_0+x^1)(x_1+x_0e_1)+\frac12(x_1-x_0e_1)(x_0+x^1)=x_0x_1-x_0e_1x^1+x_1x^1\,,\\
&\T_{(1,0,1)}(x)=\frac12(x_0+x^2)(x_1+x_0e_1)+\frac12(x_1-x_0e_1)(x_0+x^2)=x_0x_1-x_0e_1x^2+x_1x^2\,,\\
&\T_{(0,1,1)}(x)=-\frac12(x_0+x^2)(x_0-x^1)+\frac12(x_0+x^1)(x_0-x^2)=x_0x^1-x_0x^2-x^1x^2\,.
\end{align*}
The following example of degree $4$ will be relevant later:
\begin{align*}
6\T_{(1,2,1)}(x)&=\frac32\T_{(0,2,1)}(x)(x_1+x_0e_1)-3\T_{(1,1,1)}(x)(x_0-x^1)+\frac32\T_{(1,2,0)}(x)(x_0+x^2)\\
&=\frac12\left(-2\T_{(0,1,1)}(x)(x_0-x^1)+\T_{(0,2,0)}(x)(x_0+x^2)\right)(x_1+x_0e_1)\\
&\quad-\left(\T_{(0,1,1)}(x)(x_1-x_0e_1)-\T_{(1,0,1)}(x)(x_0-x^1)+\T_{(1,1,0)}(x)(x_0-x^2)\right)(x_0-x^1)\\
&\quad+\frac12\left(\T_{(0,2,0)}(x)(x_1-x_0e_1)+2\T_{(1,1,0)}(x)(x_0+x^1)\right)(x_0+x^2)\\
&=\T_{(0,1,1)}(x)(-2x_0x_1-2x_0e_1x^1+2x_1x^1)+\T_{(0,2,0)}(x)(x_0x_1-x_0e_1x^2+x_1x^2)\\
&\quad+\T_{(1,0,1)}(x)(x_0^2-2x_0x^1-\Vert x^1\Vert^2)+\T_{(1,1,0)}(x)(2x_0x^1+2x_0x^2+2x^1x^2)\,.
\end{align*}
\end{example}

While general $T$-regular functions need not be continuous, a careful choice of the domain $\Omega$ guarantees better-behaved $T$-regular functions $f:\Omega\to A$.

\begin{definition}\label{def:Tslicedomain}
A domain $\Omega\subseteq V$ is called a \emph{$T$-slice domain} if it intersects the mirror $\rr_{0,t_0}$ and if, for any $J\in\torus$, the $J$-slice $\Omega_J$ is connected (whence a domain in $\rr^{t_0+\tau+1}_J$). 
\end{definition}

Over $T$-slice domains, the work~\cite{unifiedtheory} proved the following results.

\begin{theorem}[Identity Principle]\label{thm:identityprinciple}
Let $\Omega\subseteq V$ be a $T$-slice domain and $f,g\in\reg_T(\Omega,A)$. If there exists $J\in\torus$ such that the $J$-slice $\Omega_J$ (whose dimension is $t_0+\tau+1$) contains a set of Hausdorff dimension $s\geq t_0+\tau$ where $f_J$ and $g_J$ coincide, then $f=g$ throughout $\Omega$.
\end{theorem}

\begin{proposition}[Maximum Modulus Principle]
Let $\Omega$ be a $T$-slice domain in $V$ and $f\in\reg_T(\Omega,A)$. If the function $\Vert f\Vert:\Omega\to\rr$ has a global maximum point in $\Omega$, then $f$ is constant in $\Omega$.
\end{proposition}

Symmetry, defined according to the following construction, is another relevant property for the domain $\Omega$ of a $T$-regular function.

\begin{definition}\label{def:Tsymmetric}
For every $h\in\{1,\ldots,\tau\}$, we define the reflection
\[\rr^\tau\to\rr^\tau\,,\quad\beta=(\beta_1,\ldots,\beta_\tau)\mapsto\overline{\beta}^h:=(\beta_1,\ldots,\beta_{h-1},-\beta_h,\beta_{h+1},\ldots,\beta_\tau)\,.\] 
For any $\beta=(\beta_1,\ldots,\beta_\tau)\in\rr^\tau,J=(J_1,\ldots,J_\tau)\in\torus$, we set
\[\beta\,J:=\beta_1J_1+\ldots+\beta_\tau J_\tau\in V\,.\]
If $\tau=0$, for $\beta\in\rr^0=\{0\}$ and $J=\emptyset$ we define $\beta\,J$ to be the zero element of $V$.
For every $D\subseteq\rr_{0,t_0}\times\rr^\tau$, we set
\[\Omega_D:=\{\alpha+\beta\,J: (\alpha,\beta)\in D, J\in\torus\}\subseteq V\]
if $\tau\geq1$ (and $\Omega_D:=\{\alpha\in V: (\alpha,0)\in D\}$ if $\tau=0$). A subset of $V$ is called \emph{$T$-symmetric} if it equals $\Omega_D$ for some $D\subseteq\rr_{0,t_0}\times\rr^\tau$. The \emph{$T$-symmetric completion} of a set $Y\subseteq V$ is the smallest $T$-symmetric subset of $V$ containing $Y$. For any $x\in V$, we denote by $\torus_x$ the $T$-symmetric completion of the singleton $\{x\}$.
\end{definition}

\begin{assumption}\label{ass:domain}
Assume $D$ to be a subset of $\rr_{0,t_0}\times\rr^\tau$, invariant under the reflection $(\alpha,\beta)\mapsto(\alpha,\overline{\beta}^h)$ for every $h\in\{1,\ldots,\tau\}$.
\end{assumption}

The article~\cite{unifiedtheory} proved for every $T$-regular function $f$ on a $T$-symmetric $T$-slice domain a symmetry property, which implies that $f$ is real analytic. This symmetry property is best stated in terms of $T$-functions $\Omega_D\to A$, defined in~\cite{unifiednotion} by means of the concept of $T$-stem function $D\to A\otimes\rr^{2^\tau}$. We point out that $A\otimes\rr^{2^\tau}$ is a bilateral $A$-module with $A$-basis equal to the canonical real vector basis $(E_K)_{K\in\mathscr{P}(\tau)}$ of $\rr^{2^\tau}$. Let us first recall the definition of $T$-stem function, which subsumes the notion of stem function of~\cite[Definition 4]{perotti} and follows the lines of its multivariate generalization~\cite[Definition 2.2]{gpseveral}. 

\begin{definition}
Consider a map $F=\sum_{K\in\mathscr{P}(\tau)}E_KF_K:D\to A\otimes\rr^{2^\tau}$, with components $F_K:D\to A$. The map $F$ is called a \emph{$T$-stem function} if
\[F_K(\alpha,\overline{\beta}^h)=\left\{
\begin{array}{ll}
F_K(\alpha,\beta)&\mathrm{if\ }h\not\in K\\
-F_K(\alpha,\beta)&\mathrm{if\ }h\in K
\end{array}
\right.\]
for all $K\in\mathscr{P}(\tau)$, for all $h\in\{1,\ldots,\tau\}$ and for all $(\alpha,\beta)\in D$. The class of $T$-stem functions $D\to A\otimes\rr^{2^\tau}$ is denoted by $\stem_T=\stem_T(D,A\otimes\rr^{2^\tau})$. If $D$ is an open subset of $\rr_{0,t_0}\times\rr^\tau$ and $p\in\nn\cup\{\infty,\omega\}$, we let $\stem_T^p=\stem_T^p(D,A\otimes\rr^{2^\tau})$ denote the set of $F\in\stem_T(D,A\otimes\rr^{2^\tau})$ such that $F_K\in\mathscr{C}^p(D,A)$ for all $K\in\mathscr{P}(\tau)$.
\end{definition}

We recall that the symbol $\mathscr{C}^\omega(D,A)$ denotes the bilateral $A$-module of real analytic functions $D\to A$. We also point out that $\stem_T(D,A\otimes\rr^{2^\tau})$ is a right $A$-module and that, if $D$ is an open subset of $\rr_{0,t_0}\times\rr^\tau$, then
\[\stem_T^\omega\subset\stem_T^\infty\subset\ldots\subset\stem_T^2\subset\stem_T^1\subset\stem_T^0\subset\stem_T\]
are nested right $A$-submodules. Moreover, for any $F\in\stem_T(D,A\otimes\rr^{2^\tau})$ and any $\rho\in\rr_{0,t_0}$, setting $G(\alpha,\beta):=F(\alpha+\rho,\beta)$ defines a $G\in\stem_T(D-(\rho,0),A\otimes\rr^{2^\tau})$. 

We are now ready to recall from~\cite{unifiednotion} the next definition: that of $T$-function. This notion subsumes the notion of slice function,~\cite[Definition 5]{perotti}, in its associative sub-case. The definition follows the lines of~\cite[Definition 2.5]{gpseveral}, in its associative sub-case.

\begin{definition}\label{def:Tfunction}
For any $J\in\torus$: we set $J_\emptyset:=1$ and, for $K=\{k_1,\ldots,k_p\}$ with $1\leq k_1<\ldots<k_p\leq\tau$, we set $J_K:=J_{k_1}J_{k_2}\cdots J_{k_{p-1}}J_{k_p}$.

For any $T$-stem function $F=\sum_{K\in\mathscr{P}(\tau)}E_KF_K:D\to A\otimes\rr^{2^\tau}$, the \emph{induced} function $f=\I(F):\Omega_D\to A$, is defined at $x=\alpha+\beta\,J\in\Omega_D$ by the formula
\[f(x):=\sum_{K\in\mathscr{P}(\tau)}J_K\,F_K(\alpha,\beta)\,.\]
A function induced by a $T$-stem function is called a \emph{$T$-function}. The class of $T$-functions $\Omega_D\to A$ is denoted by $\slice_T(\Omega_D,A)$. If $D$ is an open subset of $\rr_{0,t_0}\times\rr^\tau$ and  $p\in\nn\cup\{\infty,\omega\}$, then $\slice_T^p(\Omega_D,A)$ denotes the image of $\stem_T^p(D,A\otimes\rr^{2^\tau})$ through $\I$. If $\Omega_D$ is a domain in $V$, a \emph{strongly $T$-regular} function on $\Omega$ is an element of the intersection $\sr_T(\Omega_D,A):=\slice_T(\Omega_D,A)\cap\reg_T(\Omega_D,A)$.
\end{definition}

The article~\cite{unifiedtheory} proved that $\slice_T(\Omega_D,A)$ is a right $A$-module and that
\[\I:\stem_T(D,A\otimes\rr^{2^\tau})\to\slice_T(\Omega_D,A)\]
is a well-defined right $A$-module isomorphism. As a consequence, if $D$ is open, then
\[\slice_T^\omega\subset\slice_T^\infty\subset\ldots\subset\slice_T^2\subset\slice_T^1\subset\slice_T^0\subset\slice_T\]
are nested right $A$-submodules. Moreover,~\cite{unifiedtheory} proved that, if $\Omega_D$ is a domain in $V$, then $\sr_T(\Omega_D,A)$ is a right $A$-submodule of $\slice_T^\omega(\Omega_D,A)$. Additionally, it proved that for any $p\in\rr_{0,t_0}$ and any $f\in\slice_T(\Omega_D,A)$ (or $f\in\sr_T(\Omega_D,A)$), setting $g(x):=f(x+p)$ defines a $g\in\slice_T(\Omega_D-p,A)$ (a $g\in\sr_T(\Omega_D-p,A)$, respectively).  Finally,~\cite{unifiedtheory} proved the following Representation Formula along tori of the form $\torus_{\alpha+\beta I}=\Omega_{\{(\alpha,\beta)\}}=\alpha+\beta\,\torus$ with $\alpha\in\rr_{0,t_0},\beta\in\rr^\tau,I\in\torus$ (see Definition~\ref{def:Tsymmetric}).

\begin{theorem}[Representation Formula, $T$-functions]\label{thm:representationformula}
Fix $I\in\torus$. Every $f\in\slice_T(\Omega_D,A)$ is induced by a unique $T$-stem function $F=\sum_{K\in\mathscr{P}(\tau)}E_KF_K$, whose $K$-component is
\[F_K(\alpha,\beta)=2^{-\tau}I_K^{-1}\,\sum_{H\in\mathscr{P}(\tau)}(-1)^{|K \cap H|}\,f(\alpha+\overline{\beta}^H\,I)\,.\]
In particular, $F$ depends only on the restriction $f_I$. Moreover,
\begin{align}\label{eq:representationformula}
f(\alpha+\beta\,J)&=\sum_{H\in\mathscr{P}(\tau)}\gamma_H\,f(\alpha+\overline{\beta}^H\,I)
\end{align}
for all $(\alpha,\beta)\in D$ and all $J\in\torus$,
where
\[\gamma_H:=2^{-\tau}\sum_{K\in\mathscr{P}(\tau)}(-1)^{|K \cap H|}\,J_K\,I_K^{-1}\,.\]
\end{theorem}

In the previous statement, the symbol $I_K^{-1}$ denotes the multiplicative inverse of $I_K$. Following the lines of~\cite[Corollary 2.16]{gpseveral}, we now prove the converse result. The proof is based on a preliminary technical remark.

\begin{remark}\label{rmk:symmetricdifference}
Fix any $h\in\{1,\ldots,\tau\}$. The map $\mathscr{P}(\tau)\to\mathscr{P}(\tau),\ H\mapsto H\bigtriangleup\{h\}$ is an involutive bijection from $\mathscr{P}(\tau)$ onto itself. Moreover, if $\beta\in\rr^\tau$ and $\gamma=\overline{\beta}^H$, then $\overline{\gamma}^h=\overline{\beta}^{H\bigtriangleup\{h\}}$. Finally, for any $K\in\mathscr{P}(\tau)$, the cardinality $|K \cap(H\bigtriangleup\{h\})|$ equals $|K \cap H|$ if $h\not\in K$; it equals $|K \cap H|\pm1$ if $h\in K$.
\end{remark}

\begin{proposition}\label{prop:representationformula}
A function $f:\Omega_D\to A$ is a $T$-function if there exists $I\in\torus$ such that \eqref{eq:representationformula} holds true for all $(\alpha,\beta)\in D$ and all $J\in\torus$.
\end{proposition}

\begin{proof}
Assume there exists $I\in\torus$ such that \eqref{eq:representationformula} holds true for all $(\alpha,\beta)\in D$ and all $J\in\torus$. Let us define $F:=\sum_{K\in\mathscr{P}(\tau)}E_KF_K:D\to A\otimes\rr^{2^\tau}$, where
\[F_K(\alpha,\beta):=2^{-\tau}I_K^{-1}\,\sum_{H\in\mathscr{P}(\tau)}(-1)^{|K \cap H|}\,f(\alpha+\overline{\beta}^H\,I)\]
for all $K\in\mathscr{P}(\tau)$ and all $(\alpha,\beta)\in D$. We are going to prove that $F$ is a $T$-stem function, whence the equality $f=\I(F)$ and the thesis $f\in\slice_T(\Omega_D,A)$ will follow at once. By Remark~\ref{rmk:symmetricdifference}, for any $(\alpha,\beta)\in D$ the expression
\begin{align*}
F_K(\alpha,\overline{\beta}^h)&=2^{-\tau}I_K^{-1}\,\sum_{H\in\mathscr{P}(\tau)}(-1)^{|K \cap H|}\,f(\alpha+\overline{\beta}^{H\bigtriangleup\{h\}}\,I)\\
&=2^{-\tau}I_K^{-1}\,\sum_{H'\in\mathscr{P}(\tau)}(-1)^{|K \cap(H'\bigtriangleup\{h\})|}\,f(\alpha+\overline{\beta}^{H'}\,I)
\end{align*}
equals
\begin{align*}
2^{-\tau}I_K^{-1}\,\sum_{H'\in\mathscr{P}(\tau)}(-1)^{|K \cap H'|}\,f(\alpha+\overline{\beta}^{H'}\,I)=F_K(\alpha,\beta)
\end{align*}
when $h\not\in K$ and equals
\begin{align*}
2^{-\tau}I_K^{-1}\,\sum_{H'\in\mathscr{P}(\tau)}(-1)^{|K \cap H'|\pm1}\,f(\alpha+\overline{\beta}^{H'}\,I)=-F_K(\alpha,\beta)
\end{align*}
when $h\in K$. Our claim that $F$ is a $T$-stem function is therefore proven.
\end{proof}

Throughout the rest of the paper, we make an extra assumption.

\begin{assumption}\label{ass:opendomain}
Assume $\Omega_D$ to be an open subset of $V$, whence $D$ is an open subset of $\rr_{0,t_0}\times\rr^\tau$.
\end{assumption}

The next remark, also from~\cite{unifiedtheory}, will be useful in the paper. For $\beta=(\beta_1,\ldots,\beta_\tau)\in\rr^\tau$, we set the additional notation $\beta^2:=(\beta_1^2,\ldots,\beta_\tau^2)\in\rr^\tau$.

\begin{remark}\label{rmk:whitney}
Let $p\in\{\infty,\omega\}$ and let $F=\sum_{K\in\mathscr{P}(\tau)}E_KF_K\in\stem_T^p(D,A\otimes\rr^{2^\tau)}$. Set $\widehat{D}:=\{(\alpha,\beta^2):(\alpha,\beta)\in D\}\subset\rr_{0,t_0}\times\rr^\tau$. By Whitney's Theorem~\cite[page 160]{whitney}, there exist an open neighborhood $W$ of $\widehat{D}$ in $\rr_{0,t_0}\times\rr^\tau$, with $\widehat{D}=\{(\alpha,\gamma)\in W:\gamma_1,\ldots,\gamma_\tau\geq0\}$, and a finite sequence $\{G_K\}_{K\in\mathscr{P}(\tau)}$ in $\mathscr{C}^\infty(W,A)$ (or in $\mathscr{C}^\omega(W,A)$, respectively) such that, for all $(\alpha,\beta)\in D$, the following equalities hold true: $F_\emptyset(\alpha,\beta)=G_\emptyset(\alpha,\beta^2)$ and
\[F_K(\alpha,\beta)=\beta_{k_1}\cdots\beta_{k_p}\,G_K(\alpha,\beta^2)\]
if $K=\{k_1,\ldots,k_p\}$ with $1\leq k_1<\ldots<k_p\leq\tau$.
\end{remark}

Theorem~\ref{thm:representationformula} and Remark~\ref{rmk:whitney} allowed to prove the next result in~\cite[Proposition 6.11]{unifiedtheory}.

\begin{proposition}\label{prop:analyticstronglyTregular}
Every strongly $T$-regular function is a real analytic function.
\end{proposition}

We can now prove a related result.

\begin{proposition}\label{prop:analyticTfunction}
If $p\in\nn\cup\{\infty,\omega\}$, if $f\in\mathscr{C}^p(\Omega_D,A)$ and if there exists an open dense subset $D'$ of $D$ such that $f_{|_{\Omega_{D'}}}$ is a $T$-function, then $f\in\slice_T^p(\Omega_D,A)$. If, moreover, $p\in\{\infty,\omega\}$, then $\slice_T^p(\Omega_D,A)\subset\mathscr{C}^p(\Omega_D,A)$.
\end{proposition}

\begin{proof}
We first assume that $f\in\mathscr{C}^p(\Omega_D,A)$ and that there exists an open dense subset $D'$ of $D$ such that $f_{|_{\Omega_{D'}}}$ is a $T$-function. Fix $I\in\torus$: then \eqref{eq:representationformula} holds true for all $(\alpha,\beta)\in D'$ and all $J\in\torus$. In other words, for any $J\in\torus$ the map $\Psi^J:D\to A$ defined as
\[\Psi^J(\alpha,\beta):=f(\alpha+\beta\,J)-\sum_{H\in\mathscr{P}(\tau)}\gamma_H\,f(\alpha+\overline{\beta}^H\,I)\]
vanishes identically in $D'$. For any $J\in\torus$, our hypothesis $f\in\mathscr{C}^p(\Omega_D,A)$ guarantees that $\Psi^J\in\mathscr{C}^p(D,A)$, whence $\Psi^J\equiv0$ in $D$. In other words, \eqref{eq:representationformula} holds true for all $(\alpha,\beta)\in D$ and all $J\in\torus$. From Proposition~\ref{prop:representationformula}, we obtain that $f\in\slice_T(\Omega_D,A)$, i.e., that there exists $F\in\stem_T(D,A\otimes\rr^{2^\tau})$ such that $f=\I(F)$. Now, Theorem~\ref{thm:representationformula} yields that $F=\sum_{K\in\mathscr{P}(\tau)}E_KF_K$ with
\[F_K(\alpha,\beta)=2^{-\tau}I_K^{-1}\,\sum_{H\in\mathscr{P}(\tau)}(-1)^{|K \cap H|}\,f(\alpha+\overline{\beta}^H\,I)\,.\]
This formula, along with our hypothesis $f\in\mathscr{C}^p(\Omega_D,A)$, guarantees that $F\in\stem_T^p(D,A\otimes\rr^{2^\tau})$. We conclude that $f\in\slice_T^p(\Omega_D,A)$, as desired.

We now assume $p\in\{\infty,\omega\}$. By definition, $f$ belongs to $\slice_T^p(\Omega_D,A)$ if, and only if, $f=\I(F)$ for some $F\in\stem_T^p(D,A\otimes\rr^{2^\tau})=\stem_T(D,A\otimes\rr^{2^\tau})\cap\mathscr{C}^p(D,A\otimes\rr^{2^\tau})$. If this is the case and if $K=\{k_1,\ldots,k_p\}$, Remark~\ref{rmk:whitney} guarantees that
\[F_K(\alpha,\beta)=\beta_{k_1}\cdots\beta_{k_p}\,G_K(\alpha,\beta^2)\]
for appropriate $\mathscr{C}^p$ functions $\{G_K\}_{K\in\mathscr{P}(\tau)}$. The equality $f=\I(F)$ now implies that
\begin{align*}
f(\alpha+\beta\,J)&=\sum_{K\in\mathscr{P}(\tau)}J_KF_K(\alpha,\beta)\\
&=F_{\emptyset}(\alpha,\beta)+\sum_{1\leq p\leq\tau}\sum_{1\leq k_1<\ldots<k_p\leq\tau} J_{k_1}J_{k_2}\ldots J_{k_p}F_{\{k_1,\ldots,k_p\}}(\alpha,\beta)\\
&=G_{\emptyset}(\alpha,\beta^2)+\sum_{1\leq p\leq\tau}\sum_{1\leq k_1<\ldots<k_p\leq\tau} \beta_{k_1}J_{k_1}\beta_{k_2}J_{k_2}\ldots\beta_{k_p}J_{k_p}G_{\{k_1,\ldots,k_p\}}(\alpha,\beta^2)
\end{align*}
for all $(\alpha,\beta)\in D,J\in\torus$. Referring to the decomposition of the variable $x\in\Omega_D$ performed in Remark~\ref{rmk:decomposedvariable}, we conclude that
\begin{align*}
f(x)&=G_{\emptyset}(x^0,\Vert x^1\Vert^2,\ldots,\Vert x^\tau\Vert^2)\\
&\quad+\sum_{1\leq p\leq\tau}\sum_{1\leq k_1<\ldots<k_p\leq\tau} x^{k_1}x^{k_2}\ldots x^{k_p}G_{\{k_1,\ldots,k_p\}}(x^0,\Vert x^1\Vert^2,\ldots,\Vert x^\tau\Vert^2)\,.
\end{align*}
Since $V\to\rr_{0,t_0}\times\rr^\tau,\ x\mapsto(x^0,\Vert x^1\Vert^2,\ldots,\Vert x^\tau\Vert^2)$ is a real polynomial map, it is clear that each map
\[\Omega_D\to A,\quad x\mapsto G_K(x^0,\Vert x^1\Vert^2,\ldots,\Vert x^\tau\Vert^2)\]
belongs to $\mathscr{C}^p(\Omega_D,A)$. Since $V\to V\subseteq A,\ x=x^0+x^1+\ldots+x^\tau\mapsto x^{k}$ is a real linear map for all $k\in\{1,\ldots,\tau\}$, we conclude that $f\in\mathscr{C}^p(\Omega_D,A)$, as desired.
\end{proof}

We point out that the hypothesis $p\in\{\infty,\omega\}$ is crucial in the second statement of Proposition~\ref{prop:analyticTfunction}. When $p\in\nn$, there is some loss of regularity in the expression of $f=\I(F)$ in terms of $F\in\stem_T^p(D,A\otimes\rr^{2^\tau})$.

In~\cite[Proposition 7.1]{unifiedtheory}, we proved that all polynomials in $\{\T_\k\}_{\k\in\zz^{t_0+\tau}}$ are strongly $T$-regular and used this fact to obtain, for any $f\in\reg_T(\Omega,A)$ and any open ball $B\subseteq\Omega$ centered at a point $p$ in the mirror $\rr_{0,t_0}$, an explicit polynomial series expansion for $f$ converging normally in $B$, see~\cite[Theorem 7.2]{unifiedtheory}. Thanks to this fact,~\cite{unifiedtheory} proved that $T$-regular functions on $T$-symmetric $T$-slice domains are automatically strongly $T$-regular, whence real analytic. This property subsumes a renowned property of quaternionic slice-regular functions, proven in~\cite[Theorem 3.1]{advancesrevised} (see also~\cite{ghilonislicebyslice}), and the analogous property of Clifford slice-monogenic functions (see~\cite[Theorem 2.2.18]{librodaniele2} and references therein). The precise statement follows.

\begin{theorem}[Representation Formula, $T$-regular functions on $T$-symmetric $T$-slice domains]
If the $T$-symmetric set $\Omega_D$ is a $T$-slice domain, then $\reg_T(\Omega_D,A)=\sr_T(\Omega_D,A)$. As a consequence, every $f\in\reg_T(\Omega_D,A)$ is real analytic and fulfills formula~\eqref{eq:representationformula} for all $(\alpha,\beta)\in D$ and all $I,J\in\torus$.
\end{theorem}


\section{Global differential operators on $T$-functions}\label{sec:globaloperators}

Throughout this section, we fix $\tau\in\nn$ and $T=(t_0,\ldots,t_\tau)$ with $0\leq t_0<t_1<\ldots<t_\tau=N$. We are going to define global differential operators on $T$-functions to characterize $T$-regularity and $T$-harmonicity. In order to do so, we will first construct global differential operators on $T$-stem functions $D\to A\otimes\rr^{2^\tau}$.

\subsection{Global differential operators on $T$-stem functions}

Recall that every element of $D$ takes the form $(\alpha,\beta)$ with $\alpha=\sum_{s=0}^{t_0}v_sx_s\in\rr_{0,t_0}$ and $\beta=(\beta_1,\ldots,\beta_\tau)\in\rr^\tau$. In addition to the standard real linear isomorphism $L_{\B'}:\rr^{d+1}\to A$, we will therefore use
\[\Lambda:\rr^{t_0+\tau+1}\to\rr_{0,t_0}\times\rr^\tau\,,\quad \Lambda(x_0,\ldots,x_{t_0},x_{t_0+1},\ldots,x_{t_0+\tau})=\left(\sum_{s=0}^{t_0}x_s\,v_s,x_{t_0+1},\ldots,x_{t_0+\tau}\right)\,.\]
As a first step in our construction for $T$-stem functions, we pose the next definition, where we adopt the notation $\widehat{D}:=\Lambda^{-1}(D)$.

\begin{definition}\label{def:debaralpha}
Let $p\in\nn\cup\{\infty,\omega\}$ and $\Phi\in\mathscr{C}^{p+1}(D,A)$. For all $s\in\{0,\ldots,t_0\},h\in\{1,\ldots,\tau\}$, we define
\begin{align*}
\partial_{\alpha_s}\Phi&:=L_{\B'}\circ\left(\frac{\partial}{\partial x_s}\left(L_{\B'}^{-1}\circ \Phi\circ\Lambda_{|_{\widehat{D}}}\right)\right)\circ (\Lambda^{-1})_{|_D}\in\mathscr{C}^p(D,A)\\
\partial_{\beta_h}\Phi&:=L_{\B'}\circ\left(\frac{\partial}{\partial x_{t_0+h}}\left(L_{\B'}^{-1}\circ \Phi\circ\Lambda_{|_{\widehat{D}}}\right)\right)\circ (\Lambda^{-1})_{|_D}\in\mathscr{C}^p(D,A)\,.
\end{align*}
We also set
\begin{align*}
&\debar_\alpha\Phi:=\sum_{s=0}^{t_0}v_s\,\partial_{\alpha_s}\Phi=\partial_{\alpha_0}\Phi+\sum_{s=1}^{t_0}v_s\,\partial_{\alpha_s}\Phi\in\mathscr{C}^p(D,A)\,,\\
&\partial_\alpha\Phi:=\sum_{s=0}^{t_0}v_s^c\,\partial_{\alpha_s}\Phi=\partial_{\alpha_0}\Phi-\sum_{s=1}^{t_0}v_s\,\partial_{\alpha_s}\Phi\in\mathscr{C}^p(D,A)\,,\\
&\debar_\alpha^u\Phi:=\partial_{\alpha_0}\Phi-(-1)^u\sum_{s=1}^{t_0}v_s\,\partial_{\alpha_s}\Phi\in\mathscr{C}^p(D,A)
\end{align*}
for all $u\in\nn$.
\end{definition}

\begin{remark}\label{rmk:incrementalratio}
Let $\Phi\in\mathscr{C}^1(D,A)$. For all $(\alpha,\beta)=(\alpha,\beta_1,\ldots,\beta_\tau)\in D,s\in\{0,\ldots,t_0\},h\in\{1,\ldots,\tau\}$,
\begin{align*}
&\partial_{\alpha_s}\Phi(\alpha,\beta)=\lim_{\rr\ni \varepsilon\to0}\varepsilon^{-1}\left(\Phi(\alpha+\varepsilon v_s,\beta)-\Phi(\alpha,\beta)\right)\,,\\
&\partial_{\beta_h}\Phi(\alpha,\beta)=\lim_{\rr\ni \varepsilon\to0}\varepsilon^{-1}\left(\Phi(\alpha,\beta_1,\ldots,\beta_{h-1},\beta_{h}+\epsilon,\beta_{h+1},\ldots,\beta_\tau)-\Phi(\alpha,\beta)\right)\,.
\end{align*}
\end{remark}

As a consequence of Remark~\ref{rmk:incrementalratio}, the operators $\partial=\partial_{\alpha_s},\partial=\partial_{\beta_h}$ follow the usual Leibniz rule $\partial(\Phi\Psi)=(\partial\Phi)\,\Psi+\Phi\,(\partial\Psi)$ for all $\Phi,\Psi\in\mathscr{C}^1(D,A)$. Moreover, the usual rule $\partial^2(\Phi\Psi)=(\partial^2\Phi)\,\Psi+\Phi\,(\partial^2\Psi)+2(\partial\Phi)(\partial\Psi)$ applies both for $\partial=\partial_{\alpha_s}$ and for $\partial=\partial_{\beta_h}$ when $\Phi,\Psi\in\mathscr{C}^2(D,A)$.

\begin{remark}\label{rmk:dedebaralpha}
Since $v_s^cv_s=1=v_sv_s^c$ and $v_s^cv_u+v_u^cv_s=0=v_sv_u^c+v_uv_s^c$ for all distinct $s,u\in\{0,\ldots,t_0\}$, we notice that
\[\partial_\alpha\debar_\alpha=\sum_{s=0}^{t_0}\partial_{\alpha_s}^2=\debar_\alpha\partial_\alpha\,.\]
\end{remark}

We now construct the announced global differential operators on $T$-stem functions.

\begin{definition}\label{def:stemoperators}
The function $\sigma:\{1,\ldots,\tau\}\times\mathscr{P}(\tau)\to\{0,1\}$ is defined as follows: $\sigma(h,K)=0$ if there is an even number of elements in $K$ that are less than, or equal to, $h$; $\sigma(h,K)=1$ if there is an odd number of elements in $K$ that are less than, or equal to, $h$.

We define the operators $\debar_T,\partial_T:\stem_T^1(D,A\otimes\rr^{2^\tau})\to\stem_T^0(D,A\otimes\rr^{2^\tau})$ and $\Delta_T:\stem_T^2(D,A\otimes\rr^{2^\tau})\to\stem_T^0(D,A\otimes\rr^{2^\tau})$ as follows. Given $F=\sum_{K\in\mathscr{P}(\tau)}E_KF_K$, we set
\begin{align*}
&(\debar_TF)_K:=\debar_\alpha^{|K|+1}F_K+\sum_{h=1}^\tau(-1)^{\sigma(h,K)+1}\partial_{\beta_h}F_{K\bigtriangleup\{h\}}\,,\\
&(\partial_TF)_K:=\debar_\alpha^{|K|}F_K+\sum_{h=1}^\tau(-1)^{\sigma(h,K)}\partial_{\beta_h}F_{K\bigtriangleup\{h\}}\,,\\
&(\Delta_TF)_K:=\partial_\alpha\debar_\alpha F_K+\sum_{h=1}^{\tau}\partial_{\beta_h}^2F_K
\end{align*}
for all $K\in\mathscr{P}(\tau)$.
\end{definition}

The reason why $\debar_TF,\partial_TF,\Delta_TF$ are still $T$-stem functions is the following. Fix any $h\in\{1,\ldots,\tau\},K\in\mathscr{P}(\tau)$ and $F\in\stem_T^1(D,A\otimes\rr^{2^\tau})$. The functions $\debar_\alpha F_K,\partial_\alpha F_K$ are preserved by composition with the reflection $\rr^\tau\to\rr^\tau\,,\ \beta\mapsto\overline{\beta}^h$ if, and only if, $F_K$ is. Similarly, $F_K$ is preserved by composition with the same reflection if, and only if, $F_{K\bigtriangleup\{h\}}$ changes sign under composition with the same reflection, which is in turn equivalent to having $\partial_{\beta_h}F_{K\bigtriangleup\{h\}}$ preserved by composition with the same reflection.

It is easy to see that $\debar_T,\partial_T,\Delta_T$ all preserve $\stem_T^p(D,A\otimes\rr^{2^\tau})$ for $p\in\{\infty,\omega\}$. Moreover, we prove the next lemma.

\begin{lemma}
$\Delta_T=\partial_T\debar_T=\debar_T\partial_T$ on $\stem_T^2(D,A\otimes\rr^{2^\tau})$.
\end{lemma}

\begin{proof}
Fix $F\in\stem_T^2(D,A\otimes\rr^{2^\tau})$ and let $G:=\debar_TF$, so that
\[G_K=\debar_\alpha^{|K|+1}F_K+\sum_{h=1}^\tau(-1)^{\sigma(h,K)+1}\partial_{\beta_h}F_{K\bigtriangleup\{h\}}\]
for all $K\in\mathscr{P}(\tau)$. Now,
\begin{align*}
(\partial_T\debar_TF)_K&=(\partial_TG)_K=\debar_\alpha^{|K|}G_K+\sum_{h=1}^\tau(-1)^{\sigma(h,K)}\partial_{\beta_h}G_{K\bigtriangleup\{h\}}\\
&=\partial_\alpha\debar_\alpha F_K+\sum_{h=1}^\tau(-1)^{\sigma(h,K)+1}\,\debar_\alpha^{|K|}\partial_{\beta_h}F_{K\bigtriangleup\{h\}}\\
&\quad+\sum_{h=1}^\tau(-1)^{\sigma(h,K)}\partial_{\beta_h}\debar_\alpha^{|K\bigtriangleup\{h\}|+1}F_{K\bigtriangleup\{h\}}\\
&\quad+\sum_{h,\ell=1}^\tau(-1)^{\sigma(h,K)+\sigma(\ell,K\bigtriangleup\{h\})+1}\partial_{\beta_h}\partial_{\beta_\ell}F_{(K\bigtriangleup\{\ell\})\bigtriangleup\{h\}}\\
&=\partial_\alpha\debar_\alpha F_K+\sum_{h=1}^\tau\partial_{\beta_h}^2F_K\\
&\quad-\sum_{h<\ell}^\tau\left((-1)^{\sigma(h,K)+\sigma(\ell,K\bigtriangleup\{h\})}+(-1)^{\sigma(\ell,K)+\sigma(h,K\bigtriangleup\{\ell\})}\right)\partial_{\beta_h}\partial_{\beta_\ell}F_{(K\bigtriangleup\{\ell\})\bigtriangleup\{h\}}\\
&=\partial_\alpha\debar_\alpha F_K+\sum_{h=1}^\tau\partial_{\beta_h}^2F_K=(\Delta_TF)_K\,.
\end{align*}
For the third equality, we used the fact that $|K|$ and $|K|+1$ have opposite parities along with Remark~\ref{rmk:dedebaralpha}. In the fourth equality: the first two series canceled out because $|K\bigtriangleup\{h\}|+1$ has the same parity as $|K|$; and $\sigma(h,K)$ has the same parity as $\sigma(\ell,K\bigtriangleup\{h\})+1$ when $\ell=h$. For the fifth equality, we used the fact that $\sigma(h,K)=\sigma(h,K\bigtriangleup\{\ell\})$ and $\sigma(\ell,K\bigtriangleup\{h\})=\sigma(\ell,K)\pm1$ under the hypothesis $h<\ell$. Analogous computations prove that $\debar_T\partial_TF=\Delta_T$. 
\end{proof}

We conclude this subsection with remark about the kernel of $\Delta_T$.

\begin{remark}\label{rmk:Tstemharmonicareanalytic}
If $F$ belongs to the kernel of $\Delta_T:\stem_T^2(D,A\otimes\rr^{2^\tau})\to\stem_T^0(D,A\otimes\rr^{2^\tau})$, then $F\in\stem_T^\omega(D,A\otimes\rr^{2^\tau})$. This is because $\Delta_TF\equiv0$ implies that, for every $K\in\mathscr{P}(\tau)$, the component $F_K$ belongs to the kernel of the hypoelliptic operator $\sum_{s=0}^{t_0}\partial_{\alpha_s}^2+\sum_{h=1}^{\tau}\partial_{\beta_h}^2$, whence $F_K\in\mathscr{C}^\omega(D,A)$.
\end{remark}

\subsection{Global differential operators on $T$-functions}

We are now ready to define useful operators on $T$-functions, using the operators we just defined on $T$-stem functions. We will soon see that these operators are naturally connected to $T$-regularity and $T$-harmonicity. Additionally, the forthcoming Proposition~\ref{prop:globaloperators} will provide explicit expressions for these operators.

\begin{definition}
The operator $\debar_T:\slice_T^1(\Omega_D,A)\to\slice_T^0(\Omega_D,A)$ is defined to make the following diagram commutative:
\begin{center}
    \begin{tikzcd}
	{\stem_T^1(D,A\otimes\rr^{2^\tau})} && {\stem_T^0(D,A\otimes\rr^{2^\tau})} \\
	& \circlearrowleft \\
	{\slice_T^1(\Omega_D,A)} && {\slice_T^0(\Omega_D,A)\,.}
	\arrow["\debar_T", from=1-1, to=1-3]
	\arrow["{\I}"', from=1-1, to=3-1]
	\arrow["\debar_T", from=3-1, to=3-3]
	\arrow["{\I}", from=1-3, to=3-3]
\end{tikzcd}
\end{center}
The operators $\partial_T:\slice_T^1(\Omega_D,A)\to\slice_T^0(\Omega_D,A)$ and $\Delta_T:\slice_T^2(\Omega_D,A)\to\slice_T^0(\Omega_D,A)$ are similarly defined to lift to $\partial_T:\stem_T^1(D,A\otimes\rr^{2^\tau})\to\stem_T^0(D,A\otimes\rr^{2^\tau})$ and to $\Delta_T:\stem_T^2(D,A\otimes\rr^{2^\tau})\to\stem_T^0(D,A\otimes\rr^{2^\tau})$, respectively.
\end{definition}

It is easy to see that $\debar_T,\partial_T,\Delta_T$ all preserve $\slice_T^p(\Omega_D,A)$ for $p\in\{\infty,\omega\}$.

We are almost ready to characterize $T$-regularity and $T$-harmonicity in terms of $\debar_T$ and $\Delta_T$. The next remark, which uses Definition~\ref{def:stemoperators}, provides a useful tool for this characterization.

\begin{remark}\label{rmk:units}
Fix $h\in\{1,\ldots,\tau\}$ and $K\in\mathscr{P}(\tau)$. For any $J=(J_1,\ldots,J_\tau)\in\torus$,
\[J_hJ_K=(-1)^{\sigma(h,K)}J_{K\bigtriangleup\{h\}}\,.\]
\end{remark}

\begin{theorem}\label{thm:debarT}
If $f\in\slice_T^1(\Omega_D,A)$, then $(\debar_Tf)_J=\debar_Jf_J,(\partial_Tf)_J=\partial_Jf_J$ for every $J\in\torus$. If, moreover, $f\in\slice_T^2(\Omega_D,A)$, then $(\Delta_Tf)_J=\Delta_Jf_J$ for every $J\in\torus$.
\end{theorem}

\begin{proof}
Assume $f=\I(F)$, whence $f(\alpha+J\beta)=\sum_{K\in\mathscr{P}(\tau)}J_KF_K(\alpha,\beta)$ for all $(\alpha,\beta)\in D$ and all $J\in\torus$. Thus,
\begin{align*}
\debar_Jf_J(\alpha+J\beta)&=\sum_{K\in\mathscr{P}(\tau)}\left(\debar_\alpha+\sum_{h=1}^\tau J_h\partial_{\beta_h}\right)\left(J_KF_K(\alpha,\beta)\right)\\
&=\sum_{K\in\mathscr{P}(\tau)}J_K\debar_\alpha^{|K|+1} F_K(\alpha,\beta)+\sum_{K\in\mathscr{P}(\tau)}\sum_{h=1}^\tau J_hJ_K\partial_{\beta_h}F_K(\alpha,\beta)\,,\\
\partial_Jf_J(\alpha+J\beta)&=\sum_{K\in\mathscr{P}(\tau)}\left(\partial_\alpha-\sum_{h=1}^\tau J_h\partial_{\beta_h}\right)\left(J_KF_K(\alpha,\beta)\right)\\
&=\sum_{K\in\mathscr{P}(\tau)}J_K\debar_\alpha^{|K|} F_K(\alpha,\beta)-\sum_{K\in\mathscr{P}(\tau)}\sum_{h=1}^\tau J_hJ_K\partial_{\beta_h}F_K(\alpha,\beta)
\end{align*}
Now, Remarks~\ref{rmk:symmetricdifference} and~\ref{rmk:units} yield
\begin{align*}
\sum_{K\in\mathscr{P}(\tau)}\sum_{h=1}^\tau J_hJ_K\partial_{\beta_h}F_K(\alpha,\beta)&=\sum_{K\in\mathscr{P}(\tau)}\sum_{h=1}^\tau (-1)^{\sigma(h,K)}J_{K\bigtriangleup\{h\}}\partial_{\beta_h}F_K(\alpha,\beta)\\
&=\sum_{H\in\mathscr{P}(\tau)}J_{H}\sum_{h=1}^\tau (-1)^{\sigma(h,H\bigtriangleup\{h\})}\partial_{\beta_h}F_{H\bigtriangleup\{h\}}(\alpha,\beta)\\
&=\sum_{H\in\mathscr{P}(\tau)}J_{H}\sum_{h=1}^\tau (-1)^{\sigma(h,H)+1}\partial_{\beta_h}F_{H\bigtriangleup\{h\}}(\alpha,\beta)
\end{align*}
It follows that
\begin{align*}
\debar_Jf_J(\alpha+J\beta)&=\sum_{K\in\mathscr{P}(\tau)}J_K\left(\debar_\alpha^{|K|+1} F_K(\alpha,\beta)+\sum_{h=1}^\tau (-1)^{\sigma(h,K)+1}\partial_{\beta_h}F_{K\bigtriangleup\{h\}}(\alpha,\beta)\right)\\
&=\debar_Tf(\alpha+J\beta)\,,\\
\partial_Jf_J(\alpha+J\beta)&=\sum_{K\in\mathscr{P}(\tau)}J_K\left(\debar_\alpha^{|K|} F_K(\alpha,\beta)+\sum_{h=1}^\tau (-1)^{\sigma(h,K)}\partial_{\beta_h}F_{K\bigtriangleup\{h\}}(\alpha,\beta)\right)\\
&=\partial_Tf(\alpha+J\beta)\,.
\end{align*}
\end{proof}

\begin{corollary}
If $f\in\slice_T^1(\Omega_D,A)$, then $f$ is $T$-regular if, and only if, it belongs to the kernel of $\debar_T:\slice_T^1(\Omega_D,A)\to\slice_T^0(\Omega_D,A)$. Moreover, the right $A$-submodule $\sr_T(\Omega_D,A)$ of strongly $T$-regular functions is the image through the isomorphism $\I:\stem_T^1(D,A\otimes\rr^{2^\tau})\to\slice_T^1(\Omega_D,A)$ of the kernel of the operator $\debar_T:\stem_T^1(D,A\otimes\rr^{2^\tau})\to\stem_T^0(D,A\otimes\rr^{2^\tau})$.

If $f\in\slice_T^2(\Omega_D,A)$, then $f$ is $T$-harmonic if, and only if, it belongs to the kernel of $\Delta_T:\slice_T^2(\Omega_D,A)\to\slice_T^0(\Omega_D,A)$. Moreover, the right $A$-submodule of $\slice_T^2(\Omega_D,A)$ consisting of $T$-harmonic functions is the image through the isomorphism $\I:\stem_T^2(D,A\otimes\rr^{2^\tau})\to\slice_T^2(\Omega_D,A)$ of the kernel of the operator $\Delta_T:\stem_T^2(D,A\otimes\rr^{2^\tau})\to\stem_T^0(D,A\otimes\rr^{2^\tau})$.
\end{corollary}

The next remark extends Proposition~\ref{prop:analyticstronglyTregular} to all $T$-harmonic $T$-functions.

\begin{remark}\label{rmk:Tharmonicareanalytic}
As a consequence of Remark~\ref{rmk:Tstemharmonicareanalytic}, if $f$ belongs to the kernel of $\Delta_T:\slice_T^2(\Omega_D,A)\to\slice_T^0(\Omega_D,A)$, then $f\in\slice_T^\omega(\Omega_D,A)$. In this situation, Proposition~\ref{prop:analyticTfunction} implies that $f\in\mathscr{C}^\omega(\Omega_D,A)$.
\end{remark}

We conclude this section with the announced explicit expressions of $\debar_T,\partial_T,\Delta_T$. Some preparation is needed. In addition to the previously defined $L_{\B'}:\rr^{d+1}\to A$, we will use
\[L_{\B}:\rr^{N+1}\to V\,,\quad L_{\B}(x_0,\ldots,x_N)=\sum_{s=0}^{N}x_s\,v_s.\]
For every domain $\Omega$ in $V$ and every $p\in\nn$, the symbol $\mathscr{C}^p(\Omega,A)$ denotes the bilateral $A$-module of $\mathscr{C}^p$ functions $\Omega\to A$. In accordance with~\cite[Definition 2.36]{unifiedtheory}, we pose the next definition.

\begin{definition}\label{def:partialderivatives}
Fix a domain $\Omega$ in $V$, set $\Omega':=L_{\B}^{-1}(\Omega)\subseteq\rr^{N+1}$ and let $f\in\mathscr{C}^1(\Omega,A)$. For $s\in\{0,\ldots,N\}$, we define
\[\partial_{x_s}f:=L_{\B'}\circ\left(\frac{\partial}{\partial x_s}\left(L_{\B'}^{-1}\circ f\circ (L_{\B})_{|_{\Omega'}}\right)\right)\circ (L_{\B}^{-1})_{|_\Omega}\in\mathscr{C}^0(\Omega,A)\,.\]
\end{definition}

According to~\cite{unifiedtheory}, for any $x\in \Omega,s\in\{0,\ldots,N\}$,
\[\partial_{x_s}f(x)=\lim_{\rr\ni \varepsilon\to0}\varepsilon^{-1}\left(f(x+\varepsilon v_s)-f(x)\right)\,.\]
As a consequence, the operator $\partial_{x_s}$ does not depend on the whole basis $\B'$ of $A$ chosen, but only on the choice of $v_s$. Moreover, the usual Leibniz rule $\partial_{x_s}(fg)=(\partial_{x_s} f)\,g+f\,(\partial_{x_s}g)$ holds true for all $f,g\in\mathscr{C}^1(\Omega,A)$ and $\partial_{x_s}^2(fg)=(\partial_{x_s}^2f)\,g+f\,(\partial_{x_s}^2g)+2(\partial_{x_s}f)(\partial_{x_s}g)$ for all $f,g\in\mathscr{C}^2(\Omega,A)$.

We are now ready to explicitly express $\debar_T,\partial_T,\Delta_T$ and to provide examples.

\begin{proposition}\label{prop:globaloperators}
Let $D':=\{(\alpha,\beta_1,\ldots,\beta_\tau)\in D:\beta_1\cdots\beta_\tau\neq0\}$, so that $\Omega_{D'}$ equals $\Omega_D$ minus the union for $u\in\{1,\ldots,\tau\}$ of the real vector subspaces with Cartesian equations $x^u=0$. If $f\in\slice_T^1(\Omega_D,A)$, then
\begin{align*}
&\debar_Tf=\sum_{s=0}^{t_0}v_s\,\partial_{x_s}f+\sum_{u=1}^\tau \frac{x^u}{\Vert x^u\Vert^2}\sum_{s=t_{u-1}+1}^{t_u}x_s\,\partial_{x_s}f\,,\\
&\partial_Tf=\sum_{s=0}^{t_0}v_s^c\,\partial_{x_s}f+\sum_{u=1}^\tau\left(\frac{x^u}{\Vert x^u\Vert^2}\right)^c\sum_{s=t_{u-1}+1}^{t_u}x_s\,\partial_{x_s}f
\end{align*}
in $\Omega_{D'}$. If $f\in\slice_T^2(\Omega_D,A)$, then
\begin{align*}
&\Delta_Tf=\sum_{s=0}^{t_0}\partial_{x_s}^2f+\sum_{u=1}^\tau\sum_{s,s'=t_{u-1}+1}^{t_u}\frac{x_sx_{s'}}{\Vert x^u\Vert^2}\,\partial_{x_s}\partial_{x_{s'}}f
\end{align*}
in $\Omega_{D'}$.
\end{proposition}

\begin{proof}
Set $\Omega:=\Omega_{D'}$. We are going to prove that the desired equalities hold true in $\Omega_J$ for all $J\in\torus$. By Theorem~\ref{thm:debarT}, this is the same as proving that
\begin{align*}
&\debar_Jf_J=\sum_{s=0}^{t_0}v_s\,(\partial_{x_s}f)_J+\sum_{u=1}^\tau J_u\sum_{s=t_{u-1}+1}^{t_u}(J_u)_s\,(\partial_{x_s}f)_J\,,\\
&\partial_Jf_J=\sum_{s=0}^{t_0}v_s^c\,(\partial_{x_s}f)_J+\sum_{u=1}^\tau\left(J_u\right)^c\sum_{s=t_{u-1}+1}^{t_u}(J_u)_s\,(\partial_{x_s}f)_J\,,\\
&\Delta_Jf_J=\sum_{s=0}^{t_0}(\partial_{x_s}^2f)_J+\sum_{u=1}^\tau\sum_{s,s'=t_{u-1}+1}^{t_u}(J_u)_s(J_u)_{s'}\,(\partial_{x_s}\partial_{x_{s'}}f)_J
\end{align*}
in $\Omega_J$, for all $J\in\torus$. Here, we used the temporary notation $(J_u)_s$ for the $s$-th component of $J_u$. By Definition~\ref{def:Jmonogenic}, it suffices fix $J\in\torus,s\in\{0,\ldots,t_0\},u\in\{1,\ldots,\tau\}$ and to prove the following equalities: $\partial_sf_J=(\partial_{x_s}f)_J,\partial_s^2f_J=(\partial_{x_s}^2f)_J,\partial_{t_0+u}f_J=\sum_{s=t_{u-1}+1}^{t_u}(J_u)_s\,(\partial_{x_s}f)_J$ and $\partial_{t_0+u}^2f_J=\sum_{s,s'=t_{u-1}+1}^{t_u}(J_u)_s(J_u)_{s'}\,(\partial_{x_s}\partial_{x_{s'}}f)_J$. This is readily done, as follows:
\begin{align*}
(\partial_sf_J)(x)&=\lim_{\rr\ni \varepsilon\to0}\varepsilon^{-1}\left(f_J(x+\varepsilon v_s)-f_J(x)\right)=\lim_{\rr\ni \varepsilon\to0}\varepsilon^{-1}\left(f(x+\varepsilon v_s)-f(x)\right)=(\partial_{x_s}f)(x)\,,\\
(\partial_s^2f_J)(x)&=\lim_{\rr\ni \varepsilon\to0}\varepsilon^{-2}\left(f(x+\varepsilon v_s)-2f(x)+f(x-\varepsilon v_s)\right)=(\partial_{x_s}^2f)(x)\,,\\
(\partial_{t_0+u}f_J)(x)&=\lim_{\rr\ni \varepsilon\to0}\varepsilon^{-1}\left(f(x+\varepsilon J_u)-f(x)\right)=\sum_{s=t_{u-1}+1}^{t_u}(J_u)_s\,(\partial_{x_s}f)(x)\,,\\
(\partial_{t_0+u}^2f_J)(x)&=\lim_{\rr\ni \varepsilon\to0}\varepsilon^{-2}\left(f(x+\varepsilon J_u)-2f(x)+f(x-\varepsilon J_u)\right)\\
&=\sum_{s,s'=t_{u-1}+1}^{t_u}(J_u)_s(J_u)_{s'}\,(\partial_{x_s}\partial_{x_{s'}}f)(x)
\end{align*}
for all $x\in\Omega_J$.
\end{proof}

When $T=(0,N)$, the operators $\debar_T,\partial_T$ coincide with the operators $2\overline{\vartheta},2\vartheta$ introduced in~\cite{global} to characterize slice regular functions. When $\tau=1$, $A=C\ell(0,N)$ and $V=\rr^{N+1}$, the operator $\debar_T$ coincides with the operator $\overline{\vartheta}$ used in~\cite{xsgeneralizedpartialslice} to characterize generalized partial-slice monogenic functions. We now use Proposition~\ref{prop:globaloperators} to double-check $T$-regularity and $T$-harmonicity of the polynomial example of degree $3$ constructed in Example~\ref{ex:(0,3,6)}. We will use it again, throughout the paper, for computations concerning the polynomial examples constructed in Examples~\ref{ex:(0,3,6)} and~\ref{ex:(1,4,7)}.

\begin{example}\label{ex:(0,3,6)bis}
Let $A=C\ell(0,6),V=\rr^7$ and $T=(0,3,6)$. Let us consider the strongly $T$-regular polynomial
\[3\T_{(2,1)}(x)=x_0^3+3x_0^2x^2-3x_0\Vert x^1\Vert^2+6x_0x^1x^2-3\Vert x^1\Vert^2x^2\]
computed in Example~\ref{ex:(0,3,6)}, where $x^1=x_1e_1+x_2e_2+x_3e_3,x^2=x_4e_4+x_5e_5+x_6e_6$. If $x\in\rr^7$ has $x^1\neq0\neq x^2$, then
\begin{align*}
3\,\debar_T\T_{(2,1)}(x)&=3\,\partial_{x_0}\T_{(2,1)}(x)+\sum_{u=1}^2\frac{x^u}{\Vert x^u\Vert^2}\sum_{s=t_{u-1}+1}^{t_u}x_s\,3\,\partial_{x_s}\T_{(2,1)}(x)\\
&=3x_0^2+6x_0x^2-3\Vert x^1\Vert^2+6x^1x^2+\frac{x^1}{\Vert x^1\Vert^2}\sum_{s=1}^3x_s(-6x_0x_s+6x_0e_sx^2-6x_sx^2)\\
&\quad+\frac{x^2}{\Vert x^2\Vert^2}\sum_{s=4}^6x_s(3x_0^2e_s+6x_0x^1e_s-3\Vert x^1\Vert^2e_s)\\
&=3x_0^2+6x_0x^2-3\Vert x^1\Vert^2+6x^1x^2-6x_0x^1-6x_0x^2-6x^1x^2-3x_0^2+6x_0x^1+3\Vert x^1\Vert^2\\
&\equiv0
\end{align*}
and
\begin{align*}
3\,\Delta_T\T_{(2,1)}(x)&=3\,\partial_{x_0}^2\T_{(2,1)}(x)+\sum_{u=1}^\tau\sum_{s,s'=t_{u-1}+1}^{t_u}\frac{x_sx_{s'}}{\Vert x^u\Vert^2}\,3\,\partial_{x_s}\partial_{x_{s'}}\T_{(2,1)}(x)\\
&=6x_0+6x^2+\sum_{s=1}^3\frac{x_s^2}{\Vert x^1\Vert^2}\,(-6x_0-6x^2)+0=6x_0+6x^2-6x_0-6x^2\\
&\equiv0\,,
\end{align*}
as expected.
\end{example}


\section{Natural inclusion of $T$-functions among $\widetilde{T}$-functions}\label{sec:naturalinclusion}

For $\tau\in\nn$ and $T=(t_0,t_1,\ldots,t_\tau)$ with $0\leq t_0<t_1<\ldots<t_\tau=N$, we set the notation
\[\widetilde{T}:=(t_1,\ldots,t_\tau)\]
and notice that $\widetilde{T}$ is a list of $\tau-1$ steps. We recall Assumptions~\ref{ass:domain} and~\ref{ass:opendomain} and we set the notations
\begin{align}\label{eq:tildeD}
\widetilde{D}&:=\left\{(x^0+x^1,\beta_2,\ldots,\beta_\tau):x^1\in\rr_{t_0+1,t_1},\,(x^0,\Vert x^1\Vert,\beta_2,\ldots,\beta_\tau)\in D\right\}\subseteq\rr_{0,t_1}\times\rr^{\tau-1}\,,\\
\widetilde{D}_*&:=\left\{(x^0+x^1,\beta_2,\ldots,\beta_\tau):x^1\in\rr_{t_0+1,t_1}\setminus\{0\},\,(x^0,\Vert x^1\Vert,\beta_2,\ldots,\beta_\tau)\in D\right\}\subseteq\rr_{0,t_1}\times\rr^{\tau-1}\,.\notag
\end{align}
The domain of $\widetilde{\Omega}_{\widetilde{D}}\subseteq V$ of $\widetilde{T}$-functions induced by the elements of $\stem_{\widetilde{T}}(\widetilde{D},A\otimes\rr^{2^{\tau-1}})$ is the same as the domain $\Omega_D\subseteq V$ of $T$-functions induced by the elements of $\stem_T(D,A\otimes\rr^{2^\tau})$ because
\begin{align*}
\Omega_D&=\{\alpha+\beta\,J: (\alpha,\beta)\in D, J\in\torus=\s_{t_0+1,t_1}\times\widetilde{\torus}\}\\
&=\{\alpha+\beta_1J_1+\widetilde{\beta}\,\widetilde{J}: (\alpha,\beta_1,\widetilde{\beta})\in D, J_1\in\s_{t_0+1,t_1},\widetilde{J}\in\widetilde{\torus}\}\\
&=\left\{\alpha+x^1+\widetilde{\beta}\,\widetilde{J}:x^1\in\rr_{t_0+1,t_1},(\alpha,\Vert x^1\Vert,\widetilde{\beta})\in D, \widetilde{J}\in\widetilde{\torus}\right\}\\
&=\left\{\widetilde{\alpha}+\widetilde{\beta}\,\widetilde{J}: (\widetilde{\alpha},\widetilde{\beta})\in\widetilde{D}, \widetilde{J}\in\widetilde{\torus}\right\}=\widetilde{\Omega}_{\widetilde{D}}\,.
\end{align*}
This section focuses on the effect on $\slice_T^\infty(\Omega_D,A)$ (or $\slice_T^\omega(\Omega_D,A)$) of changing $T$ into $\widetilde{T}$.

We begin with a remark and a definition.

\begin{remark}\label{rmk:tilde}
$\mathscr{P}(\tau)$ is the disjoint union between the image of the injective map $\mathscr{P}(\tau-1)\to\mathscr{P}(\tau)\,,\ H\mapsto H+1$ and the image of the injective map $\mathscr{P}(\tau-1)\to\mathscr{P}(\tau)\,,\ H\mapsto\{1\}\cup(H+1)$.
\end{remark}

\begin{definition}\label{def:tilde}
For $p\in\{\infty,\omega\}$, the right $A$-module morphism
\[\widetilde{\ }:\stem_T^p(D,A\otimes\rr^{2^\tau})\to\mathscr{C}^p(\widetilde{D},A\otimes\rr^{2^{\tau-1}})\]
is defined to associate to each $F=\sum_{K\in\mathscr{P}(\tau)}E_KF_K$ the unique element $\widetilde{F}=\sum_{H\in\mathscr{P}(\tau-1)}\widetilde{E}_H\widetilde{F}_H$ of $\mathscr{C}^p(\widetilde{D},A\otimes\rr^{2^{\tau-1}})$ such that
\begin{align}\label{eq:tilde}
\widetilde{F}_H(\alpha+x^1,\widetilde{\beta})=F_{H+1}(\alpha,\Vert x^1\Vert,\widetilde{\beta})+(-1)^{|H|}\frac{x^1}{\Vert x^1\Vert}F_{\{1\}\cup(H+1)}(\alpha,\Vert x^1\Vert,\widetilde{\beta})
\end{align}
for all $(\alpha+x^1,\widetilde{\beta})\in\widetilde{D}_*$.
\end{definition}

We point out that, while formula~\eqref{eq:tilde} only defines an element $\widetilde{F}$ of $\mathscr{C}^p(\widetilde{D}_*,A\otimes\rr^{2^{\tau-1}})$, Remark~\ref{rmk:whitney} guarantees that $\widetilde{F}$ extends to a unique element of $\mathscr{C}^p(\widetilde{D},A\otimes\rr^{2^{\tau-1}})$.

We are now ready to prove that every element $f=\I(F)$ of $\slice_T^p(\Omega_D,A)$ is automatically a $\widetilde{T}$-function, induced by $\widetilde{F}$. In the process, we will also prove that $\widetilde{F}$ is a $\widetilde{T}$-stem function.

\begin{proposition}\label{TfunctionsareTtildefucntions}
Fix $p\in\{\infty,\omega\}$. The image of $\stem_T^p(D,A\otimes\rr^{2^\tau})$ through $\widetilde{\ }$ is contained in $\stem_{\widetilde{T}}^p(\widetilde{D},A\otimes\rr^{2^{\tau-1}})$. Moreover, $\slice_T^p(\Omega_D,A)$ is contained in $\slice_{\widetilde{T}}^p(\Omega_D,A)$ and the following diagram commutes:
\begin{center}
    \begin{tikzcd}
	{\stem_T^p(D,A\otimes\rr^{2^\tau})} && {\stem_{\widetilde{T}}^p(\widetilde{D},A\otimes\rr^{2^{\tau-1}})} \\
	& \circlearrowleft \\
	{\slice_T^p(\Omega_D,A)} && {\slice_{\widetilde{T}}^p(\Omega_D,A)\,.}
	\arrow["\widetilde{\ }", from=1-1, to=1-3]
	\arrow["{\I}"', from=1-1, to=3-1]
	\arrow[hookrightarrow, from=3-1, to=3-3]
	\arrow["{\I}", from=1-3, to=3-3]
\end{tikzcd}
\end{center}
\end{proposition}

\begin{proof}
Fix $f=\I(F)\in\slice_T^p(\Omega_D,A)$, whence $f\in\mathscr{C}^p(\Omega_D,A)$ by Proposition~\ref{prop:analyticTfunction}. For our computations, let us also fix $\alpha\in\rr_{0,t_0},x^1\in\rr_{t_0+1,t_1},\widetilde{\beta}=(\beta_2,\ldots,\beta_\tau)\in\rr^{\tau-1}$ such that $(\alpha+x^1,\widetilde{\beta})\in \widetilde{D}_*$ and let us set $\widetilde{\alpha}:=\alpha+x^1$. Fix $\widetilde{J}\in\widetilde{\torus}$ and remark that $\widetilde{\alpha}+\widetilde{J}\,\widetilde{\beta}=\alpha+x^1+\widetilde{J}\,\widetilde{\beta}=\alpha+J\beta$ with $J:=\left(\frac{x^1}{\Vert x^1\Vert},\widetilde{J}\right)\in\s_{t_0+1,t_1}\times\widetilde{\torus}=\torus$ and $\beta:=(\Vert x^1\Vert,\widetilde{\beta})\in\rr^\tau$.

Our first task is proving that the restriction of $\widetilde{F}$ to $\widetilde{D}_*$, defined by formula~\eqref{eq:tilde}, is an element of $\stem_{\widetilde{T}}^p(\widetilde{D}_*,A\otimes\rr^{2^{\tau-1}})$. This can be done as follows. If $h\in H\in\mathscr{P}(\tau-1)$, then $h+1\in H+1\subset\{1\}\cup(H+1)$, whence
\begin{align*}
&\widetilde{F}_H(\alpha+x^1,\beta_2,\ldots,-\beta_{h+1},\ldots,\beta_\tau)=F_{H+1}(\alpha,\Vert x^1\Vert,\beta_2,\ldots,-\beta_{h+1},\ldots,\beta_\tau)\\
&\quad+(-1)^{|H|}\frac{x^1}{\Vert x^1\Vert}F_{\{1\}\cup(H+1)}(\alpha,\Vert x^1\Vert,\beta_2,\ldots,-\beta_{h+1},\ldots,\beta_\tau)\\
&=-F_{H+1}(\alpha,\Vert x^1\Vert,\beta_2,\ldots,\beta_{h+1},\ldots,\beta_\tau)\\
&\quad-(-1)^{|H|}\frac{x^1}{\Vert x^1\Vert}F_{\{1\}\cup(H+1)}(\alpha,\Vert x^1\Vert,\beta_2,\ldots,\beta_{h+1},\ldots,\beta_\tau)\\
&=-\widetilde{F}_H(\alpha+x^1,\beta_2,\ldots,\beta_{h+1},\ldots,\beta_\tau)\,.
\end{align*}
If, instead, $h\not\in H\in\mathscr{P}(\tau-1)$, then $h+1\not\in\{1\}\cup(H+1)\supset H+1$, whence
\begin{align*}
&\widetilde{F}_H(\alpha+x^1,\beta_2,\ldots,-\beta_{h+1},\ldots,\beta_\tau)=F_{H+1}(\alpha,\Vert x^1\Vert,\beta_2,\ldots,-\beta_{h+1},\ldots,\beta_\tau)\\
&\quad+(-1)^{|H|}\frac{x^1}{\Vert x^1\Vert}F_{\{1\}\cup(H+1)}(\alpha,\Vert x^1\Vert,\beta_2,\ldots,-\beta_{h+1},\ldots,\beta_\tau)\\
&=F_{H+1}(\alpha,\Vert x^1\Vert,\beta_2,\ldots,\beta_{h+1},\ldots,\beta_\tau)\\
&\quad+(-1)^{|H|}\frac{x^1}{\Vert x^1\Vert}F_{\{1\}\cup(H+1)}(\alpha,\Vert x^1\Vert,\beta_2,\ldots,\beta_{h+1},\ldots,\beta_\tau)\\
&=\widetilde{F}_H(\alpha+x^1,\beta_2,\ldots,\beta_{h+1},\ldots,\beta_\tau)\,.
\end{align*}

Our second task is proving that $f$ coincides in $\widetilde{\Omega}_{\widetilde{D}_*}$ with the $\widetilde{T}$-function induced by $\widetilde{F}_{|_{\widetilde{D}_*}}$. Since $f=\I(F)$, we have
\begin{align*}
&f(\widetilde{\alpha}+\widetilde{J}\,\widetilde{\beta})=f(\alpha+J\beta)=\sum_{K\in\mathscr{P}(\tau)}J_KF_K(\alpha,\beta)=\sum_{K\in\mathscr{P}(\tau)}J_KF_K(\alpha,\Vert x^1\Vert,\widetilde{\beta})\\
&=\sum_{H\in\mathscr{P}(\tau-1)}J_{H+1}F_{H+1}(\alpha,\Vert x^1\Vert,\widetilde{\beta})+\sum_{H\in\mathscr{P}(\tau-1)}J_{\{1\}\cup(H+1)}F_{\{1\}\cup(H+1)}(\alpha,\Vert x^1\Vert,\widetilde{\beta})\\
&=\sum_{H\in\mathscr{P}(\tau-1)}\widetilde{J}_H\left(F_{H+1}(\alpha,\Vert x^1\Vert,\widetilde{\beta})+(-1)^{|H|}\frac{x^1}{\Vert x^1\Vert}F_{\{1\}\cup(H+1)}(\alpha,\Vert x^1\Vert,\widetilde{\beta})\right)\\
&=\sum_{H\in\mathscr{P}(\tau-1)}\widetilde{J}_H\widetilde{F}_H(\alpha+x^1,\widetilde{\beta})=\sum_{H\in\mathscr{P}(\tau-1)}\widetilde{J}_H\widetilde{F}_H(\widetilde{\alpha},\widetilde{\beta})\,.
\end{align*}
For the third equality, we used the equality $\beta=(\Vert x^1\Vert,\widetilde{\beta})$. The fourth equality follows from Remark~\ref{rmk:tilde}. For the fifth equality, we used the indentities $J_{H+1}=\widetilde{J}_H$ and $J_{\{1\}\cup(H+1)}=\frac{x^1}{\Vert x^1\Vert}J_{H+1}=(-1)^{|H|}J_{H+1}\frac{x^1}{\Vert x^1\Vert}=\widetilde{J}_H(-1)^{|H|}\frac{x^1}{\Vert x^1\Vert}$, which are consequences of the equality $J=\left(\frac{x^1}{\Vert x^1\Vert},\widetilde{J}\right)$. The sixth equality follows from formula~\eqref{eq:tilde}. The seventh equality is based on the definition of $\widetilde{\alpha}$ as $\alpha+x^1$. Since $\widetilde{J}$ is an arbitrary element of $\widetilde{\torus}$ and $(\widetilde{\alpha},\widetilde{\beta})$ is an arbitrary element of $\widetilde{D}_*$, we have indeed proven that $f$ coincides in $\widetilde{\Omega}_{\widetilde{D}_*}$ with the $\widetilde{T}$-function induced by $\widetilde{F}_{|_{\widetilde{D}_*}}$.

We now conclude the proof with the following argument. We have $\widetilde{\Omega}_{\widetilde{D}_*}=\Omega_{D'}$, where $D':=\{(\alpha,\beta_1,\ldots,\beta_\tau)\in D:\beta_1\neq0\}$ is an open dense subset of $D$. Since $f\in\mathscr{C}^p(\Omega_D,A)$ is a $\widetilde{T}$-function when restricted to $\widetilde{\Omega}_{\widetilde{D}_*}=\Omega_{D'}$, Proposition~\ref{prop:analyticTfunction} guarantees that $f\in\slice_{\widetilde{T}}^p(\Omega_D,A)$. The element of $\stem_{\widetilde{T}}^p(\widetilde{D},A\otimes\rr^{2^{\tau-1}})$ inducing $f$ coincides with $\widetilde{F}$ in $\widetilde{D}_*$, whence in $\widetilde{D}$ by continuity.
\end{proof}

We provide an explicit example.

\begin{example}\label{ex:(0,3,6)ter}
Let $A=C\ell(0,6),V=\rr^7$ and $T=(0,3,6)$, whence $\widetilde{T}=(3,6)$. We already computed in Example~\ref{ex:(0,3,6)} the strongly $T$-regular polynomial $3\T_{(2,1)}(x)=x_0^3+3x_0^2x^2-3x_0\Vert x^1\Vert^2+6x_0x^1x^2-3\Vert x^1\Vert^2x^2$. The computations
\begin{align*}
3\T_{(2,1)}(\alpha+J_1\beta_1+J_2\beta_2)&=\alpha^3-3\alpha\beta_1^2+J_2(3\alpha^2-3\beta_1^2)\beta_2+J_1J_2(6\alpha\beta_1\beta_2)\,,\\
3\T_{(2,1)}(\alpha+x^1+J_2\beta_2)&=\alpha^3-3\alpha\Vert x^1\Vert^2+J_2(3\alpha^2-6\alpha x^1-3\Vert x^1\Vert^2)\beta_2\,.
\end{align*}
show that the $T$-stem function inducing $3\T_{(2,1)}$ is $F(\alpha,\beta_1,\beta_2)=E_\emptyset(\alpha^3-3\alpha\beta_1^2)+E_{\{2\}}(3\alpha^2-3\beta_1^2)\beta_2+E_{\{1,2\}}(6\alpha\beta_1\beta_2)$ and that $3\T_{(2,1)}$ is also a $\widetilde{T}$-function, induced by the $\widetilde{T}$-stem function $\widetilde{F}$ with $\widetilde{F}(\alpha+x^1,\beta_2)=\widetilde{E}_\emptyset(\alpha^3-3\alpha\Vert x^1\Vert^2)+\widetilde{E}_{\{1\}}(3\alpha^2-6\alpha x^1-3\Vert x^1\Vert^2)\beta_2$.

We point out that, while $3\T_{(2,1)}$ is a $\widetilde{T}$-function, it is neither $\widetilde{T}$-regular nor $\widetilde{T}$-harmonic. Indeed, taking into account that $\widetilde{x}^0=x_0+x^1=\sum_{s=0}^3x_sv_s,\widetilde{x}^1=x^2=\sum_{s=4}^6x_sv_s$, we get
\begin{align*}
3\,\Delta_{\widetilde{T}}\T_{(2,1)}(x)&=\sum_{s=0}^33\,\partial_{x_s}^2\T_{(2,1)}(x)+\sum_{s,s'=4}^{6}\frac{x_sx_{s'}}{x_4^2+x_5^2+x_6^2}\,3\,\partial_{x_s}\partial_{x_{s'}}\T_{(2,1)}(x)\\
&=6x_0+6x^2+\sum_{s=1}^3(-6x_0-6x^2)+0=-12x_0-12x^2
\end{align*}
for all $x\in\rr^7$ with $x^2\neq0$, whence $3\,\Delta_{\widetilde{T}}\T_{(2,1)}(x)=-12x_0-12x^2$ throughout $\rr^7$.
\end{example}

For future use, we now wish to compute the iterates of $\ \widetilde{\ }$. To this end, we set the following notations for our fixed $\tau\in\nn$, $T=(t_0,t_1,\ldots,t_\tau)$ with $0\leq t_0<t_1<\ldots<t_\tau=N$, and $D\subseteq\rr_{0,t_0}\times\rr^\tau$.

\begin{definition}\label{def:multitilde}
Fix $p\in\{\infty,\omega\}$. We define
\[T_1:=\widetilde{T}=(t_1,\ldots,t_\tau),\ T_2:=\widetilde{T}_1=(t_2,\ldots,t_\tau),\ \ldots,\ T_\tau:=\widetilde{T}_{\tau-1}=(t_\tau)=(N)\,,\]
as well as
\[D^{1}:=\widetilde{D},\ D^{1}_*:=\widetilde{D}_*,\ D^{2}:=\widetilde{D^{1}},\ D^{2}_*:=\widetilde{(D^{1}_*)}_*,\ \ldots,\ D^{\tau}:=\widetilde{D^{\tau-1}},\ D^{\tau}_*:=\widetilde{(D^{\tau-1}_*)}_*\,.\]
For every $F\in\stem_T^p(D,A\otimes\rr^{2^\tau})$, we set
\[F^{1}:=\widetilde{F},\ F^{2}:=\widetilde{F^{1}},\ \ldots,\ F^{\tau}:=\widetilde{F^{\tau-1}}\,.\]
\end{definition}

Let us now compute $F^\sigma$ in terms of $F$, thus expressing $f$ as a $T_\sigma$-function.

\begin{proposition}\label{prop:multitilde}
Fix $p\in\{\infty,\omega\}$, let $f=\I(F)\in\slice_T^p(\Omega_D,A)$ and $\sigma\in\{1,\ldots,\tau\}$. Then $f$ is the element of $\slice_{T_\sigma}^p(\Omega_D,A)$, induced by $F^\sigma$. Moreover, for $H\in\mathscr{P}(\tau-\sigma)$ fixed,
\begin{align*}
&F^\sigma_H(\alpha+x^1+\ldots+x^\sigma,\beta_{\sigma+1},\ldots,\beta_{\tau})=F_{H+\sigma}(\alpha,\Vert x^1\Vert,\ldots,\Vert x^\sigma\Vert,\beta_{\sigma+1},\ldots,\beta_{\tau})\\
&\quad+\sum_{p=1}^\sigma\sum_{1\leq k_1<\ldots<k_p\leq \sigma}(-1)^{p|H|}\frac{x^{k_1}}{\Vert x^{k_1}\Vert}\ldots\frac{x^{k_p}}{\Vert x^{k_p}\Vert}F_{\{k_1,\ldots,k_p\}\cup(H+\sigma)}(\alpha,\Vert x^1\Vert,\ldots,\Vert x^\sigma\Vert,\beta_{\sigma+1},\ldots,\beta_{\tau})\\
&=\sum_{p=0}^\sigma\sum_{1\leq k_1<\ldots<k_p\leq \sigma}(-1)^{p|H|}\frac{x^{k_1}}{\Vert x^{k_1}\Vert}\ldots\frac{x^{k_p}}{\Vert x^{k_p}\Vert}F_{\{k_1,\ldots,k_p\}\cup(H+\sigma)}(\alpha,\Vert x^1\Vert,\ldots,\Vert x^\sigma\Vert,\beta_{\sigma+1},\ldots,\beta_{\tau})
\end{align*}
for all $\alpha\in\rr_{0,t_0},\,x^1\in\rr_{t_0+1,t_1},\,\ldots,\,x^\sigma\in\rr_{t_{\sigma-1}+1,t_\sigma},\,(\beta_{\sigma+1},\ldots,\beta_{\tau})\in\rr^{\tau-\sigma}$ chosen so that $(\alpha,\Vert x^1\Vert,\ldots,\Vert x^\sigma\Vert,\beta_{\sigma+1},\ldots,\beta_{\tau})\in D^\sigma_*$.
\end{proposition}

\begin{proof}
The first statement follows immediately by repeated applications of Proposition~\ref{TfunctionsareTtildefucntions}. To prove the second statement, we proceed by induction on $\sigma\in\{1,\ldots,\tau\}$. For $\sigma=1$, it suffices to recall that $F^1=\widetilde{F}$ and to apply Definition~\ref{def:tilde}. Let us now assume the second statement true for $\sigma$ and prove it for $\sigma+1$. Taking into account that $F^{\sigma+1}=\widetilde{F^\sigma}$ and applying Definition~\ref{def:tilde}, we get
\begin{align*}
&F^{\sigma+1}_H(\alpha+x^1+\ldots+x^{\sigma+1},\beta_{\sigma+2},\ldots,\beta_{\tau})=
F^{\sigma}_{H+1}(\alpha+x^1+\ldots+x^\sigma,\Vert x^{\sigma+1}\Vert,\beta_{\sigma+2},\ldots,\beta_{\tau})\\
&\quad+(-1)^{|H|}\frac{x^{\sigma+1}}{\Vert x^{\sigma+1}\Vert}F^\sigma_{\{1\}\cup(H+1)}(\alpha+x^1+\ldots+x^\sigma,\Vert x^{\sigma+1}\Vert,\beta_{\sigma+2},\ldots,\beta_{\tau})\\
&=\sum_{p=0}^\sigma\sum_{1\leq k_1<\ldots<k_p\leq \sigma}(-1)^{p|H|}\frac{x^{k_1}}{\Vert x^{k_1}\Vert}\ldots\frac{x^{k_p}}{\Vert x^{k_p}\Vert}\cdot\\
&\quad\cdot F_{\{k_1,\ldots,k_p\}\cup(H+\sigma+1)}(\alpha,\Vert x^1\Vert,\ldots,\Vert x^\sigma\Vert,\Vert x^{\sigma+1}\Vert,\beta_{\sigma+2},\ldots,\beta_{\tau})\\
&\quad+(-1)^{|H|}\frac{x^{\sigma+1}}{\Vert x^{\sigma+1}\Vert}\sum_{p=0}^\sigma\sum_{1\leq k_1<\ldots<k_p\leq \sigma}(-1)^{p|H|+p}\frac{x^{k_1}}{\Vert x^{k_1}\Vert}\ldots\frac{x^{k_p}}{\Vert x^{k_p}\Vert}\cdot\\
&\quad\cdot F_{\{k_1,\ldots,k_p,\sigma+1\}\cup(H+\sigma+1)}(\alpha,\Vert x^1\Vert,\ldots,\Vert x^\sigma\Vert,\Vert x^{\sigma+1}\Vert,\beta_{\sigma+2},\ldots,\beta_{\tau})\\
&=\sum_{p=0}^\sigma\sum_{1\leq k_1<\ldots<k_p\leq \sigma}(-1)^{p|H|}\frac{x^{k_1}}{\Vert x^{k_1}\Vert}\ldots\frac{x^{k_p}}{\Vert x^{k_p}\Vert}\cdot\\
&\quad\cdot F_{\{k_1,\ldots,k_p\}\cup(H+\sigma+1)}(\alpha,\Vert x^1\Vert,\ldots,\Vert x^\sigma\Vert,\Vert x^{\sigma+1}\Vert,\beta_{\sigma+2},\ldots,\beta_{\tau})\\
&\quad+\sum_{p=0}^\sigma\sum_{1\leq k_1<\ldots<k_p\leq \sigma}(-1)^{(p+1)|H|}\frac{x^{k_1}}{\Vert x^{k_1}\Vert}\ldots\frac{x^{k_p}}{\Vert x^{k_p}\Vert}\frac{x^{\sigma+1}}{\Vert x^{\sigma+1}\Vert}\cdot\\
&\quad\cdot F_{\{k_1,\ldots,k_p,\sigma+1\}\cup(H+\sigma+1)}(\alpha,\Vert x^1\Vert,\ldots,\Vert x^\sigma\Vert,\Vert x^{\sigma+1}\Vert,\beta_{\sigma+2},\ldots,\beta_{\tau})\\
&=\sum_{p=0}^{\sigma+1}\sum_{1\leq k_1<\ldots<k_p\leq \sigma+1}(-1)^{p|H|}\frac{x^{k_1}}{\Vert x^{k_1}\Vert}\ldots\frac{x^{k_p}}{\Vert x^{k_p}\Vert}\cdot\\
&\quad\cdot F_{\{k_1,\ldots,k_p\}\cup(H+\sigma+1)}(\alpha,\Vert x^1\Vert,\ldots,\Vert x^{\sigma+1}\Vert,\beta_{\sigma+2},\ldots,\beta_{\tau})\,,
\end{align*}
as desired. For the second equality, we took into account that $|H+1|=|H|$, that $|\{1\}\cup(H+1)|=|H|+1$ and that $(\{1\}\cup(H+1))+\sigma=\{\sigma+1\}\cup(H+\sigma+1)$. For the third equality, we took into account that $x^{\sigma+1}$ anticommutes with each of $x^{k_1},\ldots,x^{k_p}$. The inductive step and the proof are now concluded.
\end{proof}

\begin{example}\label{ex:(0,3,6)ter2}
Let $A=C\ell(0,6),V=\rr^7$ and $T=(0,3,6)$, whence $T_1=\widetilde{T}=(3,6),T_2=\widetilde{T_1}=(6)$. The strongly $T$-regular polynomial $3\T_{(2,1)}(x)=x_0^3+3x_0^2x^2-3x_0\Vert x^1\Vert^2+6x_0x^1x^2-3\Vert x^1\Vert^2x^2$ of Example~\ref{ex:(0,3,6)} is a $T_1$-function and, trivially, a $T_2$-function. We point out that $3\T_{(2,1)}$ is neither $T_2$-regular nor $T_2$-harmonic. Indeed,
\begin{align*}
3\,\Delta_{T_2}\T_{(2,1)}(x)&=\sum_{s=0}^6\,\partial_{x_s}^2\T_{(2,1)}(x)=6x_0+6x^2+\sum_{s=1}^3(-6x_0-6x^2)+0=-12x_0-12x^2
\end{align*}
for all $x\in\rr^7$. In other words, $3\T_{(2,1)}$ is neither a monogenic nor a harmonic function $\rr^7\to C\ell(0,6)$.
\end{example}


\section{A variety of Laplacians of $T$-functions}\label{sec:tildeTlaplacians}

Let us fix $p\in\{\infty,\omega\},\sigma\in\{1,\ldots,\tau\}$ and $f\in\slice_T^p(\Omega_D,A)$. If we refer to the notations set in Definition~\ref{def:multitilde}, we know from Propositions~\ref{TfunctionsareTtildefucntions} and~\ref{prop:multitilde} that $f$ also belongs to $\slice_{\widetilde{T}}^p(\Omega_D,A)=\slice_{T_1}^p(\Omega_D,A)$ and to $\slice_{T_\sigma}^p(\Omega_D,A)$. The present section computes the effect on $f$ of the operators $\Delta_{T_\sigma}:\slice_{T_\sigma}^2(\Omega_D,A)\to\slice_{T_\sigma}^0(\Omega_D,A)$ and $\debar_{T_1}=\debar_{\widetilde{T}}:\slice_{\widetilde{T}}^1(\Omega_D,A)\to\slice_{\widetilde{T}}^0(\Omega_D,A)$. In both cases, we will prove that the result is still an element of $\slice_T^p(\Omega_D,A)$. While these intermediate results are quite technical, they provide fundamental tools to later prove our main theorems.

Studying the operator $\Delta_{T_\sigma}:\slice_{T_\sigma}^p(\Omega_D,A)\to\slice_{T_\sigma}^p(\Omega_D,A)$ is equivalent to studying the operator $\Delta_{T_\sigma}:\stem_{T_\sigma}^p(D^\sigma,A\otimes\rr^{2^{\tau-\sigma}})\to\stem_{T_\sigma}^p(D^\sigma,A\otimes\rr^{2^{\tau-\sigma}})$ precomposed with the $\sigma$-th iterate of $\ \widetilde{\ }$. In other words: for $f=\I(F)\in\slice_T^p(\Omega_D,A)$, the $T_\sigma$-stem function inducing $\Delta_{T_\sigma}f$ is $\Delta_{T_\sigma}F^\sigma$, where $F^\sigma$ is explicitly computed in Proposition~\ref{prop:multitilde}. To compute $\Delta_{T_\sigma}F^\sigma$, the next result will be handy.

\begin{proposition}\label{prop:laplacianpower}
Fix $\sigma\in\{1,\ldots,\tau\}$ and $p\in\nn^*\cup\{\infty,\omega\}$. If $\Phi\in\mathscr{C}^p(D,A)$, then
\[\Psi:D^\sigma\to A,\quad (\alpha+x^1+\ldots+x^\sigma,\beta_{\sigma+1},\ldots,\beta_{\tau})\mapsto\Phi(\alpha,\Vert x^1\Vert,\ldots,\Vert x^\sigma\Vert,\beta_{\sigma+1},\ldots,\beta_{\tau})\]
is a $\mathscr{C}^{p-1}$ map that is $\mathscr{C}^p$ in $D^\sigma_*$. Recall that operators $\partial_{\alpha_0},\ldots,\partial_{\alpha_{t_0}},\partial_{\beta_1},\ldots,\partial_{\beta_\tau}$ on $\mathscr{C}^p(D,A)$ have been set up in Definition~\ref{def:debaralpha} and let us use for the analogous operators on $\mathscr{C}^p(D^\sigma,A)$ (or on $\mathscr{C}^p(D^\sigma_*,A)$) the temporary notation $\partial_{\widetilde{\alpha}_0},\ldots,\partial_{\widetilde{\alpha}_{t_\sigma}},\partial_{\widetilde{\beta}_1},\ldots,\partial_{\widetilde{\beta}_{\tau-\sigma}}$ . Now let $u\in\{0,\ldots,t_0\},v\in\{1,\ldots,\sigma\},s\in\{t_{v-1}+1,\ldots,t_v\},h\in\{1,\ldots,\tau-\sigma\}$. Then
\begin{align*}
&(\partial_{\widetilde{\alpha}_u}\Psi)(\alpha+x^1+\ldots+x^\sigma,\beta_{\sigma+1},\ldots,\beta_{\tau})=(\partial_{\alpha_u}\Phi)(\alpha,\Vert x^1\Vert,\ldots,\Vert x^\sigma\Vert,\beta_{\sigma+1},\ldots,\beta_{\tau})\,,\\
&(\partial_{\widetilde{\alpha}_s}\Psi)(\alpha+x^1+\ldots+x^\sigma,\beta_{\sigma+1},\ldots,\beta_{\tau})=\frac{x_s}{\Vert x^v\Vert}\,(\partial_{\beta_v}\Phi)(\alpha,\Vert x^1\Vert,\ldots,\Vert x^\sigma\Vert,\beta_{\sigma+1},\ldots,\beta_{\tau})\,,\\
&(\partial_{\widetilde{\beta}_h}\Psi)(\alpha+x^1+\ldots+x^\sigma,\beta_{\sigma+1},\ldots,\beta_{\tau})=(\partial_{\beta_{h+\sigma}}\Phi)(\alpha,\Vert x^1\Vert,\ldots,\Vert x^\sigma\Vert,\beta_{\sigma+1},\ldots,\beta_{\tau})
\end{align*}
for all $(\alpha+x^1+\ldots+x^\sigma,\beta_{\sigma+1},\ldots,\beta_{\tau})\in D^\sigma_*$. If, moreover, $p\geq2$ and if we adopt the temporary notation $\Box:=\sum_{s=t_0+1}^{t_\sigma}\partial_{\widetilde{\alpha}_s}^2$, then
\begin{align*}
&(\partial_{\widetilde{\alpha}_u}^2\Psi)(\alpha+x^1+\ldots+x^\sigma,\beta_{\sigma+1},\ldots,\beta_{\tau})=(\partial_{\alpha_u}^2\Phi)(\alpha,\Vert x^1\Vert,\ldots,\Vert x^\sigma\Vert,\beta_{\sigma+1},\ldots,\beta_{\tau})\,,\\
&(\Box\Psi)(\alpha+x^1+\ldots+x^\sigma,\beta_{\sigma+1},\ldots,\beta_{\tau})=\sum_{v=1}^\sigma(\partial_{\beta_v}^2\Phi)(\alpha,\Vert x^1\Vert,\ldots,\Vert x^\sigma\Vert,\beta_{\sigma+1},\ldots,\beta_{\tau})\\
&\quad+\sum_{v=1}^\sigma\frac{t_v-t_{v-1}-1}{\Vert x^v\Vert}(\partial_{\beta_v}\Phi)(\alpha,\Vert x^1\Vert,\ldots,\Vert x^\sigma\Vert,\beta_{\sigma+1},\ldots,\beta_{\tau})\,,\\
&(\partial_{\widetilde{\beta}_h}^2\Psi)(\alpha+x^1+\ldots+x^\sigma,\beta_{\sigma+1},\ldots,\beta_{\tau})=(\partial_{\beta_{h+\sigma}}^2\Phi)(\alpha,\Vert x^1\Vert,\ldots,\Vert x^\sigma\Vert,\beta_{\sigma+1},\ldots,\beta_{\tau})
\end{align*}
for all $(\alpha+x^1+\ldots+x^\sigma,\beta_{\sigma+1},\ldots,\beta_{\tau})\in D^\sigma_*$. Finally, for all $m\in\zz$,
\begin{align*}
&\partial_{\widetilde{\alpha}_s}\Vert x^v\Vert^m=m x_s \Vert x^v\Vert^{m-2}\,,\\
&\Box\Vert x^v\Vert^m=m(t_v-t_{v-1}+m-2)\Vert x^v\Vert^{m-2}\,,\\
&\partial_{\widetilde{\alpha}_s}\left(\frac{x^v}{\Vert x^v\Vert^m}\right)=v_s\,\Vert x^v\Vert^{-m}-mx^v\,x_s \Vert x^v\Vert^{-m-2}\,,\\
&\Box\left(\frac{x^v}{\Vert x^v\Vert^m}\right)=m\frac{m+t_{v-1}-t_v}{\Vert x^v\Vert^{m+2}}x^v\,.
\end{align*}
\end{proposition}

Here and throughout the paper, we abuse notation by denoting the map $(\alpha+x^1+\ldots+x^\sigma,\beta_{\sigma+1},\ldots,\beta_{\tau})\mapsto\Vert x^v\Vert^m$ simply by $\Vert x^v\Vert^m$ and the map $(\alpha+x^1+\ldots+x^\sigma,\beta_{\sigma+1},\ldots,\beta_{\tau})\mapsto\frac{x^v}{\Vert x^v\Vert^m}$ simply by $\frac{x^v}{\Vert x^v\Vert^m}$.

\begin{proof}
The first, third, fourth and sixth formulas are proven by applying Remark~\ref{rmk:incrementalratio} first to the operators $\partial_{\widetilde{\alpha}_u}$ and $\partial_{\widetilde{\beta}_h}$ on $\Psi$, then to the operators $\partial_{\alpha_u}$ and $\partial_{\beta_{h+\sigma}}$ on $\Phi$. The second formula follows from the computation
\begin{align*}
(\partial_{\widetilde{\alpha}_s}\Psi)(\alpha+x^1+\ldots+x^\sigma,\beta_{\sigma+1},\ldots,\beta_{\tau})&=(\partial_{\beta_v}\Phi)(\alpha,\Vert x^1\Vert,\ldots,\Vert x^\sigma\Vert,\beta_{\sigma+1},\ldots,\beta_{\tau})\;\partial_{\widetilde{\alpha}_s}\Vert x^v\Vert\\
&=\frac{x_s}{\Vert x^v\Vert}\,(\partial_{\beta_v}\Phi)(\alpha,\Vert x^1\Vert,\ldots,\Vert x^\sigma\Vert,\beta_{\sigma+1},\ldots,\beta_{\tau})\,.
\end{align*}
By applying it to the map $\Phi(\alpha,\beta)=\beta_v^m$, which yields $\Psi(\alpha+x^1+\ldots+x^\sigma,\beta_{\sigma+1},\ldots,\beta_{\tau})=\Vert x^v\Vert^m$, we obtain the seventh formula. We now compute
\begin{align*}
&(\partial_{\widetilde{\alpha}_s}^2\Psi)(\alpha+x^1+\ldots+x^\sigma,\beta_{\sigma+1},\ldots,\beta_{\tau})=\Vert x^v\Vert^{-1}(\partial_{\beta_v}\Phi)(\alpha,\Vert x^1\Vert,\ldots,\Vert x^\sigma\Vert,\beta_{\sigma+1},\ldots,\beta_{\tau})\\
&\quad+x_s\,\partial_{\widetilde{\alpha}_s}\left(\Vert x^v\Vert^{-1}\right)(\partial_{\beta_v}\Phi)(\alpha,\Vert x^1\Vert,\ldots,\Vert x^\sigma\Vert,\beta_{\sigma+1},\ldots,\beta_{\tau})\\
&\quad+x_s\Vert x^v\Vert^{-1}(\partial_{\beta_v}^2\Phi)(\alpha,\Vert x^1\Vert,\ldots,\Vert x^\sigma\Vert,\beta_{\sigma+1},\ldots,\beta_{\tau})\,\partial_{\widetilde{\alpha}_s}\Vert x^v\Vert\\
&=x_s^2\Vert x^v\Vert^{-2}(\partial_{\beta_v}^2\Phi)(\alpha,\Vert x^1\Vert,\ldots,\Vert x^\sigma\Vert,\beta_{\sigma+1},\ldots,\beta_{\tau})\\
&\quad+\left(\Vert x^v\Vert^{-1}-x_s^2\Vert x^v\Vert^{-3}\right)(\partial_{\beta_v}\Phi)(\alpha,\Vert x^1\Vert,\ldots,\Vert x^\sigma\Vert,\beta_{\sigma+1},\ldots,\beta_{\tau})\,,\\
&(\Box\Psi)(\alpha+x^1+\ldots+x^\sigma,\beta_{\sigma+1},\ldots,\beta_{\tau})=\sum_{v=1}^\sigma(\partial_{\beta_v}^2\Phi)(\alpha,\Vert x^1\Vert,\ldots,\Vert x^\sigma\Vert,\beta_{\sigma+1},\ldots,\beta_{\tau})\\
&\quad+\sum_{v=1}^\sigma(t_v-t_{v-1}-1)\Vert x^v\Vert^{-1}(\partial_{\beta_v}\Phi)(\alpha,\Vert x^1\Vert,\ldots,\Vert x^\sigma\Vert,\beta_{\sigma+1},\ldots,\beta_{\tau})\,,
\end{align*}
and the fifth formula is proven. By applying it to the map $\Phi(\alpha,\beta)=\beta_v^m$, we obtain the eighth formula. Finally, the computations
\begin{align*}
\partial_{\widetilde{\alpha}_s}\left(\frac{x^v}{\Vert x^v\Vert^m}\right)&=v_s\,\Vert x^v\Vert^{-m}+x^v\,\partial_{\widetilde{\alpha}_s}\Vert x^v\Vert^{-m}=v_s\,\Vert x^v\Vert^{-m}-mx^v\,x_s \Vert x^v\Vert^{-m-2}\,,\\
\partial_{\widetilde{\alpha}_s}^2\left(\frac{x^v}{\Vert x^v\Vert^m}\right)&=v_s\,\partial_{\widetilde{\alpha}_s}\Vert x^v\Vert^{-m}-mv_s\,x_s \Vert x^v\Vert^{-m-2}-mx^v\Vert x^v\Vert^{-m-2}-mx^v\,x_s\partial_{\widetilde{\alpha}_s}\Vert x^v\Vert^{-m-2}\\
&=-2mv_sx_s\Vert x^v\Vert^{-m-2}-mx^v\Vert x^v\Vert^{-m-2}+m(m+2)x^vx_s^2\Vert x^v\Vert^{-m-4}\\
&=\frac{m}{\Vert x^v\Vert^{m+2}}\left(x^v\left((m+2)\frac{x_s^2}{\Vert x^v\Vert^2}-1\right)-2v_sx_s\right)\,,\\
\Box\left(\frac{x^v}{\Vert x^v\Vert^m}\right)&=m\frac{x^v}{\Vert x^v\Vert^{m+2}}(m+2+t_{v-1}-t_v-2)=m(m+t_{v-1}-t_v)\frac{x^v}{\Vert x^v\Vert^{m+2}}
\end{align*}
prove the ninth and tenth formulas.
\end{proof}

We are now ready for the announced computation of $\Delta_{T_\sigma}f$ for $f\in\slice_T^p(\Omega_D,A)$ with $p\in\{\infty,\omega\}$. To fully justify the statement, we recall that $\Delta_T$ preserves $\slice_T^p(\Omega_D,A)$.

\begin{theorem}\label{thm:varietyoflaplacians}
Fix $p\in\{\infty,\omega\}$ and $\sigma\in\{1,\ldots,\tau\}$. If $f=\I(F)\in\slice_T^p(\Omega_D,A)$, then there exists $g=\I(G)\in\slice_T^p(\Omega_D,A)$ such that
\[\Delta_{T_\sigma}f=\Delta_Tf+g\in\slice_T^p(\Omega_D,A)\,.\]
For every $K\in\mathscr{P}(\tau)$, the $K$-component of $G\in\stem_{T}^p(D,A\otimes\rr^{2^\tau})$ fulfills the equality
\begin{align}\label{eq:laplacians}
&G_K(\alpha,\beta_1,\ldots,\beta_{\tau})\\
&=\sum_{v\in\{1,\ldots,\sigma\}}\frac{t_v-t_{v-1}-1}{\beta_v}\,(\partial_{\beta_v}F_K)(\alpha,\beta_1,\ldots,\beta_{\tau})+\sum_{v\in\{1,\ldots,\sigma\}\cap K}\frac{1+t_{v-1}-t_v}{\beta_v^2}\,F_K(\alpha,\beta_1,\ldots,\beta_{\tau})\,,\notag
\end{align}
in $D':=\{(\alpha,\beta_1,\ldots,\beta_{\tau})\in D:\beta_1\cdots\beta_\sigma\neq0\}$.
\end{theorem}

\begin{proof}
As a first step, we define $G_K\in\mathscr{C}^p(D',A)$ by formula~\eqref{eq:laplacians} for every $K\in\mathscr{P}(\tau)$,
we let $G:=\sum_{K\in\mathscr{P}(\tau)}E_KG_K$, and we establish that $G\in\stem_{T}^p(D',A\otimes\rr^{2^\tau})$ through the following reasoning. Fix $(\alpha,\beta)\in D'$. First assume $h\in K$, whence $F_K(\alpha,\overline{\beta}^h)=-F_K(\alpha,\beta)$, $(\partial_{\beta_h}F_K)(\alpha,\overline{\beta}^h)=(\partial_{\beta_h}F_K)(\alpha,\beta)$ and $(\partial_{\beta_v}F_K)(\alpha,\overline{\beta}^h)=-(\partial_{\beta_v}F_K)(\alpha,\beta)$ for all $v\in\{1,\ldots,\tau\}\setminus\{h\}$. If $h\in\{1,\ldots,\sigma\}$, then
\begin{align*}
G_K(\alpha,\overline{\beta}^h)&=\frac{t_h-t_{h-1}-1}{-\beta_h}\,(\partial_{\beta_h}F_K)(\alpha,\beta)+\sum_{v\in\{1,\ldots,\sigma\}\setminus\{h\}}\frac{t_v-t_{v-1}-1}{\beta_v}\,(-\partial_{\beta_v}F_K)(\alpha,\beta)\\
&\quad+\sum_{v\in\{1,\ldots,\sigma\}\cap K}\frac{1+t_{v-1}-t_v}{\beta_v^2}\,(-F_K(\alpha,\beta))=-G_K(\alpha,\beta)\,,
\end{align*}
while if $h\not\in\{1,\ldots,\sigma\}$, then
\begin{align*}
G_K(\alpha,\overline{\beta}^h)&=\sum_{v\in\{1,\ldots,\sigma\}}\frac{t_v-t_{v-1}-1}{\beta_v}\,(-\partial_{\beta_v}F_K)(\alpha,\beta)+\sum_{v\in\{1,\ldots,\sigma\}\cap K}\frac{1+t_{v-1}-t_v}{\beta_v^2}\,(-F_K(\alpha,\beta))\\
&=-G_K(\alpha,\beta)\,.
\end{align*}
Now assume $h\not\in K$, whence $F_K(\alpha,\overline{\beta}^h)=F_K(\alpha,\beta)$, $(\partial_{\beta_h}F_K)(\alpha,\overline{\beta}^h)=(-\partial_{\beta_h}F_K)(\alpha,\beta)$ and $(\partial_{\beta_v}F_K)(\alpha,\overline{\beta}^h)=(\partial_{\beta_v}F_K)(\alpha,\beta)$ for all $v\in\{1,\ldots,\tau\}\setminus\{h\}$. If $h\in\{1,\ldots,\sigma\}$, then
\begin{align*}
G_K(\alpha,\overline{\beta}^h)&=\frac{t_h-t_{h-1}-1}{-\beta_h}\,(-\partial_{\beta_h}F_K)(\alpha,\beta)+\sum_{v\in\{1,\ldots,\sigma\}\setminus\{h\}}\frac{t_v-t_{v-1}-1}{\beta_v}\,(\partial_{\beta_v}F_K)(\alpha,\beta)\\
&\quad+\sum_{v\in\{1,\ldots,\sigma\}\cap K}\frac{1+t_{v-1}-t_v}{\beta_v^2}\,F_K(\alpha,\beta)=G_K(\alpha,\beta)\,,
\end{align*}
while if $h\not\in\{1,\ldots,\sigma\}$, then
\begin{align*}
G_K(\alpha,\overline{\beta}^h)&=\sum_{v\in\{1,\ldots,\sigma\}}\frac{t_v-t_{v-1}-1}{\beta_v}\,(\partial_{\beta_v}F_K)(\alpha,\beta)+\sum_{v\in\{1,\ldots,\sigma\}\cap K}\frac{1+t_{v-1}-t_v}{\beta_v^2}\,F_K(\alpha,\beta)\\
&=G_K(\alpha,\beta)\,.
\end{align*}

We now take a second step. By Proposition~\ref{prop:multitilde}, $f\in\slice_{T_\sigma}^p(\Omega_D,A)$. Thus, $\Delta_{T_\sigma}f\in\slice_{T_\sigma}^p(\Omega_D,A)\subset\mathscr{C}^p(\Omega_D,A)$, where the last inclusion follows from Proposition~\ref{prop:analyticTfunction}. As we already mentioned, the $T_\sigma$-stem function inducing $\Delta_{T_\sigma}f$ is $\Delta_{T_\sigma}F^\sigma\in\stem_{T_\sigma}^p(D^\sigma,A\otimes\rr^{2^{\tau-\sigma}})$. For every $H\in\mathscr{P}(\tau-\sigma)$, we are now in a position to compute the $H$-component of $\Delta_{T_\sigma}F^\sigma$ as
\[(\Delta_{T_\sigma}F^\sigma)_H=\left(\sum_{u=0}^{t_0}\partial_{\widetilde{\alpha}_u}^2+\Box+\sum_{h=1}^{\tau-\sigma}\partial_{\widetilde{\beta}_h}^2\right)F^\sigma_H\,.\]
Here, we have adopted the notations of Proposition~\ref{prop:laplacianpower}. In Proposition~\ref{prop:multitilde}, we have computed the $H$-component of $F^\sigma$ as
\begin{align*}
&F^\sigma_H(\alpha+x^1+\ldots+x^\sigma,\beta_{\sigma+1},\ldots,\beta_{\tau})\\
&=\sum_{p=0}^\sigma\sum_{1\leq k_1<\ldots<k_p\leq \sigma}(-1)^{p|H|}\frac{x^{k_1}}{\Vert x^{k_1}\Vert}\ldots\frac{x^{k_p}}{\Vert x^{k_p}\Vert}F_{\{k_1,\ldots,k_p\}\cup(H+\sigma)}(\alpha,\Vert x^1\Vert,\ldots,\Vert x^\sigma\Vert,\beta_{\sigma+1},\ldots,\beta_{\tau})
\end{align*}
in $D^\sigma_*$. Thus, in the same set $D^\sigma_*$,
\[(\Delta_{T_\sigma}F^\sigma)_H=\sum_{p=0}^\sigma\sum_{1\leq k_1<\ldots<k_p\leq \sigma}(-1)^{p|H|}(I+II+III)\,,\]
where $I,II,III$ are specified as follows. Using for $\{k_1,\ldots,k_p\}\cup(H+\sigma)$ the temporary notation $K$ and setting $\Psi(\alpha+x^1+\ldots+x^\sigma,\beta_{\sigma+1},\ldots,\beta_{\tau}):=F_K(\alpha,\Vert x^1\Vert,\ldots,\Vert x^\sigma\Vert,\beta_{\sigma+1},\ldots,\beta_{\tau})$,
\begin{align*}
I&:=\Box\left(\frac{x^{k_1}}{\Vert x^{k_1}\Vert}\ldots\frac{x^{k_p}}{\Vert x^{k_p}\Vert}\right)\Psi=\left(\frac{x^{k_1}}{\Vert x^{k_1}\Vert}\ldots\frac{x^{k_p}}{\Vert x^{k_p}\Vert}\right)\sum_{s=1}^p\frac{1+t_{k_s-1}-t_{k_s}}{\Vert x^{k_s}\Vert^2}\Psi\,.
\end{align*}
For the second equality, we applied Proposition~\ref{prop:laplacianpower}. Moreover,
\begin{align*}
&II(\alpha+x^1+\ldots+x^\sigma,\beta_{\sigma+1},\ldots,\beta_{\tau}):=\\
&=2\sum_{s=t_0+1}^{t_\sigma}\partial_{\widetilde{\alpha}_s}\left(\frac{x^{k_1}}{\Vert x^{k_1}\Vert}\ldots\frac{x^{k_p}}{\Vert x^{k_p}\Vert}\right)\,(\partial_{\widetilde{\alpha}_s}\Psi)(\alpha+x^1+\ldots+x^\sigma,\beta_{\sigma+1},\ldots,\beta_{\tau})\\
&=2\sum_{v=1}^p\sum_{s=t_{k_v-1}+1}^{t_{k_v}}\frac{x^{k_1}}{\Vert x^{k_1}\Vert}\ldots\left(\frac{v_s}{\Vert x^{k_v}\Vert}-\frac{x^{k_v}\,x_s}{\Vert x^{k_v}\Vert^3}\right)\ldots\frac{x^{k_p}}{\Vert x^{k_p}\Vert}\frac{x_s}{\Vert x^{k_v}\Vert}\,(\partial_{\beta_v}F_K)(\alpha,\Vert x^1\Vert,\ldots)\\
&=2\sum_{v=1}^p\frac{x^{k_1}}{\Vert x^{k_1}\Vert}\ldots\left(\frac{x^{k_v}}{\Vert x^{k_v}\Vert^2}-\frac{x^{k_v}\,\Vert x^{k_v}\Vert^2}{\Vert x^{k_v}\Vert^4}\right)\ldots\frac{x^{k_p}}{\Vert x^{k_p}\Vert}(\partial_{\beta_v}F_K)(\alpha,\Vert x^1\Vert,\ldots,\Vert x^\sigma\Vert,\beta_{\sigma+1},\ldots,\beta_{\tau})\\
&\equiv0\,,
\end{align*}
where the second equality follows by applying Proposition~\ref{prop:laplacianpower} twice. Finally,
\begin{align*}
&III(\alpha+x^1+\ldots+x^\sigma,\beta_{\sigma+1},\ldots,\beta_{\tau}):=\\
&=\left(\frac{x^{k_1}}{\Vert x^{k_1}\Vert}\ldots\frac{x^{k_p}}{\Vert x^{k_p}\Vert}\right)\left(\sum_{u=0}^{t_0}\partial_{\widetilde{\alpha}_u}^2\Psi+\Box\Psi+\sum_{h=1}^{\tau-\sigma}\partial_{\widetilde{\beta}_h}^2\Psi\right)(\alpha+x^1+\ldots+x^\sigma,\beta_{\sigma+1},\ldots,\beta_{\tau})\\
&=\left(\frac{x^{k_1}}{\Vert x^{k_1}\Vert}\ldots\frac{x^{k_p}}{\Vert x^{k_p}\Vert}\right)\left(\sum_{u=0}^{t_0}\partial_{\alpha_u}^2F_K+\sum_{v=1}^{\tau}\partial_{\beta_v}^2F_K+\sum_{v=1}^\sigma\frac{t_v-t_{v-1}-1}{\Vert x^v\Vert}\,\partial_{\beta_v}F_K\right)(\alpha,\Vert x^1\Vert,\ldots)\\
&=\left(\frac{x^{k_1}}{\Vert x^{k_1}\Vert}\ldots\frac{x^{k_p}}{\Vert x^{k_p}\Vert}\right)\left((\Delta_TF)_K+\sum_{v=1}^\sigma\frac{t_v-t_{v-1}-1}{\Vert x^v\Vert}\,\partial_{\beta_v}F_K\right)(\alpha,\Vert x^1\Vert,\ldots)\,.
\end{align*}
For the second equality, we applied again Proposition~\ref{prop:laplacianpower}. We conclude that
\begin{align*}
&(\Delta_{T_\sigma}F^\sigma)_H(\alpha+x^1+\ldots+x^\sigma,\beta_{\sigma+1},\ldots,\beta_{\tau})\\
&=\sum_{p=0}^\sigma\sum_{1\leq k_1<\ldots<k_p\leq \sigma}(-1)^{p|H|}\frac{x^{k_1}}{\Vert x^{k_1}\Vert}\ldots\frac{x^{k_p}}{\Vert x^{k_p}\Vert}(\Delta_TF)_{\{k_1,\ldots,k_p\}\cup(H+\sigma)}(\alpha,\Vert x^1\Vert,\ldots,\Vert x^\sigma\Vert,\beta_{\sigma+1},\ldots)\\
&\quad+\sum_{p=0}^\sigma\sum_{1\leq k_1<\ldots<k_p\leq \sigma}(-1)^{p|H|}\frac{x^{k_1}}{\Vert x^{k_1}\Vert}\ldots\frac{x^{k_p}}{\Vert x^{k_p}\Vert}G_{\{k_1,\ldots,k_p\}\cup(H+\sigma)}(\alpha,\Vert x^1\Vert,\ldots,\Vert x^\sigma\Vert,\beta_{\sigma+1},\ldots)\\
&=(\Delta_TF)^\sigma_H(\alpha+x^1+\ldots+x^\sigma,\beta_{\sigma+1},\ldots,\beta_{\tau})+G^\sigma_H(\alpha+x^1+\ldots+x^\sigma,\beta_{\sigma+1},\ldots,\beta_{\tau})
\end{align*}
in $D^\sigma_*$, whence $\Delta_{T_\sigma}F^\sigma$ equals $(\Delta_TF)^\sigma+G^\sigma$ in $D^\sigma_*$. Equivalently, $\Delta_{T_\sigma}f$ equals in $\Omega_{D'}$ the $T_\sigma$-function induced by $(\Delta_TF+G)^\sigma$, whence the $T$-function induced by $\Delta_TF+G$. Taking into account that $\Delta_{T_\sigma}f\in\mathscr{C}^p(\Omega_D,A)$, Proposition~\ref{prop:analyticTfunction} guarantees that $\Delta_{T_\sigma}f\in\slice_T^p(\Omega_D,A)$. Setting $g:=\Delta_{T_\sigma}f-\Delta_Tf\in\slice_T^p(\Omega_D,A)$ yields the desired conclusion.
\end{proof}

The following result concerning the effect of $\debar_{\widetilde{T}}=\debar_{T_1}$ on a $T$-function $f$ will be useful to achieve our main theorem in the forthcoming Section~\ref{sec:FueterSceTfunctions}.

\begin{theorem}\label{thm:debartilde}
Let $p\in\{\infty,\omega\}$. If $f=\I(F)\in\slice_T^p(\Omega_D,A)$, then there exists $g\in\slice_T^p(\Omega_D,A)$ such that $\debar_{\widetilde{T}}f=\debar_T f+(1+t_0-t_1)g$ and $\partial_{\widetilde{T}}f=\partial_T f+(t_1-t_0-1)g$ (whence $\debar_{\widetilde{T}}f,\partial_{\widetilde{T}}f\in\slice_T^p(\Omega_D,A)$). For every $K\in\mathscr{P}(\tau)$, the $K$-component of the function $G\in\stem_{T}^p(D,A\otimes\rr^{2^\tau})$ inducing $g$ fulfills the equality
\begin{align}\label{eq:G}
&G_K(\alpha,\beta_1,\ldots,\beta_\tau)=\left\{
\begin{array}{ll}
\beta_1^{-1}F_{\{1\}\cup K}(\alpha,\beta_1,\ldots,\beta_\tau)&\mathrm{if\ }1\not\in K\\
0&\mathrm{if\ }1\in K
\end{array}
\right.\,,
\end{align}
in $D':=\{(\alpha,\beta_1,\ldots,\beta_\tau)\in D:\beta_1\neq0\}$.
\end{theorem}

\begin{proof}
Since $f=\I(F)$ for some $T$-stem function on $D$, we know that $f$ is also a $\widetilde{T}$-function, induced by the $\widetilde{T}$-stem function $\widetilde{F}=\sum_{H\in\mathscr{P}(\tau-1)}\widetilde{E}_H\widetilde{F}_H$ such that
\[\widetilde{F}_H(\alpha+x^1,\widetilde{\beta})=F_{H+1}(\alpha,\Vert x^1\Vert,\widetilde{\beta})+(-1)^{|H|}\frac{x^1}{\Vert x^1\Vert}F_{\{1\}\cup(H+1)}(\alpha,\Vert x^1\Vert,\widetilde{\beta})\]
for all $\alpha\in\rr_{0,t_0},x^1\in\rr_{t_0+1,t_1},\widetilde{\beta}\in\rr^{\tau-1}$ such that $(\alpha,\Vert x^1\Vert,\widetilde{\beta})\in D$. In particular, $\debar_{\widetilde{T}}f\in\slice_{\widetilde{T}}^p(\Omega_D,A)\subset\mathscr{C}^p(\Omega_D,A)$. Computing $\debar_{\widetilde{T}}f$ is the same as computing $\debar_{\widetilde{T}}\widetilde{F}$, where
\begin{align*}
(\debar_{\widetilde{T}}\widetilde{F})_H&=\debar_{\widetilde{\alpha}}^{|H|+1}\widetilde{F}_H+\sum_{h=1}^{\tau-1}(-1)^{\sigma(h,H)+1}\partial_{\widetilde{\beta}_h}\widetilde{F}_{H\bigtriangleup\{h\}}\\
&=\partial_{\widetilde{\alpha}_0}\widetilde{F}_H+(-1)^{|H|}\sum_{s=0}^{t_1}v_s\partial_{\widetilde{\alpha}_s}\widetilde{F}_H+\sum_{h=1}^{\tau-1}(-1)^{\sigma(h,H)+1}\partial_{\widetilde{\beta}_h}\widetilde{F}_{H\bigtriangleup\{h\}}
\end{align*}
for all $H\in\mathscr{P}(\tau-1)$ by Definitions~\ref{def:debaralpha} and~\ref{def:stemoperators}.

Fix $H\in\mathscr{P}(\tau-1)$. Using Proposition~\ref{prop:laplacianpower} in its special case $\sigma=1$, we get
\[(\partial_{\widetilde{\alpha}_0}\widetilde{F}_H)(\alpha+x^1,\widetilde{\beta})=(\partial_{\alpha_0}F_{H+1})(\alpha,\Vert x^1\Vert,\widetilde{\beta})+(-1)^{|H|}\frac{x^1}{\Vert x^1\Vert}(\partial_{\alpha_0}F_{\{1\}\cup(H+1)})(\alpha,\Vert x^1\Vert,\widetilde{\beta})\,,\]
as well as
\begin{align*}
&(-1)^{|H|}\sum_{s=1}^{t_0}v_s(\partial_{\widetilde{\alpha}_s}\widetilde{F}_H)(\alpha+x^1,\widetilde{\beta})\\
&=(-1)^{|H|}\sum_{s=1}^{t_0}v_s(\partial_{\alpha_s}F_{H+1})(\alpha,\Vert x^1\Vert,\widetilde{\beta})-\frac{x^1}{\Vert x^1\Vert}\sum_{s=1}^{t_0}v_s(\partial_{\alpha_s}F_{\{1\}\cup(H+1)})(\alpha,\Vert x^1\Vert,\widetilde{\beta})\\
&=(-1)^{|H+1|}\sum_{s=1}^{t_0}v_s(\partial_{\alpha_s}F_{H+1})(\alpha,\Vert x^1\Vert,\widetilde{\beta})\\
&\quad+(-1)^{|H|}\frac{x^1}{\Vert x^1\Vert}(-1)^{|\{1\}\cup(H+1)|}\sum_{s=1}^{t_0}v_s(\partial_{\alpha_s}F_{\{1\}\cup(H+1)})(\alpha,\Vert x^1\Vert,\widetilde{\beta})\,,
\end{align*}
where we took into account the equalities $v_s\frac{x^1}{\Vert x^1\Vert}=-\frac{x^1}{\Vert x^1\Vert}v_s$ (valid for all $s\in\{1,\ldots,t_0\}$), $|H+1|=|H|$ and $|\{1\}\cup(H+1)|=|H|+1$.

Another application of Proposition~\ref{prop:laplacianpower} (in its special case $\sigma=1$) gives
\begin{align*}
&\sum_{h=1}^{\tau-1}(-1)^{\sigma(h,H)+1}(\partial_{\widetilde{\beta}_h}\widetilde{F}_{H\bigtriangleup\{h\}})(\alpha+x^1,\widetilde{\beta})=\sum_{v=2}^{\tau}(-1)^{\sigma(v-1,H)+1}(\partial_{\beta_v}F_{(H\bigtriangleup\{v-1\})+1})(\alpha,\Vert x^1\Vert,\widetilde{\beta})\\
&\quad+\frac{x^1}{\Vert x^1\Vert}\sum_{v=2}^{\tau}(-1)^{\sigma(v-1,H)+1+|H\bigtriangleup\{v-1\}|}(\partial_{\beta_v}F_{\{1\}\cup((H\bigtriangleup\{v-1\})+1)})(\alpha,\Vert x^1\Vert,\widetilde{\beta})\\
&=\sum_{v=2}^{\tau}(-1)^{\sigma(v,H+1)+1}(\partial_{\beta_v}F_{(H+1)\bigtriangleup\{v\}})(\alpha,\Vert x^1\Vert,\widetilde{\beta})\\
&\quad+(-1)^{|H|}\frac{x^1}{\Vert x^1\Vert}\sum_{v=2}^{\tau}(-1)^{\sigma(v,\{1\}\cup(H+1))+1}(\partial_{\beta_v}F_{(\{1\}\cup(H+1))\bigtriangleup\{v\}})(\alpha,\Vert x^1\Vert,\widetilde{\beta})
\end{align*}
for all $(\alpha+x^1,\widetilde{\beta})\in\widetilde{D}_*$. Here, we used the equalities $\sigma(v-1,H)=\sigma(v,H+1)=\sigma(v,\{1\}\cup(H+1))\pm1$, $|H\bigtriangleup\{v-1\}|=|H|\pm1$, $(H\bigtriangleup\{v-1\})+1=(H+1)\bigtriangleup\{v\}$ and
\[\{1\}\cup(({H\bigtriangleup\{v-1\}})+1)=\{1\}\cup((H+1)\bigtriangleup\{v\})=(\{1\}\cup(H+1))\bigtriangleup\{v\}\,,\]
which are valid for all $v\in\{2,\ldots,\tau\}$.

We now compute, using again Proposition~\ref{prop:laplacianpower} in its special case $\sigma=1$,
\begin{align*}
&(-1)^{|H|}\sum_{s=t_0+1}^{t_1}v_s(\partial_{\widetilde{\alpha}_s}\widetilde{F}_H)(\alpha+x^1,\widetilde{\beta})=(-1)^{|H|}\sum_{s=t_0+1}^{t_1}v_s\frac{x_s}{\Vert x^1\Vert}\,(\partial_{\beta_1}F_{H+1})(\alpha,\Vert x^1\Vert,\widetilde{\beta})\\
&\quad+\sum_{s=t_0+1}^{t_1}v_s\left(\frac{v_s}{\Vert x^1\Vert}-\frac{x_s\,x^1}{\Vert x^1\Vert^3}\right)F_{\{1\}\cup(H+1)}(\alpha,\Vert x^1\Vert,\widetilde{\beta})\\
&\quad+\sum_{s=t_0+1}^{t_1}v_s\frac{x_sx^1}{\Vert x^1\Vert^2}\,(\partial_{\beta_1}F_{\{1\}\cup(H+1)})(\alpha,\Vert x^1\Vert,\widetilde{\beta})\\
&=(-1)^{|H|}\frac{x^1}{\Vert x^1\Vert}\,(\partial_{\beta_1}F_{(\{1\}\cup(H+1))\bigtriangleup\{1\}})(\alpha,\Vert x^1\Vert,\widetilde{\beta})+\frac{1+t_0-t_1}{\Vert x^1\Vert}F_{\{1\}\cup(H+1)}(\alpha,\Vert x^1\Vert,\widetilde{\beta})\\
&\quad-(\partial_{\beta_1}F_{(H+1)\bigtriangleup\{1\}})(\alpha,\Vert x^1\Vert,\widetilde{\beta})\\
&=(-1)^{\sigma(1,H+1)+1}(\partial_{\beta_1}F_{(H+1)\bigtriangleup\{1\}})(\alpha,\Vert x^1\Vert,\widetilde{\beta})\\
&\quad+(-1)^{|H|}\frac{x^1}{\Vert x^1\Vert}(-1)^{\sigma(1,\{1\}\cup(H+1))+1}\,(\partial_{\beta_1}F_{(\{1\}\cup(H+1))\bigtriangleup\{1\}})(\alpha,\Vert x^1\Vert,\widetilde{\beta})\\
&\quad+\frac{1+t_0-t_1}{\Vert x^1\Vert}F_{\{1\}\cup(H+1)}(\alpha,\Vert x^1\Vert,\widetilde{\beta})\,.
\end{align*}
For the second equality, we took into account that $\sum_{s=t_0+1}^{t_1}v_s^2=(t_1-t_0)(-1)=t_0-t_1$ and that $\sum_{s=t_0+1}^{t_1}\frac{v_sx_sx^1}{\Vert x^1\Vert^2}=\frac{x^1}{\Vert x^1\Vert}\frac{x^1}{\Vert x^1\Vert}=-1$. For the third equality, we used the identities $-1=(-1)^{\sigma(1,H+1)+1},1=(-1)^{\sigma(1,\{1\}\cup(H+1))+1}$.

Now let us sum up the four parts we computed: taking into account (for both $K=H+1$ and $K=\{1\}\cup(H+1)$) the formula
\[(\debar_TF)_K=\partial_{\alpha_0}F_K+(-1)^{|K|}\sum_{s=1}^{t_0}v_s\,\partial_{\alpha_s}F_K+\sum_{h=1}^\tau(-1)^{\sigma(h,K)+1}\partial_{\beta_h}F_{K\bigtriangleup\{h\}}\,,\]
we get
\begin{align*}
&(\debar_{\widetilde{T}}\widetilde{F})_H(\alpha+x^1,\widetilde{\beta})=(\debar_T F)_{H+1}(\alpha,\Vert x^1\Vert,\widetilde{\beta})+(-1)^{|H|}\frac{x^1}{\Vert x^1\Vert}(\debar_T F)_{\{1\}\cup(H+1)}(\alpha,\Vert x^1\Vert,\widetilde{\beta})\\
&\quad+\frac{1+t_0-t_1}{\Vert x^1\Vert}F_{\{1\}\cup(H+1)}(\alpha,\Vert x^1\Vert,\widetilde{\beta})\\
&=\widetilde{(\debar_T F)}_H(\alpha+x^1,\widetilde{\beta})+\frac{1+t_0-t_1}{\Vert x^1\Vert}F_{\{1\}\cup(H+1)}(\alpha,\Vert x^1\Vert,\widetilde{\beta})\,.
\end{align*}
If we define $G:=\sum_{K\in\mathscr{P}(\tau)}E_KG_K$ in $D'$ according to~\eqref{eq:G}, then $G\in\stem_T^p(D',A\otimes\rr^{2^\tau})$ thanks to the following remarks, valid for $K\in\mathscr{P}(\tau),h\in\{2,\ldots,\tau\},(\alpha,\beta)\in D'$: if $1\in K$, then $G_K(\alpha,\overline{\beta}^1),G_K(\alpha,\overline{\beta}^h)$ and $G_K(\alpha,\beta)$ all vanish; if $1\not\in K$, then
\begin{align*}
G_K(\alpha,\overline{\beta}^1)&=(-\beta_1)^{-1}F_{\{1\}\cup K}(\alpha,\overline{\beta}^1)=\beta_1^{-1}F_{\{1\}\cup K}(\alpha,\beta)=G_K(\alpha,\beta)\,,&\\
G_K(\alpha,\overline{\beta}^h)&=\beta_1^{-1}F_{\{1\}\cup K}(\alpha,\overline{\beta}^h)=\beta_1^{-1}F_{\{1\}\cup K}(\alpha,\beta)=G_K(\alpha,\beta)&\text{if }h\not\in K\,,\\
G_K(\alpha,\overline{\beta}^h)&=\beta_1^{-1}F_{\{1\}\cup K}(\alpha,\overline{\beta}^h)=-\beta_1^{-1}F_{\{1\}\cup K}(\alpha,\beta)=-G_K(\alpha,\beta)&\text{if }h\in K\,.
\end{align*}
Using~\eqref{eq:G}, we now compute
\begin{align*}
\widetilde{G}_H(\alpha+x^1,\widetilde{\beta})&=G_{H+1}(\alpha,\Vert x^1\Vert,\widetilde{\beta})+(-1)^{|H|}\frac{x^1}{\Vert x^1\Vert}G_{\{1\}\cup(H+1)}(\alpha,\Vert x^1\Vert,\widetilde{\beta})\\
&=\Vert x^1\Vert^{-1}F_{\{1\}\cup(H+1)}(\alpha,\Vert x^1\Vert,\widetilde{\beta})+0
\end{align*}
in $D'$ and conclude that $\debar_{\widetilde{T}}\widetilde{F}=\widetilde{(\debar_T F)}+(1+t_0-t_1)\widetilde{G}$ in $D'$. Since $\debar_{\widetilde{T}}f$ is the $\widetilde{T}$-function induced by $\debar_{\widetilde{T}}\widetilde{F}$, we conclude that $\debar_{\widetilde{T}}f$ coincides in $\Omega_{D'}$ with the $T$-function induced by $\debar_T F+(1+t_0-t_1)G$. Taking into account that $\debar_{\widetilde{T}}f\in\mathscr{C}^p(\Omega_D,A)$, Proposition~\ref{prop:analyticTfunction} yields that $\debar_{\widetilde{T}}f\in\slice_T^p(\Omega_D,A)$. Choosing $g\in\slice_T^p(\Omega_D,A)$ so that $(1+t_0-t_1)g=\debar_{\widetilde{T}}f-\debar_Tf$ yields the desired conclusion.

Running through the proof with appropriate sign changes proves that $\partial_{\widetilde{T}}f=\partial_T f+(t_1-t_0-1)g$, as desired.
\end{proof}

In case $t_1=t_0+1$, Theorem~\ref{thm:debartilde} yields that $\debar_{\widetilde{T}}=\debar_T$ and $\partial_{\widetilde{T}}=\partial_T$, consistently with Proposition~\ref{prop:shortsteps}. Let us now compare $\debar_{\widetilde{T}}$ to $\debar_T$ in an explicit example.

\begin{example}
Let $A=C\ell(0,6),V=\rr^7$ and $T=(0,3,6)$. We already computed in Example~\ref{ex:(0,3,6)} the strongly $T$-regular polynomial $3\T_{(2,1)}(x)=x_0^3+3x_0^2x^2-3x_0\Vert x^1\Vert^2+6x_0x^1x^2-3\Vert x^1\Vert^2x^2$. We also saw in Example~\ref{ex:(0,3,6)ter} that the $T$-stem function inducing $3\T_{(2,1)}$ is $F(\alpha,\beta_1,\beta_2)=E_\emptyset(\alpha^3-3\alpha\beta_1^2)+E_{\{2\}}(3\alpha^2-3\beta_1^2)\beta_2+E_{\{1,2\}}(6\alpha\beta_1\beta_2)$. We know by construction that $3\debar_T\T_{(2,1)}\equiv0$. Using Proposition~\ref{prop:globaloperators} and taking into account that $\widetilde{x}^0=x_0+x^1$ and $\widetilde{x}^1=x_0+x^2$, we compute
\begin{align*}
3\debar_{\widetilde{T}}\T_{(2,1)}&=3\partial_{x_0}\T_{(2,1)}(x)+3\sum_{s=1}^3e_s\partial_{x_s}\T_{(2,1)}(x)+3\frac{x^2}{\Vert x^2\Vert^2}\sum_{s=4}^{6}\,x_s\,\partial_{x_s}\T_{(2,1)}(x)\\
&=3x_0^2+6x_0x^2-3\Vert x^1\Vert^2+6x^1x^2+\sum_{s=1}^3e_s\left(-6x_0x_s+6x_0e_sx^2-6x_sx^2\right)\\
&\quad+\frac{x^2}{\Vert x^2\Vert^2}\sum_{s=4}^{6}\,x_s\,\left(3x_0^2e_s+6x_0x^1e_s-3\Vert x^1\Vert^2e_s\right)\\
&=3x_0^2+6x_0x^2-3\Vert x^1\Vert^2+6x^1x^2-6x_0x^1+6x_0(-3)x^2-6x^1x^2\\
&\quad+\frac{x^2}{\Vert x^2\Vert^2}\left(3x_0^2x^2+6x_0x^1x^2-3\Vert x^1\Vert^2x^2\right)\\
&=3x_0^2-12x_0x^2-3\Vert x^1\Vert^2-6x_0x^1-3x_0^2+6x_0x^1+3\Vert x^1\Vert^2=-12x_0x^2
\end{align*}
Now, taking into account that $F_1\equiv0$ and $F_{\{1,2\}}(\alpha,\beta_1,\beta_2)=6\alpha\beta_1\beta_2$, we have
\[G(\alpha,\beta_1,\beta_2)=E_{\{1,2\}}\beta_1^{-1}6\alpha\beta_1\beta_2=E_{\{1,2\}}6\alpha\beta_2\,,\]
whence $g(x)=6x_0x^2$. The function $3\debar_T\T_{(2,1)}(x)+(1+3-6)g(x)=-2g(x)=-12x_0x^2$ coincides with $3\debar_{\widetilde{T}}\T_{(2,1)}$ by Theorem~\ref{thm:debartilde} or by direct inspection.
\end{example}


\section{Iterates of the $\widetilde{T}$-Laplacian on $T$-harmonic $T$-functions}\label{sec:iteratesharmonic}

We now wish to study the effect on a $T$-harmonic $T$-function $f$ of the iterates of $\Delta_{\widetilde{T}}:\slice_{\widetilde{T}}^2(\Omega_D,A)\to\slice_{\widetilde{T}}^0(\Omega_D,A)$. This is equivalent to studying the effect of the iterates of $\Delta_{\widetilde{T}}:\stem_{\widetilde{T}}^2(\widetilde{D},A\otimes\rr^{2^\tau})\to\stem_{\widetilde{T}}^0(\widetilde{D},A\otimes\rr^{2^\tau})$ on $\widetilde{F}$ for any $T$-stem function $F$ belonging to the kernel of $\Delta_T:\stem_T^2(D,A\otimes\rr^{2^\tau})\to\stem_T^0(D,A\otimes\rr^{2^\tau})$.

We begin with a useful definition and some related properties.

\begin{definition}
We define $b_{-1,-1}:=1$. For every $n\in\nn$, we define $b_{n,-1}:=0$ and, for every $\ell\in\{0,\ldots,n\}$
\[b_{n,\ell}:=\frac{(-1)^{n+\ell}}{2^{n-\ell}}\,\frac{(2n-\ell)!}{\ell!(n-\ell)!}\,.\]
\end{definition}

We point out that $(-1)^{n+\ell}b_{n,\ell}$ are the coefficients of the reverse Bessel polynomial, see~\cite[page 6]{librogrosswald}. The next properties will be useful.

\begin{lemma}\label{lem:besselcoefficients}
If $n\in\nn$, then the following properties hold true:
\begin{align*}
&b_{n,n}=1\,,&\\
&b_{n+1,\ell}=(\ell-2n-1)b_{n,\ell}+b_{n,\ell-1}&\text{if }0\leq\ell\leq n\,,\\
&b_{n-1,\ell-1}=b_{n,\ell}+(\ell+1)b_{n,\ell+1}&\text{if }0\leq\ell\leq n-1\,,\\
&\ell(\ell-2n-1)b_{n,\ell}+2(\ell-n-1)b_{n,\ell-1}=0&\text{if }0\leq\ell\leq n\,.
\end{align*}
\end{lemma}

\begin{proof}
By direct computation,
\[b_{n,n}=\frac{(-1)^{2n}}{2^{0}}\frac{(n)!}{n!\,0!}=1\,.\]
Moreover,
\begin{align*}
&(\ell-2n-1)b_{n,\ell}+b_{n,\ell-1}\\
&=(\ell-2n-1)\frac{(-1)^{n+\ell}}{2^{n-\ell}}\frac{(2n-\ell)!}{\ell!(n-\ell)!}+\frac{(-1)^{n+\ell-1}}{2^{n-\ell+1}}\frac{(2n-\ell+1)!}{(\ell-1)!(n-\ell+1)!}\\
&=\frac{(-1)^{n+1+\ell}}{2^{n+1-\ell}}\frac{(2n+1-\ell)!}{\ell!(n+1-\ell)!}(2(n+1-\ell)+\ell)\\
&=\frac{(-1)^{n+1+\ell}}{2^{n+1-\ell}}\frac{(2n+2-\ell)!}{\ell!(n+1-\ell)!}=b_{n+1,\ell}
\end{align*}
for $1\leq\ell\leq n$ and
\[(-2n-1)b_{n,0}+b_{n,-1}=(-2n-1)\frac{(-1)^n}{2^n}\frac{(2n)!}{n!}=\frac{(-1)^{n+1}}{2^n}\frac{(2n+1)!}{n!}=\frac{(-1)^{n+1}}{2^{n+1}}\frac{(2n+2)!}{(n+1)!}=b_{n+1,0}\]
for the case $\ell=0$. Additionally,
\begin{align*}
&b_{n,\ell}+(\ell+1)b_{n,\ell+1}\\
&\frac{(-1)^{n+\ell}}{2^{n-\ell}}\frac{(2n-\ell)!}{\ell!(n-\ell)!}+(\ell+1)\frac{(-1)^{n+\ell+1}}{2^{n-\ell-1}}\frac{(2n-\ell-1)!}{(\ell+1)!(n-\ell-1)!}\\
&=\frac{(-1)^{n+\ell}}{2^{n-\ell}}\frac{(2n-\ell-1)!}{\ell!(n-\ell)!}(2n-\ell-2(n-\ell))\\
&=\frac{(-1)^{n+\ell}}{2^{n-\ell}}\frac{(2n-\ell-1)!}{(\ell-1)!(n-\ell)!}\\
&=\frac{(-1)^{n-1+\ell-1}}{2^{n-1-(\ell-1)}}\frac{(2(n-1)-(\ell-1))!}{(\ell-1)!(n-1-(\ell-1))!}=b_{n-1,\ell-1}\,,
\end{align*}
provided $n\geq1$. Finally,
\begin{align*}
&\ell(\ell-2n-1)b_{n,\ell}+2(\ell-n-1)b_{n,\ell-1}\\
&=\ell(\ell-2n-1)\frac{(-1)^{n+\ell}}{2^{n-\ell}}\frac{(2n-\ell)!}{\ell!(n-\ell)!}+2(\ell-n-1)\frac{(-1)^{n+\ell-1}}{2^{n-\ell+1}}\frac{(2n-\ell+1)!}{(\ell-1)!(n-\ell+1)!}\\
&=-\frac{(-1)^{n+\ell}}{2^{n-\ell}}\frac{(2n+1-\ell)!}{(\ell-1)!(n-\ell)!}+\frac{(-1)^{n+\ell}}{2^{n-\ell}}\frac{(2n-\ell+1)!}{(\ell-1)!(n-\ell)!}=0\,.\qedhere
\end{align*}
\end{proof}

Let $f\in\slice_T^2(\Omega_D,A)$ and assume $\Delta_Tf\equiv0$. We are now ready for the announced study of the effect of the iterates of $\Delta_{\widetilde{T}}$ on $f$. We recall that, by Remark~\ref{rmk:Tharmonicareanalytic}, $f$ automatically belongs to $\slice_T^\omega(\Omega_D,A)\subset \mathscr{C}^\omega(\Omega_D,A)$.

\begin{theorem}\label{thm:powersoflaplacian}
Let $f$ belong to the kernel of $\Delta_T:\slice_T^2(\Omega_D,A)\to\slice_T^0(\Omega_D,A)$.
\begin{enumerate}
\item Assume $t_1-t_0$ to be an even natural number. For any $n\in\nn$ there exists a unique element $f^{[n]}\in\slice_T^\omega(\Omega_D,A)$ such that 
\[\Delta_{\widetilde{T}}^nf=f^{[n]}\,\prod_{\ell=1}^n(t_1-t_0-2\ell+1)\,.\]
\item Assume $t_1-t_0$ to be an odd natural number $2n_1+1$. Then $\Delta_{\widetilde{T}}^nf\equiv0$ in $\Omega_D$ for all $n>n_1$. For any $n\leq n_1$, there exists a unique element $f^{[n]}\in\slice_T^\omega(\Omega_D,A)$ such that 
\[\Delta_{\widetilde{T}}^nf=f^{[n]}\,\prod_{\ell=1}^n(t_1-t_0-2\ell+1)\,.\]
\end{enumerate}
The $T$-stem function $F^{[n]}\in\stem_T^\omega(D,A\otimes\rr^{2^\tau})$ inducing $f^{[n]}$ (whenever the latter is defined) has the following property: given $K\in\mathscr{P}(\tau)$,
\begin{align}\label{eq:F[n]}
&F^{[n]}_K(\alpha,\beta)=\left\{
\begin{array}{ll}
\sum_{\ell=0}^n\frac{b_{n-1,\ell-1}}{\beta_1^{2n-\ell}}(\partial_{\beta_1}^\ell F_K)(\alpha,\beta)&\text{if }1\not\in K\\
\sum_{\ell=0}^n\frac{b_{n,\ell}}{\beta_1^{2n-\ell}}(\partial_{\beta_1}^\ell F_K)(\alpha,\beta)&\text{if }1\in K
\end{array}
\right.
\end{align}
in $D':=\{(\alpha,\beta)\in D:\beta_1\neq0\}$.
\end{theorem}

\begin{proof}
Our proof is by induction on $n$.

For $n=0$, the thesis follows by setting $f^{[0]}:=f$ and $F^{[0]}:=F$, if we take into account that $b_{-1,-1}=1=b_{0,0}$.

Now let us deal with the inductive step from $n$ to $n+1$. This step is trivial if $t_1-t_0$ is an odd natural number $2n_1+1$ and $n>n_1$. In all other cases, the inductive hypothesis is that there exists $f^{[n]}=\I(F^{[n]})\in\slice_T^\omega(\Omega_D,A)$ with $F^{[n]}$ fulfilling condition~\eqref{eq:F[n]} in $D'$ and that $\Delta_{\widetilde{T}}^nf=f^{[n]}\,\prod_{\ell=1}^n(t_1-t_0-2\ell+1)$. Let $F^{[n+1]}$ denote the element of $\stem_T^\omega(D',A\otimes\rr^{2^\tau})$ defined by (the $n+1$ instance of) formula~\eqref{eq:F[n]}: we claim that $\Delta_{\widetilde{T}}\widetilde{F}^{[n]}=(t_1-t_0-2n-1)\,\widetilde{F}^{[n+1]}$ in $\widetilde{D}_*$. Our claim implies that $\Delta_{\widetilde{T}}f^{[n]}=(t_1-t_0-2n-1)\,f^{[n+1]}$ in $\Omega_{D'}$, if $f^{[n+1]}$ denotes the element of $\slice_T^\omega(\Omega_{D'},A)$ induced by $F^{[n+1]}$. Thus, $\Delta_{\widetilde{T}}^{n+1}f=\Delta_{\widetilde{T}}f^{[n]}\,\prod_{\ell=1}^n(t_1-t_0-2\ell+1)$ is an element of $\mathscr{C}^\omega(\Omega_D,A)$ whose restriction to $\Omega_{D'}$ is the $T$-function $(t_1-t_0-2n-1)\,f^{[n+1]}$. Proposition~\ref{prop:analyticTfunction} yields that $\Delta_{\widetilde{T}}^{n+1}f\in\slice_T^\omega(\Omega_D,A)$. In case $t_1-t_0=2n+1$, it follows that $\Delta_{\widetilde{T}}^{n+1}f\equiv0$ in $\Omega_D$, as desired. If, instead, $t_1-t_0\neq2n+1$, then $T$-stem function inducing $\frac{\Delta_{\widetilde{T}}^{n+1}f}{t_1-t_0-2n-1}$ coincides with $F^{[n+1]}$ in $D'$; we still denote it by $F^{[n+1]}$. Since $\I(F^{[n+1]})$ coincides with $f^{[n+1]}$ in $\Omega_{D'}$, we still denote it by $f^{[n+1]}$. It follows that 
\[\Delta_{\widetilde{T}}^{n+1}f=\Delta_{\widetilde{T}}f^{[n]}\prod_{\ell=1}^{n}(t_1-t_0-2\ell+1)=f^{[n+1]}\prod_{\ell=1}^{n+1}(t_1-t_0-2\ell+1)\,,\]
as desired.

We are left with proving our claim that $\Delta_{\widetilde{T}}\widetilde{F}^{[n]}=(t_1-t_0-2n-1)\,\widetilde{F}^{[n+1]}$ in $\widetilde{D}_*$. From formula~\eqref{eq:F[n]} and Definition~\ref{def:tilde}, we get $\widetilde{F}^{[n]}:=\sum_{H\in\mathscr{P}(\tau-1)}\widetilde{F}^{[n]}_HE_H$ with
\begin{align}\label{eq:Gn}
&\widetilde{F}^{[n]}_H(\alpha+x^1,\widetilde{\beta}):=\\
&\sum_{\ell=0}^n\left(\frac{b_{n-1,\ell-1}}{\Vert x^1\Vert^{2n-\ell}}(\partial_{\beta_1}^\ell F_{H+1})(\alpha,\Vert x^1\Vert,\widetilde{\beta})+(-1)^{|H|}\frac{b_{n,\ell}\,x^1}{\Vert x^1\Vert^{2n+1-\ell}}(\partial_{\beta_1}^\ell F_{\{1\}\cup(H+1)})(\alpha,\Vert x^1\Vert,\widetilde{\beta})\right)\notag
\end{align}
for all $H\in\mathscr{P}(\tau-1),\alpha\in\rr_{0,t_0},x^1\in\rr_{t_0+1,t_1},\widetilde{\beta}\in\rr^{\tau-1}$ such that $(\alpha+x^1,\widetilde{\beta})\in\widetilde{D}_*$. By Definition~\ref{def:stemoperators} and Remark~\ref{rmk:dedebaralpha}, for all $H\in\mathscr{P}(\tau-1)$ we have
\[(\Delta_{\widetilde{T}}\widetilde{F}^{[n]})_H=\sum_{s=0}^{t_1}\partial_{\widetilde{\alpha}_s}^2\widetilde{F}^{[n]}_H+\sum_{h=1}^{\tau-1}\partial_{\widetilde{\beta}_h}^2\widetilde{F}^{[n]}_H\,.\]
Using Proposition~\ref{prop:laplacianpower} in its special case $\sigma=1$, we get
\begin{align*}
&\sum_{s=0}^{t_0}\left(\partial_{\widetilde{\alpha}_s}^2\widetilde{F}^{[n]}_H\right)(\alpha+x^1,\widetilde{\beta})=\sum_{\ell=0}^n\frac{b_{n-1,\ell-1}}{\Vert x^1\Vert^{2n-\ell}}\left(\partial_{\beta_1}^\ell \sum_{s=0}^{t_0}\partial_{\alpha_s}^2F_{H+1}\right)(\alpha,\Vert x^1\Vert,\widetilde{\beta})\\
&\quad+(-1)^{|H|}\sum_{\ell=0}^n\frac{b_{n,\ell}\,x^1}{\Vert x^1\Vert^{2n+1-\ell}}\left(\partial_{\beta_1}^\ell \sum_{s=0}^{t_0}\partial_{\alpha_s}^2F_{\{1\}\cup(H+1)}\right)(\alpha,\Vert x^1\Vert,\widetilde{\beta})
\end{align*}
and
\begin{align*}
&\sum_{h=1}^{\tau-1}(\partial_{\widetilde{\beta}_h}^2\widetilde{F}^{[n]}_H)(\alpha+x^1,\widetilde{\beta})=\sum_{\ell=0}^n\frac{b_{n-1,\ell-1}}{\Vert x^1\Vert^{2n-\ell}}\left(\partial_{\beta_1}^\ell\sum_{u=2}^{\tau}\partial_{\beta_u}^2F_{H+1}\right)(\alpha,\Vert x^1\Vert,\widetilde{\beta})\\
&+(-1)^{|H|}\sum_{\ell=0}^n\frac{b_{n,\ell}\,x^1}{\Vert x^1\Vert^{2n+1-\ell}}\left(\partial_{\beta_1}^\ell \sum_{u=2}^{\tau}\partial_{\beta_u}^2F_{\{1\}\cup(H+1)}\right)(\alpha,\Vert x^1\Vert,\widetilde{\beta})
\end{align*}
in $\widetilde{D}_*$. Now, $0\equiv(\Delta_TF)_K=\sum_{s=0}^{t_0}\partial_{\alpha_s}^2F_K+\sum_{u=1}^{\tau}\partial_{\beta_u}^2F_K$ implies $\sum_{s=0}^{t_0}\partial_{\alpha_s}^2F_K+\sum_{u=2}^{\tau}\partial_{\beta_u}^2F_K=-\partial_{\beta_1}^2F_K$ for all $K\in\mathscr{P}(\tau)$. We conclude that
\begin{align}\label{eq:easypartoflaplacian}
&\sum_{s=0}^{t_0}(\partial_{\widetilde{\alpha}_s}^2\widetilde{F}^{[n]}_H)(\alpha+x^1,\widetilde{\beta})+\sum_{h=1}^{\tau-1}(\partial_{\widetilde{\beta}_h}^2\widetilde{F}^{[n]}_H)(\alpha+x^1,\widetilde{\beta})=I+(-1)^{|H|}II\,,\\
&I:=-\sum_{\ell=0}^n\frac{b_{n-1,\ell-1}}{\Vert x^1\Vert^{2n-\ell}}(\partial_{\beta_1}^{\ell+2} F_{H+1})(\alpha,\Vert x^1\Vert,\widetilde{\beta})\,,\notag\\
&II:=-\sum_{\ell=0}^n\frac{b_{n,\ell}\,x^1}{\Vert x^1\Vert^{2n+1-\ell}}(\partial_{\beta_1}^{\ell+2} F_{\{1\}\cup(H+1)})(\alpha,\Vert x^1\Vert,\widetilde{\beta})\notag
\end{align}
in $\widetilde{D}_*$. We must now compute $\sum_{s=t_0+1}^{t_1}\partial_{\widetilde{\alpha}_s}^2\widetilde{F}^{[n]}_H$, which is the same as $\Box\widetilde{F}^{[n]}_H$ in the notation of Proposition~\ref{prop:laplacianpower} (case $\sigma=1$). Thus,
\begin{align*}
&\sum_{s=t_0+1}^{t_1}\partial_{\widetilde{\alpha}_s}^2\widetilde{F}^{[n]}_H=\Box\widetilde{F}^{[n]}_H=\sum_{\ell=0}^nb_{n-1,\ell-1}\;\Box(\Vert x^1\Vert^{\ell-2n}\,\Psi_\ell)+(-1)^{|H|}\sum_{\ell=0}^nb_{n,\ell}\;\Box\left(\frac{x^1}{\Vert x^1\Vert^{2n+1-\ell}}\,\Xi_\ell\right)
\end{align*}
where
\begin{align*}
\Psi_\ell(\alpha+x^1,\widetilde{\beta})&:=(\partial_{\beta_1}^\ell F_{H+1})(\alpha,\Vert x^1\Vert,\widetilde{\beta})\,,\\
\Xi_\ell(\alpha+x^1,\widetilde{\beta})&:=(\partial_{\beta_1}^\ell F_{\{1\}\cup(H+1)})(\alpha,\Vert x^1\Vert,\widetilde{\beta})\,.
\end{align*}
For pur computation, let us start with single ingredients. Using Proposition~\ref{prop:laplacianpower} with $\sigma=1=m$, we compute
\begin{align*}
&\Box\left(\Vert x^1\Vert^{\ell-2n}\,\Psi_\ell(\alpha+x^1,\widetilde{\beta})\right)=\left(\Box\Vert x^1\Vert^{\ell-2n}\right)\Psi_\ell(\alpha+x^1,\widetilde{\beta})\\
&\quad+\Vert x^1\Vert^{\ell-2n}\,(\Box\Psi_\ell)(\alpha+x^1,\widetilde{\beta})+2\sum_{s=t_0+1}^{t_1}\left(\partial_{\widetilde{\alpha}_s}\Vert x^1\Vert^{\ell-2n}\right)\,(\partial_{\widetilde{\alpha}_s}\Psi_\ell)(\alpha+x^1,\widetilde{\beta})\\
&=(\ell-2n)(t_1-t_0+\ell-2n-2)\,\Vert x^1\Vert^{\ell-2n-2}\,(\partial_{\beta_1}^\ell F_{H+1})(\alpha,\Vert x^1\Vert,\widetilde{\beta})\\
&\quad+\Vert x^1\Vert^{\ell-2n}\,\left(\frac{t_1-t_0-1}{\Vert x^1\Vert}(\partial_{\beta_1}^{\ell+1}F_{H+1})(\alpha,\Vert x^1\Vert,\widetilde{\beta})+(\partial_{\beta_1}^{\ell+2}F_{H+1})(\alpha,\Vert x^1\Vert,\widetilde{\beta})\right)\\
&\quad+2\sum_{s=t_0+1}^{t_1}(\ell-2n)\,x_s\,\Vert x^1\Vert^{\ell-2n-2}\,\frac{x_s}{\Vert x^1\Vert}(\partial_{\beta_1}^{\ell+1}F_{H+1})(\alpha,\Vert x^1\Vert,\widetilde{\beta})\\
&=(\ell-2n)(t_1-t_0+\ell-2n-2)\,\Vert x^1\Vert^{\ell-2n-2}\,(\partial_{\beta_1}^\ell F_{H+1})(\alpha,\Vert x^1\Vert,\widetilde{\beta})\\
&\quad+(t_1-t_0-1+2\ell-4n)\,\Vert x^1\Vert^{\ell-2n-1}\,(\partial_{\beta_1}^{\ell+1}F_{H+1})(\alpha,\Vert x^1\Vert,\widetilde{\beta})\\
&\quad+\Vert x^1\Vert^{\ell-2n}\,(\partial_{\beta_1}^{\ell+2}F_{H+1})(\alpha,\Vert x^1\Vert,\widetilde{\beta})
\end{align*}
Using the temporary notation $K$ for $\{1\}\cup(H+1)$, Proposition~\ref{prop:laplacianpower} (with $\sigma=1=m$) also allows the following computation:
\begin{align*}
&\Box\left(\frac{x^1}{\Vert x^1\Vert^{2n+1-\ell}}\,\Xi_\ell(\alpha+x^1,\widetilde{\beta})\right)=\Box\left(\frac{x^1}{\Vert x^1\Vert^{2n+1-\ell}}\right)\,\Xi_\ell(\alpha+x^1,\widetilde{\beta})\\
&\quad+\frac{x^1}{\Vert x^1\Vert^{2n+1-\ell}}\,(\Box\Xi_\ell)(\alpha+x^1,\widetilde{\beta})+2\sum_{s=t_0+1}^{t_1}\left(\partial_{\widetilde{\alpha}_s}\frac{x^1}{\Vert x^1\Vert^{2n+1-\ell}}\right)\,(\partial_{\widetilde{\alpha}_s}\Xi_\ell)(\alpha+x^1,\widetilde{\beta})\\
&=(2n+1-\ell)(2n+1-\ell+t_0-t_1)\frac{x^1}{\Vert x^1\Vert^{2n+3-\ell}}\,(\partial_{\beta_1}^\ell F_K)(\alpha,\Vert x^1\Vert,\widetilde{\beta})\\
&\quad+\frac{x^1}{\Vert x^1\Vert^{2n+1-\ell}}\,\left(\frac{t_1-t_0-1}{\Vert x^1\Vert}(\partial_{\beta_1}^{\ell+1} F_K)(\alpha,\Vert x^1\Vert,\widetilde{\beta})+(\partial_{\beta_1}^{\ell+2} F_K)(\alpha,\Vert x^1\Vert,\widetilde{\beta})\right)\\
&\quad+2\sum_{s=t_0+1}^{t_1}\frac{v_s+(\ell-2n-1)x^1\,x_s \Vert x^1\Vert^{-2}}{\Vert x^1\Vert^{2n+1-\ell}}\frac{x_s}{\Vert x^1\Vert}(\partial_{\beta_1}^{\ell+1} F_K)(\alpha,\Vert x^1\Vert,\widetilde{\beta})\\
&=(2n+1-\ell)(2n+1-\ell+t_0-t_1)\frac{x^1}{\Vert x^1\Vert^{2n+3-\ell}}\,(\partial_{\beta_1}^\ell F_K)(\alpha,\Vert x^1\Vert,\widetilde{\beta})\\
&\quad+(t_1-t_0-1+2\ell-4n)\frac{x^1}{\Vert x^1\Vert^{2n+2-\ell}}\,(\partial_{\beta_1}^{\ell+1} F_K)(\alpha,\Vert x^1\Vert,\widetilde{\beta})\\
&\quad+\frac{x^1}{\Vert x^1\Vert^{2n+1-\ell}}\,(\partial_{\beta_1}^{\ell+2} F_K)(\alpha,\Vert x^1\Vert,\widetilde{\beta})\,.
\end{align*}
It follows that
\begin{align}\label{eq:hardpartoflaplacian}
&\sum_{s=t_0+1}^{t_1}(\partial_{\widetilde{\alpha}_s}^2\widetilde{F}^{[n]}_H)(\alpha+x^1,\widetilde{\beta})=\Box\widetilde{F}^{[n]}_H(\alpha+x^1,\widetilde{\beta})=III+(-1)^{|H|}IV\,,\\
&III:=\sum_{\ell=0}^nb_{n-1,\ell-1}\Big((\ell-2n)(t_1-t_0+\ell-2n-2)\,\Vert x^1\Vert^{\ell-2n-2}\,(\partial_{\beta_1}^\ell F_{H+1})(\alpha,\Vert x^1\Vert,\widetilde{\beta})\notag\\
&\quad+(t_1-t_0-1+2\ell-4n)\,\Vert x^1\Vert^{\ell-2n-1}\,(\partial_{\beta_1}^{\ell+1}F_{H+1})(\alpha,\Vert x^1\Vert,\widetilde{\beta})\notag\\
&\quad+\Vert x^1\Vert^{\ell-2n}\,(\partial_{\beta_1}^{\ell+2}F_{H+1})(\alpha,\Vert x^1\Vert,\widetilde{\beta})\Big)\,,\notag\\
&IV:=\sum_{\ell=0}^nb_{n,\ell}\Big(
(2n+1-\ell)(2n+1-\ell+t_0-t_1)\frac{x^1}{\Vert x^1\Vert^{2n+3-\ell}}\,(\partial_{\beta_1}^\ell F_{\{1\}\cup(H+1)})(\alpha,\Vert x^1\Vert,\widetilde{\beta})\notag\\
&\quad+(t_1-t_0-1+2\ell-4n)\frac{x^1}{\Vert x^1\Vert^{2n+2-\ell}}\,(\partial_{\beta_1}^{\ell+1} F_{\{1\}\cup(H+1)})(\alpha,\Vert x^1\Vert,\widetilde{\beta})\notag\\
&\quad+\frac{x^1}{\Vert x^1\Vert^{2n+1-\ell}}\,(\partial_{\beta_1}^{\ell+2} F_{\{1\}\cup(H+1)})(\alpha,\Vert x^1\Vert,\widetilde{\beta})\Big)\notag
\end{align}
in $\widetilde{D}_*$. Summing~\eqref{eq:easypartoflaplacian} and~\eqref{eq:hardpartoflaplacian}, we conclude that
\[(\Delta_{\widetilde{T}}\widetilde{F}^{[n]})_H(\alpha+x^1,\widetilde{\beta})=I+III+(-1)^{|H|}(II+IV)\,.\]
Now,
\begin{align*}
I+III&=\sum_{\ell=0}^nb_{n-1,\ell-1}(\ell-2n)(t_1-t_0+\ell-2n-2)\,\Vert x^1\Vert^{\ell-2n-2}\,(\partial_{\beta_1}^\ell F_{H+1})(\alpha,\Vert x^1\Vert,\widetilde{\beta})\\
&\quad+\sum_{\ell=0}^nb_{n-1,\ell-1}(t_1-t_0-1+2\ell-4n)\,\Vert x^1\Vert^{\ell-2n-1}\,(\partial_{\beta_1}^{\ell+1}F_{H+1})(\alpha,\Vert x^1\Vert,\widetilde{\beta})\\
&=\sum_{\ell=0}^nb_{n-1,\ell-1}(\ell-2n)(t_1-t_0+\ell-2n-2)\,\Vert x^1\Vert^{\ell-2n-2}\,(\partial_{\beta_1}^\ell F_{H+1})(\alpha,\Vert x^1\Vert,\widetilde{\beta})\\
&\quad+\sum_{m=1}^{n+1}b_{n-1,m-2}(t_1-t_0+2m-4n-3)\,\Vert x^1\Vert^{m-2n-2}\,(\partial_{\beta_1}^mF_{H+1})(\alpha,\Vert x^1\Vert,\widetilde{\beta})\,,
\end{align*}
where we set $m:=\ell+1$. Using the identities $b_{n-1,-1}(-2n)=0,b_{n-1,n-1}=1$, we find that
\begin{align*}
I+III&=\sum_{\ell=1}^n\Big((t_1-t_0-2n-1)\left((\ell-2n)b_{n-1,\ell-1}+b_{n-1,\ell-2}\right)+(\ell-1)(\ell-2n)b_{n-1,\ell-1}\\
&\quad+2(\ell-n-1)b_{n-1,\ell-2}\Big)\,\Vert x^1\Vert^{\ell-2n-2}\,(\partial_{\beta_1}^\ell F_{H+1})(\alpha,\Vert x^1\Vert,\widetilde{\beta})\\
&\quad+(t_1-t_0-2n-1)\,\Vert x^1\Vert^{-n-1}\,(\partial_{\beta_1}^{n+1}F_{H+1})(\alpha,\Vert x^1\Vert,\widetilde{\beta})\\
&=(t_1-t_0-2n-1)\sum_{\ell=1}^nb_{n,\ell-1}\Vert x^1\Vert^{\ell-2n-2}\,(\partial_{\beta_1}^\ell F_{H+1})(\alpha,\Vert x^1\Vert,\widetilde{\beta})\\
&\quad+(t_1-t_0-2n-1)\,b_{n,n}\,\Vert x^1\Vert^{-n-1}\,(\partial_{\beta_1}^{n+1}F_{H+1})(\alpha,\Vert x^1\Vert,\widetilde{\beta})\\
&=(t_1-t_0-2n-1)\sum_{\ell=0}^{n+1}b_{n,\ell-1}\Vert x^1\Vert^{\ell-2n-2}\,(\partial_{\beta_1}^\ell F_{H+1})(\alpha,\Vert x^1\Vert,\widetilde{\beta})\,.
\end{align*}
For the second equality, we applied Lemma~\ref{lem:besselcoefficients} and the identity $1=b_{n,n}$. For the third equality, we used the identity $b_{n,-1}=0$. Similarly,
\begin{align*}
II+IV&=\sum_{\ell=0}^nb_{n,\ell}(2n+1-\ell)(2n+1-\ell+t_0-t_1)\frac{x^1}{\Vert x^1\Vert^{2n+3-\ell}}\,(\partial_{\beta_1}^\ell F_{\{1\}\cup(H+1)})(\alpha,\Vert x^1\Vert,\widetilde{\beta})\notag\\
&\quad+\sum_{\ell=0}^nb_{n,\ell}(t_1-t_0-1+2\ell-4n)\frac{x^1}{\Vert x^1\Vert^{2n+2-\ell}}\,(\partial_{\beta_1}^{\ell+1} F_{\{1\}\cup(H+1)})(\alpha,\Vert x^1\Vert,\widetilde{\beta})\\
&=\sum_{\ell=0}^nb_{n,\ell}(2n+1-\ell)(2n+1-\ell+t_0-t_1)\frac{x^1}{\Vert x^1\Vert^{2n+3-\ell}}\,(\partial_{\beta_1}^\ell F_{\{1\}\cup(H+1)})(\alpha,\Vert x^1\Vert,\widetilde{\beta})\notag\\
&\quad+\sum_{m=1}^{n+1}b_{n,m-1}(t_1-t_0+2m-4n-3)\frac{x^1}{\Vert x^1\Vert^{2n+3-m}}\,(\partial_{\beta_1}^m F_{\{1\}\cup(H+1)})(\alpha,\Vert x^1\Vert,\widetilde{\beta})\,,
\end{align*}
where we set $m:=\ell+1$. Using the identities $b_{n,-1}=0,b_{n,n}=1$,
\begin{align*}
II+IV&=\sum_{\ell=0}^n\Big((t_1-t_0-2n-1)\left((\ell-2n-1)b_{n,\ell}+b_{n,\ell-1}\right)\\
&\quad+\ell(\ell-2n-1)b_{n,\ell}+2(\ell-n-1)b_{n,\ell-1}\Big)\frac{x^1}{\Vert x^1\Vert^{2n+3-\ell}}\,(\partial_{\beta_1}^\ell F_{\{1\}\cup(H+1)})(\alpha,\Vert x^1\Vert,\widetilde{\beta})\\
&\quad+(t_1-t_0-2n-1)\,\frac{x^1}{\Vert x^1\Vert^{n+2}}\,(\partial_{\beta_1}^{n+1} F_{\{1\}\cup(H+1)})(\alpha,\Vert x^1\Vert,\widetilde{\beta})\\
&=(t_1-t_0-2n-1)\sum_{\ell=0}^nb_{n+1,\ell}\frac{x^1}{\Vert x^1\Vert^{2n+3-\ell}}\,(\partial_{\beta_1}^\ell F_{\{1\}\cup(H+1)})(\alpha,\Vert x^1\Vert,\widetilde{\beta})\\
&\quad+(t_1-t_0-2n-1)\,b_{n+1,n+1}\,\frac{x^1}{\Vert x^1\Vert^{n+2}}\,(\partial_{\beta_1}^{n+1} F_{\{1\}\cup(H+1)})(\alpha,\Vert x^1\Vert,\widetilde{\beta})\\
&=(t_1-t_0-2n-1)\sum_{\ell=0}^{n+1}b_{n+1,\ell}\frac{x^1}{\Vert x^1\Vert^{2n+3-\ell}}\,(\partial_{\beta_1}^\ell F_{\{1\}\cup(H+1)})(\alpha,\Vert x^1\Vert,\widetilde{\beta})\,.
\end{align*}
For the second equality, we applied Lemma~\ref{lem:besselcoefficients} and the identity $1=b_{n+1,n+1}$. We conclude that
\begin{align*}
&(\Delta_{\widetilde{T}}\widetilde{F}^{[n]})_H(\alpha+x^1,\widetilde{\beta})=(t_1-t_0-2n-1)\sum_{\ell=0}^{n+1}b_{n,\ell-1}\Vert x^1\Vert^{\ell-2n-2}\,(\partial_{\beta_1}^\ell F_{H+1})(\alpha,\Vert x^1\Vert,\widetilde{\beta})\\
&\quad+(-1)^{|H|}(t_1-t_0-2n-1)\sum_{\ell=0}^{n+1}b_{n+1,\ell}\frac{x^1}{\Vert x^1\Vert^{2n+3-\ell}}\,(\partial_{\beta_1}^\ell F_{\{1\}\cup(H+1)})(\alpha,\Vert x^1\Vert,\widetilde{\beta})\\
&=(t_1-t_0-2n-1)\,\widetilde{F}^{[n+1]}(\alpha+x^1,\widetilde{\beta})
\end{align*}
in $\widetilde{D}_*$. This establishes our claim and completes the proof.
\end{proof}

Let us provide an explicit example.

\begin{example}\label{ex:(0,3,6)quater}
Let $A=C\ell(0,6),V=\rr^7$ and $T=(0,3,6)$, whence $\widetilde{T}=(3,6)$. For the strongly $T$-regular polynomial $\T_{(3,2)}(x)$, we computed in Example~\ref{ex:(0,3,6)} the expression
\begin{align*}
10\T_{(3,2)}(x)&=3\T_{(1,2)}(x)(x_0^2+2x_0x^1-\Vert x^1\Vert^2)+3\T_{(2,1)}(x)(2x_0x^1+2x_0x^2-2x^1x^2)\\
&\quad+\T_{(3,0)}(x)(x_0^2-2x_0x^2-\Vert x^2\Vert^2)\,.
\end{align*}
The same example obtained the expressions
\begin{align*}
3\T_{(1,2)}(x)&=x_0^3+3x_0^2x^1-6x_0x^1x^2-3x_0\Vert x^2\Vert^2-3x^1\Vert x^2\Vert^2\,,\\
3\T_{(2,1)}(x)&=x_0^3+3x_0^2x^2-3x_0\Vert x^1\Vert^2+6x_0x^1x^2-3\Vert x^1\Vert^2x^2\,,\\
\T_{(3,0)}(x)&=x_0^3+3x_0^2x^1-3x_0\Vert x^1\Vert^2-\Vert x^1\Vert^2x^1\,.
\end{align*}
Taking into account that $\widetilde{x}^0=x_0+x^1=\sum_{s=0}^3x_se_s$ and $\widetilde{x}^1=x^2=\sum_{s=4}^6x_se_s$, we now wish to compute
\begin{align*}
\Delta_{\widetilde{T}}\T_{(3,2)}(x)&=\sum_{s=0}^3\partial_{x_s}^2\T_{(3,2)}(x)+\sum_{s,s'=4}^{6}\frac{x_sx_{s'}}{x_4^2+x_5^2+x_6^2}\,\partial_{x_s}\partial_{x_{s'}}\T_{(3,2)}(x)
\end{align*}
for all $x\in\rr^7$ with $x^2\neq0$. We begin with
\begin{align*}
&10\,\partial_{x_0}^2\T_{(3,2)}(x)=(6x_0+6x^1)(x_0^2+2x_0x^1-\Vert x^1\Vert^2)\\
&\quad+(x_0^3+3x_0^2x^1-6x_0x^1x^2-3x_0\Vert x^2\Vert^2-3x^1\Vert x^2\Vert^2)2\\
&\quad+2(3x_0^2+6x_0x^1-6x^1x^2-3\Vert x^2\Vert^2)(2x_0+2x^1)\\
&\quad+(6x_0+6x^2)(2x_0x^1+2x_0x^2-2x^1x^2)+0\\
&\quad+2(3x_0^2+6x_0x^2-3\Vert x^1\Vert^2+6x^1x^2)(2x^1+2x^2)\\
&\quad+(6x_0+6x^1)(x_0^2-2x_0x^2-\Vert x^2\Vert^2)\\
&\quad+(x_0^3+3x_0^2x^1-3x_0\Vert x^1\Vert^2-\Vert x^1\Vert^2x^1)2\\
&\quad+2(3x_0^2+6x_0x^1-3\Vert x^1\Vert^2)(2x_0-2x^2)\\
&=6x_0^3+18x_0^2x^1-18x_0\Vert x^1\Vert^2-6\Vert x^1\Vert^2x^1\\
&\quad+2x_0^3+6x_0^2x^1-12x_0x^1x^2-6x_0\Vert x^2\Vert^2-6x^1\Vert x^2\Vert^2\\
&\quad+12x_0^3+36x_0^2x^1-24x_0\Vert x^1\Vert^2-24x_0x^1x^2-12x_0\Vert x^2\Vert^2-24\Vert x^1\Vert^2x^2-12x^1\Vert x^2\Vert^2\\
&\quad+12x_0^2x^1+12x_0^2x^2-24x_0x^1x^2-12x_0\Vert x^2\Vert^2-12x^1\Vert x^2\Vert^2\\
&\quad+12x_0^2x^1+12x_0^2x^2-24x_0x^1x^2-24x_0\Vert x^2\Vert^2-12\Vert x^1\Vert^2x^1+12\Vert x^1\Vert^2x^2-24x^1\Vert x^2\Vert^2\\
&\quad+6x_0^3+6x_0^2x^1-12x_0^2x^2-12x_0x^1x^2-6x_0\Vert x^2\Vert^2-6x^1\Vert x^2\Vert^2\\
&\quad+2x_0^3+6x_0^2x^1-6x_0\Vert x^1\Vert^2-2\Vert x^1\Vert^2x^1\\
&\quad+12x_0^3+24x_0^2x^1-12x_0^2x^2-12x_0\Vert x^1\Vert^2-24x_0x^1x^2+12\Vert x^1\Vert^2x^2\\
&=40x_0^3+120x_0^2x^1-60x_0\Vert x^1\Vert^2-120x_0x^1x^2-60x_0\Vert x^2\Vert^2-20\Vert x^1\Vert^2x^1-60x^1\Vert x^2\Vert^2\,,
\end{align*}
whence
\[\partial_{x_0}^2\T_{(3,2)}(x)=4x_0^3+12x_0^2x^1-6x_0\Vert x^1\Vert^2-12x_0x^1x^2-6x_0\Vert x^2\Vert^2-2\Vert x^1\Vert^2x^1-6x^1\Vert x^2\Vert^2\,.\]
For $s\in\{1,2,3\}$, we compute
\begin{align*}
&10\,\partial_{x_s}^2\T_{(3,2)}(x)=0+(x_0^3+3x_0^2x^1-6x_0x^1x^2-3x_0\Vert x^2\Vert^2-3x^1\Vert x^2\Vert^2)(-2)\\
&\quad+2(3x_0^2e_s-6x_0e_sx^2-3e_s\Vert x^2\Vert^2)(2x_0e_s-2x_s)\\
&\quad+(-6x_0-6x^2)(2x_0x^1+2x_0x^2-2x^1x^2)+0\\
&\quad+2(-6x_0x_s+6x_0e_sx^2-6x_sx^2)(2x_0e_s-2e_sx^2)\\
&\quad+(-6x_0-2x^1-4x_se_s)(x_0^2-2x_0x^2-\Vert x^2\Vert^2)+0+0\\
&=-2x_0^3-6x_0^2x^1+12x_0x^1x^2+6x_0\Vert x^2\Vert^2+6x^1\Vert x^2\Vert^2\\
&\quad-12x_0^3-24x_0^2x^2+12x_0\Vert x^2\Vert^2-12x_0^2x_se_s+24x_0x_se_sx^2+12x_se_s\Vert x^2\Vert^2\\
&\quad-12x_0^2x^1-12x_0^2x^2+24x_0x^1x^2+12x_0\Vert x^2\Vert^2+12x^1\Vert x^2\Vert^2\\
&\quad+24x_0^2x^2+24x_0\Vert x^2\Vert^2-24x_0^2x_se_s+48x_0x_se_sx^2+24x_se_s\Vert x^2\Vert^2\\
&\quad-6x_0^3-2x_0^2x^1+12x_0^2x^2+4x_0x^1x^2+6x_0\Vert x^2\Vert^2+2x^1\Vert x^2\Vert^2\\
&\quad-4x_0^2x_se_s+8x_0x_se_sx^2+4x_se_s\Vert x^2\Vert^2\\
&=-20x_0^3-20x_0^2x^1+40x_0x^1x^2+60x_0\Vert x^2\Vert^2+20x^1\Vert x^2\Vert^2\\
&\quad-40x_0^2x_se_s+80x_0x_se_sx^2+40x_se_s\Vert x^2\Vert^2\,,
\end{align*}
whence
\begin{align*}
\sum_{s=1}^3\partial_{x_s}^2\T_{(3,2)}(x)
&=3(-2x_0^3-2x_0^2x^1+4x_0x^1x^2+6x_0\Vert x^2\Vert^2+2x^1\Vert x^2\Vert^2)\\
&\quad-4x_0^2x^1+8x_0x^1x^2+4x^1\Vert x^2\Vert^2\\
&=-6x_0^3-10x_0^2x^1+20x_0x^1x^2+18x_0\Vert x^2\Vert^2+10x^1\Vert x^2\Vert^2\,.
\end{align*}
For distinct $s,s'\in\{4,5,6\}$, we see that
\begin{align*}
&10\,\partial_{x_s}\partial_{x_{s'}}\T_{(3,2)}(x)\equiv0+0+0=0\,,\\
&10\,\partial_{x_s}^2\T_{(3,2)}(x)=(-6x_0-6x^1)(x_0^2+2x_0x^1-\Vert x^1\Vert^2)+0+0\\\
&\quad+0+0+2(3x_0^2e_s+6x_0x^1e_s-3\Vert x^1\Vert^2e_s)(2x_0e_s-2x^1e_s)\\
&\quad+0+(x_0^3+3x_0^2x^1-3x_0\Vert x^1\Vert^2-\Vert x^1\Vert^2x^1)(-2)+0\\
&=-6x_0^3-18x_0^2x^1+18x_0\Vert x^1\Vert^2+6\Vert x^1\Vert^2x^1\\
&\quad-12x_0^3-36x_0^2x^1+36x_0\Vert x^1\Vert^2+12\Vert x^1\Vert^2x^1\\
&\quad-2x_0^3-6x_0^2x^1+6x_0\Vert x^1\Vert^2+2\Vert x^1\Vert^2x^1\\
&=-20x_0^3-60x_0^2x^1+60x_0\Vert x^1\Vert^2+20\Vert x^1\Vert^2x^1\,,
\end{align*}
whence
\begin{align*}
\sum_{s,s'=4}^{6}\frac{x_sx_{s'}}{x_4^2+x_5^2+x_6^2}\,\partial_{x_s}\partial_{x_{s'}}\T_{(3,2)}(x)&=\sum_{s=4}^{6}\frac{x_s^2}{x_4^2+x_5^2+x_6^2}\,\partial_{x_s}^2\T_{(3,2)}(x)\\
&=-2x_0^3-6x_0^2x^1+6x_0\Vert x^1\Vert^2+2\Vert x^1\Vert^2x^1\,.
\end{align*}
Overall, we get that
\begin{align*}
\Delta_{\widetilde{T}}\T_{(3,2)}(x)&=4x_0^3+12x_0^2x^1-6x_0\Vert x^1\Vert^2-12x_0x^1x^2-6x_0\Vert x^2\Vert^2-2\Vert x^1\Vert^2x^1-6x^1\Vert x^2\Vert^2\\
&\quad-6x_0^3-10x_0^2x^1+20x_0x^1x^2+18x_0\Vert x^2\Vert^2+10x^1\Vert x^2\Vert^2\\
&\quad-2x_0^3-6x_0^2x^1+6x_0\Vert x^1\Vert^2+2\Vert x^1\Vert^2x^1\\
&=-4x_0^3-4x_0^2x^1+8x_0x^1x^2+12x_0\Vert x^2\Vert^2+4x^1\Vert x^2\Vert^2\,.
\end{align*}
for all $x\in\rr^7$. We have $\Delta_{\widetilde{T}}^2\T_{(3,2)}\equiv0$ in $\rr^7$ by Theorem~\ref{thm:powersoflaplacian} or by the direct computation
\begin{align*}
\Delta_{\widetilde{T}}^2\T_{(3,2)}(x)&=\sum_{s=0}^3\partial_{x_s}^2\Delta_{\widetilde{T}}\T_{(3,2)}(x)+\sum_{s,s'=4}^{6}\frac{x_sx_{s'}}{x_4^2+x_5^2+x_6^2}\,\partial_{x_s}\partial_{x_{s'}}\Delta_{\widetilde{T}}\T_{(3,2)}(x)\\
&=\partial_{x_0}^2\Delta_{\widetilde{T}}\T_{(3,2)}(x)+\sum_{s=4}^{6}\frac{x_s^2}{x_4^2+x_5^2+x_6^2}\,\partial_{x_s}^2\Delta_{\widetilde{T}}\T_{(3,2)}(x)\\
&=-24x_0-8x^1+\sum_{s=4}^{6}\frac{x_s^2}{x_4^2+x_5^2+x_6^2}(24x_0+8x^1)=-24x_0-8x^1+24x_0+8x^1\\
&\equiv0\,,
\end{align*}
valid for all $x\in\rr^7$ with $x_4e_4+x_5e_5+x_6e_6\neq0$.
\end{example}


\section{The Fueter-Sce phenomenon for $T$-functions}\label{sec:FueterSceTfunctions}

This section is devoted to the Fueter-Sce phenomenon for $T$-functions. Theorem~\ref{thm:powersoflaplacian} suggests the next definition.

\begin{definition}
Let $\tau\in\nn,T=(t_0,t_1,\ldots,t_\tau)\in\nn^{\tau+1}$ with $0\leq t_0<t_1<\ldots<t_\tau=N$ and set $\widetilde{T}:=(t_1,\ldots,t_\tau)$. Assume $t_1-t_0$ to be an odd natural number $2n_1+1$. For any strongly $T$-regular function $f=\I(F)\in\sr_T(\Omega_D,A)\subset\slice_T^\omega(\Omega_D,A)$, the function $\Delta_{\widetilde{T}}^{n_1}f$ is called the \emph{first Fueter transform} of $f$.
\end{definition}

The first Fueter transform of a strongly $T$-regular function, whenever defined, is a $T$-function by Theorem~\ref{thm:varietyoflaplacians} and is $\widetilde{T}$-harmonic by Theorem~\ref{thm:powersoflaplacian}. In the next theorem, we prove that it is actually $\widetilde{T}$-regular.

\begin{theorem}\label{thm:fuetersce}
Let $\tau\in\nn,T=(t_0,t_1,\ldots,t_\tau)\in\nn^{\tau+1}$ with $0\leq t_0<t_1<\ldots<t_\tau=N$ and set $\widetilde{T}:=(t_1,\ldots,t_\tau)$. Assume $t_1-t_0$ to be an odd natural number $2n_1+1$ and consider any strongly $T$-regular function $f\in\sr_T(\Omega_D,A)$. Then the first Fueter transform of $f$ is strongly $\widetilde{T}$-regular and still is a $T$-function. In symbols: $\Delta_{\widetilde{T}}^{n_1}f\in\sr_{\widetilde{T}}(\Omega_D,A)\cap\slice_T^\omega(\Omega_D,A)$.
\end{theorem}

\begin{proof}
For the sake of simplicity, let us write $n$ instead of $n_1$. We have $f\in\sr_T(\Omega_D,A)\subset\slice_T^\omega(\Omega_D,A)$ by Proposition~\ref{prop:analyticstronglyTregular}. If we apply $n$ times Theorem~\ref{thm:varietyoflaplacians} (in its special case $\sigma=1$), we conclude that $\Delta_{\widetilde{T}}^nf\in\slice_T^\omega(\Omega_D,A)$. Proving the leftover inclusion $\Delta_{\widetilde{T}}^nf\in\sr_{\widetilde{T}}(\Omega_D,A)$ is equivalent, by Theorem~\ref{thm:powersoflaplacian}, to proving the inclusion $f^{[n]}\in\sr_{\widetilde{T}}(\Omega_D,A)$; this is, in turn, equivalent to proving that $\debar_{\widetilde{T}}f^{[n]}\equiv0$ in $\Omega_D$.

We apply Theorem~\ref{thm:debartilde} to get $\debar_{\widetilde{T}}f^{[n]}=\debar_T f^{[n]}+(1+t_0-t_1)g=\debar_T f^{[n]}-2ng$, where $g=\I(G)$ with $G$ defined by formula~\eqref{eq:G} (with $F^{[n]}$ in lieu of $F$). Thus, $\debar_{\widetilde{T}}f^{[n]}$ is the $T$-function induced by the $T$-stem function $L\in\stem_T^\omega(D,A\otimes\rr^{2^\tau})$ whose $K$-component (for $K\in\mathscr{P}(\tau)$) can be expressed as
\begin{align*}
&L_K(\alpha,\beta)=\left\{
\begin{array}{ll}
(\debar_T F^{[n]})_K(\alpha,\beta)-2n\beta_1^{-1}F^{[n]}_{\{1\}\cup K}(\alpha,\beta)&\mathrm{if\ }1\not\in K\\
(\debar_T F^{[n]})_K(\alpha,\beta)&\mathrm{if\ }1\in K
\end{array}
\right.
\end{align*}
in $D':=\{(\alpha,\beta)\in D:\beta_1\neq0\}$. We have therefore translated our goal to prove that $\debar_{\widetilde{T}}f^{[n]}\equiv0$ in $\Omega_D$ into an equivalent goal: proving that $L_K\equiv0$ in $D'$ (whence in $D$) for all $K\in\mathscr{P}(\tau)$. By Definitions~\ref{def:debaralpha} and~\ref{def:stemoperators},
\[(\debar_TF^{[n]})_K=\partial_{\alpha_0}F^{[n]}_K+(-1)^{|K|}\sum_{s=1}^{t_0}v_s\,\partial_{\alpha_s}F^{[n]}_K+\sum_{h=1}^\tau(-1)^{\sigma(h,K)+1}\partial_{\beta_h}F^{[n]}_{K\bigtriangleup\{h\}}\,.\]
Analogous considerations about $(\debar_TF)_K$, which vanishes identically in $D$ by our hypothesis $f\in\sr_T(\Omega_D,A)$, yield the equality
\begin{align}\label{eq:technical}
\partial_{\alpha_0}F_K+(-1)^{|K|}\sum_{s=1}^{t_0}v_s\,\partial_{\alpha_s}F_K+\sum_{h=2}^\tau(-1)^{\sigma(h,K)+1}\partial_{\beta_h}F_{K\bigtriangleup\{h\}}=(-1)^{\sigma(1,K)}\partial_{\beta_1}F_{K\bigtriangleup\{1\}}\,,
\end{align}
which will soon be useful.

For $K\ni 1$, using formula~\eqref{eq:F[n]} and omitting the variable $(\alpha,\beta)$ for readability, we compute
\begin{align*}
L_K&=(\debar_T F^{[n]})_K=\sum_{\ell=0}^n\frac{b_{n,\ell}}{\beta_1^{2n-\ell}}\partial_{\beta_1}^\ell\left(\left(\partial_{\alpha_0}+(-1)^{|K|}\sum_{s=1}^{t_0}v_s\,\partial_{\alpha_s}\right)F_K\right)\\
&\quad+(-1)^{\sigma(1,K)+1}\partial_{\beta_1}\left(\sum_{\ell=0}^n\frac{b_{n-1,\ell-1}}{\beta_1^{2n-\ell}}(\partial_{\beta_1}^\ell F_{K\bigtriangleup\{1\}})\right)\\
&\quad+\sum_{\ell=0}^n\frac{b_{n,\ell}}{\beta_1^{2n-\ell}}\partial_{\beta_1}^\ell\left(\sum_{h=2}^\tau(-1)^{\sigma(h,K)+1}\partial_{\beta_h}F_{K\bigtriangleup\{h\}}\right)\\
&=\sum_{\ell=0}^nb_{n,\ell}\beta_1^{\ell-2n}\partial_{\beta_1}^\ell\left((-1)^{\sigma(1,K)}\partial_{\beta_1}F_{K\bigtriangleup\{1\}}\right)+\sum_{\ell=0}^nb_{n-1,\ell-1}\beta_1^{\ell-2n}(\partial_{\beta_1}^{\ell+1}F_{K\bigtriangleup\{1\}})\\
&\quad+\sum_{\ell=0}^nb_{n-1,\ell-1}(\ell-2n)\beta_1^{\ell-2n-1}(\partial_{\beta_1}^\ell F_{K\bigtriangleup\{1\}})\\
&=\sum_{\ell=0}^n(-b_{n,\ell}+b_{n-1,\ell-1})\beta_1^{\ell-2n}(\partial_{\beta_1}^{\ell+1}F_{K\bigtriangleup\{1\}})+\sum_{m=-1}^{n-1}b_{n-1,m}(m-2n+1)\beta_1^{m-2n}(\partial_{\beta_1}^{m+1} F_{K\bigtriangleup\{1\}})\\
&=\sum_{\ell=0}^{n-1}\left((\ell-2n+1)b_{n-1,\ell}+b_{n-1,\ell-1}-b_{n,\ell}\right)\beta_1^{\ell-2n}(\partial_{\beta_1}^{\ell+1}F_{K\bigtriangleup\{1\}})\equiv0
\end{align*}
in $D'$. For the third equality, we used~\eqref{eq:technical} and $\sigma(1,K)=1$. The fourth equality uses $\sigma(1,K)=1$ and $m:=\ell-1$. The fifth equality follows from $b_{n-1,n-1}-b_{n,n}=1-1=0$ and $b_{n-1,-1}(-1-2n+1)=-2nb_{n-1,-1}=0$. The sixth equality follows from the second property listed in Lemma~\ref{lem:besselcoefficients}.

For $K\not\ni 1$, using formula~\eqref{eq:F[n]} and omitting the variable $(\alpha,\beta)$ for readability, we compute
\begin{align*}
L_K&=(\debar_T F^{[n]})_K-2n\beta_1^{-1}F^{[n]}_{\{1\}\cup K}=\sum_{\ell=0}^n\frac{b_{n-1,\ell-1}}{\beta_1^{2n-\ell}}\partial_{\beta_1}^\ell\left(\left(\partial_{\alpha_0}+(-1)^{|K|}\sum_{s=1}^{t_0}v_s\,\partial_{\alpha_s}\right)F_K\right)\\
&\quad+(-1)^{\sigma(1,K)+1}\partial_{\beta_1}\left(\sum_{\ell=0}^n\frac{b_{n,\ell}}{\beta_1^{2n-\ell}}(\partial_{\beta_1}^\ell F_{K\bigtriangleup\{1\}})\right)\\
&\quad+\sum_{\ell=0}^n\frac{b_{n-1,\ell-1}}{\beta_1^{2n-\ell}}\partial_{\beta_1}^\ell\left(\sum_{h=2}^\tau(-1)^{\sigma(h,K)+1}\partial_{\beta_h}F_{K\bigtriangleup\{h\}}\right)-2n\beta_1^{-1}\sum_{\ell=0}^n\frac{b_{n,\ell}}{\beta_1^{2n-\ell}}(\partial_{\beta_1}^\ell F_{\{1\}\cup K})\\
&=\sum_{\ell=0}^nb_{n-1,\ell-1}\beta_1^{\ell-2n}\partial_{\beta_1}^\ell\left((-1)^{\sigma(1,K)}\partial_{\beta_1}F_{K\bigtriangleup\{1\}}\right)-\sum_{\ell=0}^nb_{n,\ell}\beta_1^{\ell-2n}(\partial_{\beta_1}^{\ell+1}F_{K\bigtriangleup\{1\}})\\
&\quad+\sum_{\ell=0}^nb_{n,\ell}(2n-\ell)\beta_1^{\ell-2n-1}(\partial_{\beta_1}^\ell F_{K\bigtriangleup\{1\}})-2n\sum_{\ell=0}^nb_{n,\ell}{\beta_1^{\ell-2n-1}}(\partial_{\beta_1}^\ell F_{K\bigtriangleup\{1\}})\\
&=\sum_{\ell=0}^n(b_{n-1,\ell-1}-b_{n,\ell})\beta_1^{\ell-2n}(\partial_{\beta_1}^{\ell+1}F_{K\bigtriangleup\{1\}})-\sum_{\ell=1}^n\ell b_{n,\ell}\beta_1^{\ell-2n-1}(\partial_{\beta_1}^\ell F_{K\bigtriangleup\{1\}})\\
&=\sum_{\ell=0}^{n-1}(b_{n-1,\ell-1}-b_{n,\ell})\beta_1^{\ell-2n}(\partial_{\beta_1}^{\ell+1}F_{K\bigtriangleup\{1\}})-\sum_{m=0}^{n-1}(m+1)b_{n,m+1}\beta_1^{m-2n}(\partial_{\beta_1}^{m+1} F_{K\bigtriangleup\{1\}})\\
&=\sum_{\ell=0}^{n-1}(b_{n-1,\ell-1}-b_{n,\ell}-(\ell+1)b_{n,\ell+1})\beta_1^{\ell-2n}(\partial_{\beta_1}^{\ell+1}F_{K\bigtriangleup\{1\}})\equiv0
\end{align*}
in $D'$. For the third equality, we used~\eqref{eq:technical} and $\sigma(1,K)=0$. For the fourth equality, we used again $\sigma(1,K)=0$ and the equality $2n-\ell-2n=-\ell$ (which vanishes for $\ell=0$). For the fifth equality, we used $b_{n-1,n-1}-b_{n,n}=1-1=0$ and we set $m:=\ell-1$. The seventh equality follows from the third property listed in Lemma~\ref{lem:besselcoefficients}.

Since we proved that $L_K\equiv0$ in $D'$ for all $K\in\mathscr{P}(\tau)$, the proof is complete.
\end{proof}

Theorem~\ref{thm:fuetersce} recovers Fueter's theorem of~\cite{fueter1} as its special case with $A=\hh=V,T=(0,3)$. It recovers Sce's theorem of~\cite{sce} as its special case with $A=C\ell(0,N),V=\rr^{N+1},T=(0,N)$ for some odd $N$. It recovers Xu and Sabadini's result of~\cite{xsfuetersce} as its special case with $A=C\ell(0,N),V=\rr^{N+1},T=(t_0,N)$ with odd $N-t_0$.

For two of our polynomial examples, let us compute the first Fueter transform of each polynomial and explicitly check its $\widetilde{T}$-regularity. The first example is not covered by the results known in literature because the list of steps $T$ has length $2$.

\begin{example}\label{ex:(0,3,6)quinquies}
Let $A=C\ell(0,6),V=\rr^7$ and $T=(0,3,6)$, whence $\widetilde{T}=(3,6)$. Consider the strongly $T$-regular polynomial $3\T_{(2,1)}(x)=x_0^3+3x_0^2x^2-3x_0\Vert x^1\Vert^2+6x_0x^1x^2-3\Vert x^1\Vert^2x^2$, computed in Example~\ref{ex:(0,3,6)}. Its first Fueter transform $3\,\Delta_{\widetilde{T}}\T_{(2,1)}(x)=-12x_0-12x^2$, computed in Example~\ref{ex:(0,3,6)ter}, is strongly $\widetilde{T}$-regular by Theorem~\ref{thm:fuetersce} or by the direct computation
\[3\,\debar_{\widetilde{T}}\Delta_{\widetilde{T}}\T_{(2,1)}(x)=-12+\frac{x_4e_4+x_5e_5+x_6e_6}{x_4^2+x_5^2+x_6^2}\sum_{s=4}^6x_s(-12e_s)\equiv-12+12=0\,,\]
performed according to Proposition~\ref{prop:globaloperators} in the open dense subset of $\rr^7$ where $x_4e_4+x_5e_5+x_6e_6\neq0$.

In Example~\ref{ex:(0,3,6)ter2}, we computed $3\,\Delta_{(6)}\T_{(2,1)}(x)=-12x_0-12x^2$: it is not $(6)$-regular because $3\,\debar_{(6)}\Delta_{(6)}\T_{(2,1)}\equiv-12+\sum_{s=4}^6e_s(-12e_s)=-12+36=24$ for all $x\in\rr^7$. In other words, the full Laplacian of $3\T_{(2,1)}$ is not a monogenic function $\rr^7\to C\ell(0,6)$.
\end{example}

Our second example is also not covered by previously known results, both because the list of steps $T$ has length $2$ and because the domain properly includes the paravector space. 

\begin{example}\label{ex:(1,4,7)bis}
Let $A=C\ell(0,6), V=\Span(e_\emptyset,e_1,e_2,e_3,e_4,e_5,e_6,e_{123456})$ and $T=(1,4,7)$, whence $\widetilde{T}=(4,7)$. In Example~\ref{ex:(1,4,7)}, we obtained the expression
\begin{align*}
6\T_{(1,2,1)}(x)&=\T_{(0,1,1)}(x)(-2x_0x_1-2x_0e_1x^1+2x_1x^1)+\T_{(0,2,0)}(x)(x_0x_1-x_0e_1x^2+x_1x^2)\\
&\quad+\T_{(1,0,1)}(x)(x_0^2-2x_0x^1-\Vert x^1\Vert^2)+\T_{(1,1,0)}(x)(2x_0x^1+2x_0x^2+2x^1x^2)\,,
\end{align*}
where
\begin{align*}
&\T_{(0,1,1)}(x)=x_0x^1-x_0x^2-x^1x^2\,,\qquad\T_{(0,2,0)}(x)=x_0^2+2x_0x^1-\Vert x^1\Vert^2\\
&\T_{(1,0,1)}(x)=x_0x_1-x_0e_1x^2+x_1x^2\,,\qquad\T_{(1,1,0)}(x)=x_0x_1-x_0e_1x^1+x_1x^1\,.
\end{align*}
Let us compute the first Fueter transform $\Delta_{\widetilde{T}}\T_{(1,2,1)}$ of $\T_{(1,2,1)}$. We start with
\begin{align*}
6\,\partial_{x_0}^2\T_{(1,2,1)}(x)&=0+0+2(x^1-x^2)(-2x_1-2e_1x^1)\\
&\quad+2(x_0x_1-x_0e_1x^2+x_1x^2)+0+2(2x_0+2x^1)(x_1-e_1x^2)\\
&\quad+0+(x_0x_1-x_0e_1x^2+x_1x^2)2+2(x_1-e_1x^2)(2x_0-2x^1)\\
&\quad+0+0+2(x_1-e_1x^1)(2x^1+2x^2)\\
&=4(-x_1x^1+x_1x^2-e_1\Vert x^1\Vert^2+e_1x^1x^2)+4(x_0x_1-x_0e_1x^2+x_1x^2)\\
&\quad+4(x_0x_1-x_0e_1x^2+x_1x^1+e_1x^1x^2)+4(x_0x_1-x_0e_1x^2-x_1x^1-e_1x^1x^2)\\
&\quad+4(x_1x^1+x_1x^2+e_1\Vert x^1\Vert^2-e_1x^1x^2)\\
&=12(x_0x_1-x_0e_1x^2+x_1x^2)\,,\\
6\,\partial_{x_1}^2\T_{(1,2,1)}(x)&=0+0+0+0+0+0+0+0+0+0+0+0=0
\end{align*}
and, for $s=2,3,4$,
\begin{align*}
6\,\partial_{x_s}^2\T_{(1,2,1)}(x)&=0+0+2(x_0e_s-e_sx^2)(-2x_0e_1e_s+2x_1e_s)\\
&\quad-2(x_0x_1-x_0e_1x^2+x_1x^2)+0+0\\
&\quad+0+(x_0x_1-x_0e_1x^2+x_1x^2)(-2)+0\\
&\quad+0+0+2(-x_0e_1e_s+x_1e_s)(2x_0e_s+2e_sx^2)\\
&=4(-x_0^2e_1-x_0x_1+x_0e_1x^2-x_1x^2)+4(-x_0x_1+x_0e_1x^2-x_1x^2)\\
&\quad+4(x_0^2e_1-x_0x_1+x_0e_1x^2-x_1x^2)\\
&=12(-x_0x_1+x_0e_1x^2-x_1x^2)\,.
\end{align*}
We also remark that $6\T_{(1,2,1)}(x)$ has degree $1$ in $x_5,x_6,x_{123456}$, whence $6\,\partial_{x_s}\partial_{x_{s'}}\T_{(1,2,1)}(x)\equiv0$ for $s,s'\in\{5,6,123456\}$. Thus,
\begin{align*}
\Delta_{\widetilde{T}}\T_{(1,2,1)}(x)&=\sum_{s=0}^4\,\partial_{x_s}^2\T_{(1,2,1)}(x)+0\\
&=2(x_0x_1-x_0e_1x^2+x_1x^2)+0+6(-x_0x_1+x_0e_1x^2-x_1x^2)\\
&=-4x_0x_1+4x_0e_1x^2-4x_1x^2
\end{align*}
for all $x\in V$ with $x_5e_5+x_6e_6+x_{123456}e_{123456}\neq0$. It follows at once that
\[\Delta_{\widetilde{T}}\T_{(1,2,1)}(x)=-4x_0x_1+4x_0e_1x^2-4x_1x^2\]
throughout $V$. Analogous computations prove that $\Delta_{(7)}\T_{(2,1)}(x)=-4x_0x_1+4x_0e_1x^2-4x_1x^2$ for all $x\in V$.

The first Fueter transform $\Delta_{\widetilde{T}}\T_{(1,2,1)}$ of $T_{(1,2,1)}$ is strongly $\widetilde{T}$-regular by Theorem~\ref{thm:fuetersce} or by the direct computation
\begin{align*}
\debar_{\widetilde{T}}\Delta_{\widetilde{T}}\T_{(1,2,1)}(x)&=-4x_1+4e_1x^2+e_1(-4x_0-4x^2)+0\\
&\quad+\frac{x_5e_5+x_6e_6+x_{123456}e_{123456}}{x_5^2+x_6^2+x_{123456}^2}\sum_{s\in\{5,6,123456\}}x_s(4x_0e_1e_s-4x_1e_s)\\
&=-4x_1-4x_0e_1+\frac{(x_5e_5+x_6e_6+x_{123456}e_{123456})^2}{x_5^2+x_6^2+x_{123456}^2}(-4x_0e_1-4x_1)\\
&=-4x_1-4x_0e_1+4x_0e_1+4x_1\equiv0\,,
\end{align*}
valid in the open dense subset of $V$ where $x_5e_5+x_6e_6+x_{123456}e_{123456}\neq0$. On the other hand, the polynomial function $\Delta_{(7)}\T_{(1,2,1)}(x)=-4x_0x_1+4x_0e_1x^2-4x_1x^2$ is not $(7)$-regular because
\begin{align*}
\debar_{(7)}\Delta_{(7)}\T_{(1,2,1)}&\equiv-4x_1+4e_1x^2+e_1(-4x_0-4x^2)+0+\sum_{s\in\{5,6,123456\}}e_s(4x_0e_1e_s-4x_1e_s)\\
&=-4x_1-4x_0e_1+3(4x_0e_1+4x_1)=8x_0e_1+8x_1
\end{align*}
for all $x\in V$.
\end{example}

Now assume $t_h-t_{h-1}$ to be an odd natural number not only for $h=1$, but for $\sigma$ distinct choices of $h$ in $\{1,\ldots,\tau\}$. There is no loss of generality in assuming $t_h-t_{h-1}$ to be odd exactly for the first $\sigma$ choices of $h\in\{1,\ldots,\tau\}$ because the elements $v_1,\ldots,v_N$ of the basis $\B$ in Assumption~\ref{ass:associative} can be reordered. In this situation, we articulate the Fueter-Sce phenomenon into $\sigma$ steps. As a preparation, we pose the next definition.

\begin{definition}
Let $\sigma,\tau\in\nn$ with $\sigma\leq\tau$, let $T_0=(t_0,t_1,\ldots,t_\tau)\in\nn^{\tau+1}$ and assume $t_h-t_{h-1}$ to be an odd natural number $2n_h+1$ for every $h\in\{1,\ldots,\sigma\}$. Set $T_1:=(t_1,\ldots,t_\tau),\ldots,T_\sigma:=(t_\sigma,\ldots,t_\tau)$. For every strongly $T_0$-regular function $f\in\sr_{T_0}(\Omega_D,A)$, the function
\[\Delta_{T_\sigma}^{n_\sigma}\ldots\Delta_{T_2}^{n_2}\Delta_{T_1}^{n_1}f\in\slice_{T_\sigma}^\omega(\Omega_D,A)\]
is called the \emph{$\sigma$-th Fueter transform} of $f$.
 \end{definition}

We are now in a position to state the next theorem, which is the main result in the present work.
 
\begin{theorem}\label{cor:fuetersce}
Let $\sigma,\tau\in\nn$ with $\sigma\leq\tau$ and let $T_0=(t_0,t_1,\ldots,t_\tau)\in\nn^{\tau+1}$. Assume $t_h-t_{h-1}$ to be an odd natural number for every $h\in\{1,\ldots,\sigma\}$ and set $T_1:=(t_1,\ldots,t_\tau),\ldots,T_\sigma:=(t_\sigma,\ldots,t_\tau)$. For every strongly $T_0$-regular function $f$, the $\sigma$-th Fueter transform of $f$ is a strongly $T_\sigma$-regular function and still is a $T_0$-function. In symbols, $\Delta_{T_\sigma}^{n_\sigma}\ldots\Delta_{T_2}^{n_2}\Delta_{T_1}^{n_1}f\in\sr_{T_\sigma}(\Omega_D,A)\cap\slice_{T_0}^\omega(\Omega_D,A)$.
\end{theorem}

\begin{proof}
We proceed by induction on $\sigma$. The basic case $\sigma=1$ has been established in Theorem~\ref{thm:fuetersce}. Now assume the thesis true for $\sigma=h$ and let us prove it for $\sigma=h+1$. Our inductive hypothesis is that the $h$-th Fueter transform $g:=\Delta_{T_h}^{n_h}\ldots\Delta_{T_2}^{n_2}\Delta_{T_1}^{n_1}f$ of $f$ belongs to both $\sr_{T_h}(\Omega_D,A)$ and $\slice_T^\omega(\Omega_D,A)$. Now, the $(h+1)$-th Fueter transform of $f$ is the first Fueter transform $\Delta_{T_{h+1}}^{n_{h+1}}g$ of $g$. Another application of Theorem~\ref{thm:fuetersce} guarantees that $\Delta_{T_{h+1}}^{n_{h+1}}g\in\sr_{T_{h+1}}(\Omega_D,A)$, as desired. Furthermore, by applying Theorem~\ref{thm:varietyoflaplacians} $n_{h+1}$ times, we obtain that $\Delta_{T_{h+1}}^{n_{h+1}}g$ still belongs to $\slice_T^\omega(\Omega_D,A)$. The proof is now complete.
\end{proof}

Since the $\sigma$-th Fueter transform is still a $T_0$-function, where $T_0$ is a list of $\tau$ steps, we can describe it as a $\tau$-axial function in the sense of Eelbode's work~\cite{eelbode2014}. In the very special case when $t_h-t_{h-1}=1$ for every $h\in\{1,\ldots,\sigma\}$, then the $\sigma$-th Fueter transform of $f$ coincides with $f$ and Theorem~\ref{cor:fuetersce} recovers Proposition~\ref{prop:shortsteps}. Furthermore, Theorem~\ref{cor:fuetersce} has the following immediate consequence.

\begin{corollary}\label{cor:monogenic}
Let $\tau\in\nn$ and let $T=(t_0,t_1,\ldots,t_\tau)\in\nn^{\tau+1}$ with $0\leq t_0<t_1<\ldots<t_\tau=N$. If $t_h-t_{h-1}$ is an odd natural number for every $h\in\{1,\ldots,\tau\}$, then the $\tau$-th Fueter transform of $f$ is a strongly $(N)$-regular function that is also a $T$-function. In particular: for every strongly $T$-regular function $\rr^{N+1}\to C\ell(0,N)$, the $\tau$-th Fueter transform of $f$ (if defined) is a monogenic function $\rr^{N+1}\to C\ell(0,N)$ that is also a $T$-function.
\end{corollary}

Corollary~\ref{cor:monogenic} recovers Fueter's theorem of~\cite{fueter1} as its special case with $A=\hh=V,\tau=1,t_0=0$. It recovers Sce's theorem of~\cite{sce} as its special case with $A=C\ell(0,N),V=\rr^{N+1},\tau=1,t_0=0$. It recovers Xu and Sabadini's result of~\cite{xsfuetersce} as its special case with $A=C\ell(0,N),V=\rr^{N+1},\tau=1$. For $A=C\ell(0,N),V=\rr^{N+1},\tau\geq2$, Corollary~\ref{cor:monogenic} produces $\tau$-axial monogenic functions.

We now provide an explicit example with $\tau=2$, by computing the second Fueter transform of our polynomial example of degree $5$.

\begin{example}\label{ex:(0,3,6)sexties}
Let $A=C\ell(0,6),V=\rr^7$ and $T=(0,3,6)$, whence $T_1=\widetilde{T}=(3,6),T_2=\widetilde{T_1}=(6)$. In Example~\ref{ex:(0,3,6)quater}, we computed the first Fueter transform of $\T_{(3,2)}$: namely,
\[\Delta_{T_1}\T_{(3,2)}(x)=-4x_0^3-4x_0^2x^1+8x_0x^1x^2+12x_0\Vert x^2\Vert^2+4x^1\Vert x^2\Vert^2\,.\]
We can easily compute the second Fueter transform of $\T_{(3,2)}$, as follows:
\begin{align*}
\Delta_{T_2}\Delta_{T_1}\T_{(3,2)}(x)&=\sum_{s=0}^6\partial_{x_s}^2\Delta_{T_1}\T_{(3,2)}(x)=-24x_0-8x^1+\sum_{s=1}^30+\sum_{s=4}^6(24x_0+8x^1)\\
&=48x_0+16x^1\,.
\end{align*}
The second Fueter transform of $\T_{(3,2)}$ is $T_2$-regular, i.e., monogenic, by Corollary~\ref{cor:monogenic} or by the direct computation
\[\debar_{T_2}\Delta_{T_2}\Delta_{T_1}\T_{(3,2)}\equiv48+\sum_{s=1}^3e_s16e_s+\sum_{s=4}^6e_s0=48-48+0=0\,.\]
\end{example}

\section{Concluding remarks}\label{sec:conclusions}

This work unifies into a single statement, Theorem~\ref{thm:fuetersce}, several results known in literature: Fueter's theorem of~\cite{fueter1}, Sce's theorem of~\cite{sce}, as well as the more recent results of Xu and Sabadini~\cite{xsfuetersce}. Theorem~\ref{thm:fuetersce} applies not only to Clifford algebras, but to general associative $*$-algebras. Even over Clifford algebras, it applies not only to the paravector subspace but to general hypercomplex subspaces. Moreover, Theorem~\ref{thm:fuetersce} is multi-axial in the sense of~\cite{eelbode2014}.

More importantly, this work uncovers a new phenomenon: our main Theorem~\ref{cor:fuetersce} and its Corollary~\ref{cor:monogenic} show that the Fueter-Sce theorem is just the last step in a longer process involving several steps.

Beyond our main results, we believe this work opens the path for mixed phenomena of Sce~\cite{sce} and Qian~\cite{qian1997} type. Indeed, we conjecture that the oddness hypotheses in Theorem~\ref{cor:fuetersce} and in Corollary~\ref{cor:monogenic} might be removed. In other words, we conjecture that for arbitrary $T$, $\sigma\in\{1,\ldots,\tau\}$ and $f\in\sr_T(\Omega_D,A)$, an appropriately defined $\sigma$-th Fueter transform of $f$ is strongly $T_\sigma$-regular and still is a $T$-function. This means monogenic and $\tau$-axial (in the sense of~\cite{eelbode2014}) when $A=C\ell(0,N),V=\rr^{N+1},\sigma=\tau$.



\section*{Acknowledgements}

Both authors are partly supported by: GNSAGA INdAM; Progetto “Teoria delle funzioni ipercomplesse e applicazioni” Università di Firenze. The second author was also partly supported by: PRIN 2022 “Real and complex manifolds: geometry and holomorphic dynamics” (2022AP8HZ9) MIUR; Finanziamento Premiale “Splines for accUrate NumeRics: adaptIve models for Simulation Environments” INdAM.

The authors warmly thank the anonymous referee, whose careful reading helped them polishing this work.







\begin{thebibliography}{10}

\bibitem{librosommen}
F.~Brackx, R.~Delanghe, and F.~Sommen.
\newblock {\em Clifford analysis}, volume~76 of {\em Research Notes in
  Mathematics}.
\newblock Pitman (Advanced Publishing Program), Boston, MA, 1982.

\bibitem{librocnops}
J.~Cnops and H.~Malonek.
\newblock {\em An introduction to {C}lifford analysis}.
\newblock Textos de Matem{\'a}tica. S{\'e}rie B [Texts in Mathematics. Series
  B], 7. Universidade de Coimbra Departamento de Matem{\'a}tica, Coimbra, 1995.

\bibitem{advancesrevised}
F.~Colombo, G.~Gentili, I.~Sabadini, and D.~Struppa.
\newblock Extension results for slice regular functions of a quaternionic
  variable.
\newblock {\em Adv. Math.}, 222(5):1793--1808, 2009.

\bibitem{integralfueter}
F.~Colombo, I.~Sabadini, and F.~Sommen.
\newblock The {F}ueter mapping theorem in integral form and the {F}-functional
  calculus.
\newblock {\em Math. Methods Appl. Sci.}, 33(17):2050--2066, 2010.

\bibitem{inversefueter}
F.~Colombo, I.~Sabadini, and F.~Sommen.
\newblock The inverse {F}ueter mapping theorem.
\newblock {\em Commun. Pure Appl. Anal.}, 10(4):1165--1181, 2011.

\bibitem{colombosabadinisommen2013}
F.~Colombo, I.~Sabadini, and F.~Sommen.
\newblock The inverse {F}ueter mapping theorem in integral form using spherical
  monogenics.
\newblock {\em Israel J. Math.}, 194(1):485--505, 2013.

\bibitem{israel}
F.~Colombo, I.~Sabadini, and D.~C. Struppa.
\newblock Slice monogenic functions.
\newblock {\em Israel J. Math.}, 171:385--403, 2009.

\bibitem{librodaniele2}
F.~Colombo, I.~Sabadini, and D.~C. Struppa.
\newblock {\em Noncommutative functional calculus. Theory and applications of
  slice hyperholomorphic functions}, volume 289 of {\em Progress in
  Mathematics}.
\newblock Birkh{\"a}user/Springer Basel AG, Basel, 2011.

\bibitem{librotraduzionesce}
F.~Colombo, I.~Sabadini, and D.~C. Struppa.
\newblock {\em Michele {S}ce's works in hypercomplex analysis---a translation
  with commentaries}.
\newblock Birkh\"{a}user/Springer, Cham, [2020] \copyright 2020.

\bibitem{cullen}
C.~G. Cullen.
\newblock An integral theorem for analytic intrinsic functions on quaternions.
\newblock {\em Duke Math. J.}, 32:139--148, 1965.

\bibitem{dentonisce}
P.~Dentoni and M.~Sce.
\newblock Funzioni regolari nell'algebra di {C}ayley.
\newblock {\em Rend. Sem. Mat. Univ. Padova}, 50:251--267 (1974), 1973.

\bibitem{dongkouqian}
B.~Dong, K.~I. Kou, T.~Qian, and I.~Sabadini.
\newblock On the inversion of {F}ueter's theorem.
\newblock {\em J. Geom. Phys.}, 108:102--116, 2016.

\bibitem{dongqian}
B.~Dong and T.~Qian.
\newblock Uniform generalizations of {F}ueter's theorem.
\newblock {\em Ann. Mat. Pura Appl. (4)}, 200(1):229--251, 2021.

\bibitem{ebbinghaus}
H.-D. Ebbinghaus, H.~Hermes, F.~Hirzebruch, M.~Koecher, K.~Mainzer,
  J.~Neukirch, A.~Prestel, and R.~Remmert.
\newblock {\em Numbers}, volume 123 of {\em Graduate Texts in Mathematics}.
\newblock Springer-Verlag, New York, 1990.
\newblock With an introduction by K. Lamotke, Translated from the second German
  edition by H. L. S. Orde, Translation edited and with a preface by J. H.
  Ewing, Readings in Mathematics.

\bibitem{eelbode2014}
D.~Eelbode.
\newblock The biaxial {F}ueter theorem.
\newblock {\em Israel J. Math.}, 201(1):233--245, 2014.

\bibitem{eelbodesouckvanlancker2012}
D.~Eelbode, V.~Sou\v{c}ek, and P.~Van~Lancker.
\newblock The fueter theorem by representation theory.
\newblock In T.~E. Simos, G.~Psihoyios, C.~Tsitouras, and Z.~Anastassi,
  editors, {\em Numerical analysis and applied mathematics {ICNAAM} 2012},
  volume 1479 of {\em AIP Conference Proceedings}, pages 340--343. American
  Institute of Physics, College Park, MD, 2012.

\bibitem{eelbodesouckvanlancker2014}
D.~Eelbode, V.~Sou\v{c}ek, and P.~Van~Lancker.
\newblock Gegenbauer polynomials and the {F}ueter theorem.
\newblock {\em Complex Var. Elliptic Equ.}, 59(6):826--840, 2014.

\bibitem{feicerejeiraskaehler}
M.~Fei, P.~Cerejeiras, and U.~K\"{a}hler.
\newblock Fueter's theorem and its generalizations in {D}unkl-{C}lifford
  analysis.
\newblock {\em J. Phys. A}, 42(39):395209, 15, 2009.

\bibitem{fueter1}
R.~Fueter.
\newblock Die {F}unktionentheorie der {D}ifferentialgleichungen {$\Delta u=0$}
  und {$\Delta\Delta u=0$} mit vier reellen {V}ariablen.
\newblock {\em Comment. Math. Helv.}, 7(1):307--330, 1934.

\bibitem{fueter2}
R.~Fueter.
\newblock \"{U}ber die analytische {D}arstellung der regul\"aren {F}unktionen
  einer {Q}uaternionenvariablen.
\newblock {\em Comment. Math. Helv.}, 8(1):371--378, 1935.

\bibitem{librospringer2}
G.~Gentili, C.~Stoppato, and D.~C. Struppa.
\newblock {\em Regular functions of a quaternionic variable}.
\newblock Springer Monographs in Mathematics. Springer, Cham, 2022.
\newblock Second edition.

\bibitem{cras}
G.~Gentili and D.~C. Struppa.
\newblock A new approach to {C}ullen-regular functions of a quaternionic
  variable.
\newblock {\em C. R. Math. Acad. Sci. Paris}, 342(10):741--744, 2006.

\bibitem{advances}
G.~Gentili and D.~C. Struppa.
\newblock A new theory of regular functions of a quaternionic variable.
\newblock {\em Adv. Math.}, 216(1):279--301, 2007.

\bibitem{rocky}
G.~Gentili and D.~C. Struppa.
\newblock Regular functions on the space of {C}ayley numbers.
\newblock {\em Rocky Mountain J. Math.}, 40(1):225--241, 2010.

\bibitem{ghilonislicebyslice}
R.~Ghiloni.
\newblock Slice-by-slice and global smoothness of slice regular and
  polyanalytic functions.
\newblock {\em Ann. Mat. Pura Appl. (4)}, 201(5):2549--2573, 2022.

\bibitem{perotti}
R.~Ghiloni and A.~Perotti.
\newblock Slice regular functions on real alternative algebras.
\newblock {\em Adv. Math.}, 226(2):1662--1691, 2011.

\bibitem{volumeintegral}
R.~Ghiloni and A.~Perotti.
\newblock Volume {C}auchy formulas for slice functions on real associative
  *-algebras.
\newblock {\em Complex Var. Elliptic Equ.}, 58(12):1701--1714, 2013.

\bibitem{global}
R.~Ghiloni and A.~Perotti.
\newblock Global differential equations for slice regular functions.
\newblock {\em Math. Nachr.}, 287(5-6):561--573, 2014.

\bibitem{gpseveral}
R.~Ghiloni and A.~Perotti.
\newblock Slice regular functions in several variables.
\newblock {\em Math. Z.}, 302(1):295--351, 2022.

\bibitem{unifiednotion}
R.~Ghiloni and C.~Stoppato.
\newblock A unified notion of regularity in one hypercomplex variable.
\newblock {\em J. Geom. Phys.}, 202:Paper No. 105219, 2024.

\bibitem{unifiedtheory}
R.~Ghiloni and C.~Stoppato.
\newblock A unified theory of regular functions of a hypercomplex variable.
\newblock Preprint, arXiv:2408.01523 [math.CV], 2024.

\bibitem{librogrosswald}
E.~Grosswald.
\newblock {\em Bessel polynomials}, volume 698 of {\em Lecture Notes in
  Mathematics}.
\newblock Springer, Berlin, 1978.

\bibitem{librogurlebeck2}
K.~G{{\"u}}rlebeck, K.~Habetha, and W.~Spr{{\"o}}{\ss}ig.
\newblock {\em Holomorphic functions in the plane and {$n$}-dimensional space}.
\newblock Birkh{\"a}user Verlag, Basel, 2008.
\newblock Translated from the 2006 German original, With 1 CD-ROM (Windows and
  UNIX).

\bibitem{kouqiansommen}
K.~I. Kou, T.~Qian, and F.~Sommen.
\newblock Generalizations of {F}ueter's theorem.
\newblock {\em Methods Appl. Anal.}, 9(2):273--289, 2002.

\bibitem{penapenaqiansommen}
D.~Pe\~{n}a Pe\~{n}a, T.~Qian, and F.~Sommen.
\newblock An alternative proof of {F}ueter's theorem.
\newblock {\em Complex Var. Elliptic Equ.}, 51(8-11):913--922, 2006.

\bibitem{penapenasommen2006}
D.~Pe\~{n}a Pe\~{n}a and F.~Sommen.
\newblock A generalization of {F}ueter's theorem.
\newblock {\em Results Math.}, 49(3-4):301--311, 2006.

\bibitem{penapenasommen2010}
D.~Pe\~{n}a Pe\~{n}a and F.~Sommen.
\newblock Fueter's theorem: the saga continues.
\newblock {\em J. Math. Anal. Appl.}, 365(1):29--35, 2010.

\bibitem{penapenasommen2010bis}
D.~Pe\~{n}a Pe\~{n}a and F.~Sommen.
\newblock A note on the {F}ueter theorem.
\newblock {\em Adv. Appl. Clifford Algebr.}, 20(2):379--391, 2010.

\bibitem{perotticr}
A.~Perotti.
\newblock Cauchy-{R}iemann operators and local slice analysis over real
  alternative algebras.
\newblock {\em J. Math. Anal. Appl.}, 516(1):Paper No. 126480, 34, 2022.

\bibitem{qian1997}
T.~Qian.
\newblock Generalization of {F}ueter's result to {${\bf R}^{n+1}$}.
\newblock {\em Atti Accad. Naz. Lincei Cl. Sci. Fis. Mat. Natur. Rend. Lincei
  (9) Mat. Appl.}, 8(2):111--117, 1997.

\bibitem{qian2015}
T.~Qian.
\newblock Fueter mapping theorem in hypercomplex analysis.
\newblock In D.~Alpay, editor, {\em Operator theory. {V}ols. 1, 2}, pages
  xix+1860. Springer, Basel, 2015.

\bibitem{sce}
M.~Sce.
\newblock Osservazioni sulle serie di potenze nei moduli quadratici.
\newblock {\em Atti Accad. Naz. Lincei Rend. Cl. Sci. Fis. Mat. Nat. (8)},
  23:220--225, 1957.

\bibitem{sommen1988}
F.~Sommen.
\newblock Special functions in {C}lifford analysis and axial symmetry.
\newblock {\em J. Math. Anal. Appl.}, 130(1):110--133, 1988.

\bibitem{sommen2000}
F.~Sommen.
\newblock On a generalization of {F}ueter's theorem.
\newblock {\em Z. Anal. Anwendungen}, 19(4):899--902, 2000.

\bibitem{sudbery}
A.~Sudbery.
\newblock Quaternionic analysis.
\newblock {\em Math. Proc. Cambridge Philos. Soc.}, 85(2):199--224, 1979.

\bibitem{whitney}
H.~Whitney.
\newblock Differentiable even functions.
\newblock {\em Duke Math. J.}, 10:159--160, 1943.

\bibitem{xsannouncement}
Z.~Xu and I.~Sabadini.
\newblock Generalized partial-slice monogenic functions: a synthesis of two
  function theories.
\newblock {\em Adv. Appl. Clifford Algebr.}, 34(2):Paper No. 10, 14, 2024.

\bibitem{xsgeneralizedpartialslice}
Z.~Xu and I.~Sabadini.
\newblock Generalized partial-slice monogenic functions.
\newblock {\em Trans. Amer. Math. Soc.}, 378(2):851--883, 2025.

\bibitem{xsfuetersce}
Z.~Xu and I.~Sabadini.
\newblock On the {F}ueter-{S}ce theorem for generalized partial-slice monogenic
  functions.
\newblock {\em Ann. Mat. Pura Appl. (4)}, 204(2):835--857, 2025.

\end{thebibliography}
\end{document}